\documentclass[11pt]{amsart}
\usepackage[utf8]{inputenc}
\usepackage[margin=.9in]{geometry}
\usepackage{geometry, url}
\usepackage{datetime}
\usepackage{amsmath,amsthm,amssymb,amsfonts}
\usepackage{graphicx, caption, subcaption, listings, float}
\usepackage{multicol, multirow,enumitem}
\usepackage{array,algorithm, algpseudocode}

\usepackage{setspace}
\usepackage{cite}

\usepackage{xcolor}

\usepackage[foot]{amsaddr}

\theoremstyle{plain}
\newtheorem{thm}{Theorem}[section]
\newtheorem{cor}[thm]{Corollary}
\newtheorem{lem}[thm]{Lemma}

\newtheorem{claim}[thm]{Claim}
\newtheorem{fact}[thm]{Fact}

\newtheorem{ass}[thm]{Assumption}
\newtheorem{Rest}[thm]{Restriction}

\newtheorem{defn}[thm]{Definition}
\newtheorem{rem}[thm]{Remark}
\newtheorem{rems}[thm]{Remarks}
\newtheorem{exm}[thm]{Example}

\newcommand{\x}{x}
\newcommand{\y}{y}

\newcommand{\R}{{\mathbb R}}

\newcommand{\Z}{{\mathbb Z}}
\newcommand{\Rn}{{\R^n}}

\newcommand{\LL}{{L}}
\newcommand{\Leb}{{\mathcal L}}
\newcommand{\cont}{{\mathcal C}}
\newcommand{\norm}[1]{\left\lVert#1\right\rVert}

\renewcommand{\epsilon}{\varepsilon}

\begin{document}
	
\title{Solving Semi-Discrete Optimal Transport Problems:\\
 star shapedeness and Newton's method}
\author[Dieci]{Luca Dieci}
\address{School of Mathematics, Georgia Institute of Technology,
	Atlanta, GA 30332 U.S.A.}
\email{dieci@math.gatech.edu}
\author[Omarov]{Daniyar Omarov}
\address{School of Mathematics, Georgia Institute of Technology,
	Atlanta, GA 30332 U.S.A.}
\email{domarov3@gatech.edu}



\begin{abstract}
In this work, we propose a novel implementation of Newton's method
for solving  semi-discrete optimal transport (OT) problems for cost functions
which are a positive combination of $p$-norms, $1<p<\infty$.
It is well understood that the solution of a semi-discrete OT problem
is equivalent to finding a partition of a bounded region in
Laguerre cells, and we prove that the Laguerre cells are star-shaped
with respect to the target points.  
By exploiting the geometry of the Laguerre cells, we obtain an efficient and 
reliable implementation of Newton's method to find the sought network
structure.
We provide implementation details and extensive results in support of our
technique in 2-d problems, as well as comparison with other 
approaches used in the literature.
\end{abstract}
	
\begin{flushright} \small Version of \today $, \,\,\,$ \xxivtime \end{flushright}

\maketitle

\pagestyle{myheadings}
\thispagestyle{plain}
\markboth{Dieci and Omarov}{Semi-Discrete OT Problems}
	
%

\section{Introduction}\label{Intro}
In this work we provide an implementation of Newton's method to solve
semi-discrete optimal transport (OT) problems, for several cost functions
given by combination of $p$-norms ($p\ne 1,\infty$), 
several source and target probabilities, and
several different domains.  Our main goals are: (i) to justify and exploit
the geometry
of the decomposition in Laguerre cells of the underlying solution of the
problem, (ii) to provide robust algorithmic development, (iii) to give extensive
numerical testing of our implementation on several examples, and
(iv) to give quantitative measure of the goodness of our results, also
in comparison with existing approaches.

Optimal transport problems have
been receiving considerable attention from the analytical community for
many years.  The history of the field dates back 200+ years with the work
of Monge in 1781, \cite{Monge1781a}, though it was only in the 1940's with 
the works of Kantorovich, \cite{Kantorovich1942a, Kantorovich1948a},
that the very fruitful relation to optimization, duality, and linear programming 
ideas made it possible to settle important theoretical questions as well
as open the door to development of numerical methods.  Because of this history,
it is common to use the naming of Monge-Kantorovich (MK) problem for
the classical OT problem.  Numerical
techniques for approximating the solution of OT problems have followed suite
in the past 30-35 years, and new methods continue being developed and
analyzed, in no small part due to the many
applications (e.g., in imaging, economics, machine learning) where
OT problems arise.
Excellent expository books of both theory and numerics have recently
appeared and we refer to these as general introduction to the
topic; in particular, see the book of Santambrogio,
\cite{santambrogio2015optimal}, and also the more computationally
oriented monograph of Peyr\'e and Cuturi, \cite{PeyreCuturi}.

A fairly general formulation is the following.  We are given
two compactly supported measures $\mu$ and $\nu$ on $\Rn$, $n>1$,
absolutely continuous with respect to Lebesgue measure,
and with respective supports $X$ and $Y$.  Further, let
$c: X\times Y \to \R^+$ be a cost function, whereby
$c(x,y)$ describes the cost of transporting one unit of mass 
$x$ to $y$.   The original Monge problem problem is to find
$$\inf_{T:\ T_\sharp \mu=\nu}  \ \int_X c(\x,T(\x))d\mu(x) \ , $$
where $T_\sharp$ is the push-forward of $\mu$ through $T$,
that is $(T_\sharp\mu)(A)=\mu(T^{-1}(A))$.
As stated, it is far from obvious that a map (Monge map) $T$ can
be found, and in fact the theoretical problem stayed open until the relaxation
introduced by Kantorovich in the 1940's.
After the work of Kantorovich, the MK problem was reformulated 
as finding a transport plan $\pi$ of probability densities
while minimizing a cost functional.  

That is, the MK problem consists in determining the 
{\emph{optimal transport plan}}, that is the
joint pdf $\pi \in {\mathcal{M}}(X \times Y)$, with marginals $\mu$ and
$\nu$, which realizes the $\inf$ of
the functional $J(\pi)$ below:
\begin{equation}\label{InfOT}
\inf_{\pi} J  := \int_{X\times Y} c(x,y) \,d\pi\,.
\end{equation}
Under suitable assumptions on $\mu,\nu,c$, as well explained in the book
of Santambrogio \cite{santambrogio2015optimal}, 
thanks to works by Ambrosio, Benamou, Brenier, Gangbo, McCann, Villani, and others,
(see \cite{Ambrosio2003, BenamouBrenier, GangboMcCann, Villani}), 
it was established that an optimal plan exists and can be
chosen as a measure concentrated in $\Rn \times \Rn$ on the graph of a map 
$T : \Rn \to \Rn$, the optimal transport map.  Oversimplifying the theory,
a key point was to restrict to a strictly convex cost like $c(\x,\y)=\|\x-\y\|_2^2$.
In this case, the map is obtained
as gradient of a smooth potential $\phi$ that satisfies the Monge-Ampere partial
differential equation; numerical methods in this case have also been studied,
e.g., see \cite{BFO, Froese} as well as our recent
comparison \cite{DieciOmarov-Comparison} of different methods for solving the 
continuous OT problem with cost given by the 2-norm square.

However, here our own interest is when $\nu$ is discrete, that is it is supported at finitely
many distinct points $\y_i$'s:  
$\nu(\y)=\nu_i\delta( \y_i)$, $i=1,\dots, N$, $\nu_i>0$ and $\sum_{i=1}^N\nu_i = 1$.  
These are called {\emph{semi-discrete}} optimal transport problems, and one
has the restriction that $\y_i$'s are in $X$.
In this case, 
a remarkable
consequence of duality affords a major simplification of the structure
of the optimal map, under much less stringent conditions on the cost
function.  As recalled in \cite[Chapter 5]{PeyreCuturi}
(and see also Section \ref{SubS_MinProb} and Theorems 
\ref{Thm_ExistUniq} and \ref{Thm_Minimizer} below), 
under suitable assumptions on the cost, 
there is an optimal transport plan, and it is realized as a decomposition of $X$
in Laguerre cells, which are fully determined by a vector of optimal weights
$w\in \R^N$ (see \eqref{LagOmega} below)
$$A(\y_i)=\{\x\in X \ \ | \ \ c(\x,\y_i) - w_i\leq c(\x,\y_j)-w_j, \ \ j\neq i\}, \ \
i=1,2,\dots,N.$$
Once an optimal vector of weights is determined, and for as long as the
cells' boundaries have $\mu$-measure $0$, it is a simple observation
that there is an optimal map and that it is piecewise constant, in that
it maps every $\x\in A_i$ to the same $\y_i$, up to sets of $\mu$-measure $0$.

Given the relevance in applications, there are several works concerned 
with numerical solution of semi-discrete OT problems, the chief
differences between them being the cost functions considered.
By far, the most widely studied case is for the cost function given by
the 2-norm square: $c(\x,\y)=\|\x-\y\|_2^2$.  In this case, the Laguerre
tessellation is also known as ``{\sl Power Diagram}'', see \cite{AurenhammerPowerDiagrams}.  
The beauty of this case is in the
simplicity of the cells' structure: they are polytopes, their boundaries are given by 
(hyper) planes.  Unsurprisingly, very elegant and efficient tools of
computational geometry are used to compute power diagrams, with
remarkable success and stunning display of the resulting subdivision.
A likely incomplete list of computational works which tackle semi-discrete
optimal transport in the case of the 2-norm squared includes the works
\cite{AurenhammerPowerDiagrams},  \cite{SolomonGoesEtc},
\cite{Cuturi2013a}, \cite{Sinkhorn1967a}, \cite[Chapter 5]{PeyreCuturi},
\cite{Levy2015a}, \cite{Merigot2013a}, 
\cite{Ruschendorf2000a}, \cite{de2019differentiation}, \cite{kitagawa2019convergence},
\cite{BourneRoper}.
Conceivably, the reason for the algorithmic successes in this
case of the 2-norm squared as cost is given by the intimate relation of
power diagrams and Voronoi tessellation.  In fact, a most useful
algorithmic result in this case allows to lift a power diagram to
be the same as a Voronoi tessellation in one higher space dimension, 
see \cite{Levy2015a}, and computation of Voronoi tessellation has reached
remarkable maturity in computational geometry: in the plane, the classical sweeping algorithm 
of Fortune \cite{Fortune1986a} is still very popular, though there
are of course several more modern presentations, e.g. \cite{DuFaberGunz}, and there
are also widely available codes for commonly used
computational environments, like the {\tt Matlab} functions {\tt voronoi} and
{\tt voronoin}.

The situation for other cost functions is much less developed, and it is our present
concern.  To be sure, there are important theoretical results for more general
costs than the 2-norm squared, for example see \cite{geiss2013optimally}.  However,
computational techniques do not appear to be very well developed, probably because of
the lack of simple geometrical shapes of the  Laguerre cells.  Notable exceptions
are the work \cite{LucaJD}, that considered general $p$-norm costs ($p\ne 1,\infty$), 
and the work \cite{Hartmann2020SemidiscreteOT} that considered the case of $p=2$,
the standard Euclidean norm.  As in \cite{LucaJD}, we will also
restrict to $p$-norms as costs ($1<p<\infty$), but we will follow a very
different algorithmic path from \cite{LucaJD}.    (Comparison with the approach
of \cite{Hartmann2020SemidiscreteOT} is done in the main body of
this paper; in particular, see Section \ref{2dInt}).
The key ideas of the ``boundary method'' of \cite{LucaJD}
were to track/refine only the boundary-zone of the Laguerre cells, and not their
interiors, by a grid-subdivision/refinement algorithm, and to improve upon the
weights $w$ through an original adaptation of the auction algorithm.  In the present
work, our key insight is that --although the shape of the Laguerre cells is not as simple
as having boundary given by intersection of hyper-planes-- the Laguerre cells are
star-shaped.  This is the most we can say for general cost functions of the
type we consider.  Moreover, 
we will adopt a Newton's method approach to refine the weights, and very accurate 
manifold techniques to track the boundaries between cells.  As a result, in the end 
our method appears to be much more accurate than any other method we know of, but the
price we pay is that we need the density $\rho$ associated to the measure $\mu$
(i.e., $d\mu=\rho(\x)d\x$) to be a sufficiently smooth density. That being the case, we
develop a fully adaptive algorithm with error control, and we show that it performs
extremely well both in terms of speed and accuracy on a variety of problems.  
Our main contributions are the following: (i) new algorithms for Newton's method, 
exploiting star-shapedness of the Laguerre cells, 
(ii) unified treatment of the problem, (iii) comparison with
other techniques, (iv) a full set of replicable results, with accuracy assessment
well beyond the graphical display of results witnessed in the literature (how can
one distinguish two close figures with the naked eye?), and (v) a rigorous
numerical analysis, including the introduction of a proposal for a condition
number of semi-discrete OT problems.

A plan of the paper is as follows.  In Chapter \ref{Theory}, we first present
results on Laguerre cells, under generous conditions on the cost
function, see Section \ref{SubS_Star}.  Then, in Section \ref{SubS_Smooth} we
further restrict to cost functions that allow us to obtain smooth boundaries
between the cells.  These first two set of results hold regardless of the
measure $\mu$ (and $\nu$).  Then, in Section \ref{SubS_MinProb}, we finally
recall results identifying the solution of the semi-discrete problem as the
minimum of a convex functional.  Chapter \ref{Lag2d} details the algorithms we
implemented in 2-d, Chapter \ref{Newton} discusses Newton's method in more
details, and Chapter \ref{CompExamp} gives computational examples.  Some
theoretical results are given in an Appendix, and pseudo-codes of all of our
algorithms are given at the end so that our results can be replicated by others.

\section{Theoretical Results}\label{Theory}
In this Chapter, we give results about the geometrical and
analytical (measure theoretical) properties of the partition in
Laguerre cells we are after.    Some of these
results have appeared before, often for 
more specialized cost functions
than those we consider, and appropriate references will be given.

Although Laguerre tessellations, and optimal transport maps,
can be --and are-- defined in the entire ambient space $\R^n$, 
in practice we will restrict to working on a bounded set $\Omega$,
that is henceforth chosen to satisfy the following characterization.

\begin{ass}[Bounded domain $\Omega$]\label{WhatOmega}
$\Omega$ will always refer to a compact convex domain
of $\R^n$, of nonzero $n$-dimensional Lebesgue measure,
and with $\cont^k$ piecewise smooth boundary, $k\ge 1$.
\end{ass}

For completeness, recall that a $\cont^k$ piecewise smooth boundary of
$\Omega$ means that
$$\partial \Omega = \bigcup_{i=1}^N \bar \Gamma_i\ ,\,\ $$
and each $\Gamma_i$ is a $\cont^k$ $(n-1)$-dimensional surface.
Moreover, two of the $\bar \Gamma_i$'s may intersect each other along
a $\cont^k$ $(n-2)$-dimensional surface, etc..

\begin{rems}\label{Cubes}
\item[(a)]
The value of
$k$ in Assumption \ref{WhatOmega} should be at least the same value
of $k$ as we have for the function $F$ in Lemma \ref{Lem_Smooth}.
\item[(b)]
In practice, 
$\Omega$ is often given by a (hyper)-cube, which surely
satisfies Assumption \ref{WhatOmega} with piecewise $\cont^\infty$
boundary.  To witness, the standard unit square $S$ will be our
prototypical choice in $2$-d, and its boundary is clearly piecewise
smooth, being the union
of four closed intervals, each intersecting two other intervals at
a vertex. 
\end{rems}

Although the results below hold regardless of the measures
$\mu$ and $\nu$, we will need to restrict to appropriate classes
of cost functions.
At first, we will consider cost functions that satisfy the following
assumption, which is sufficient to prove geometrical properties
of the Laguerre cells.  Later, see Restriction \ref{Rest_Cost2}, we will
further restrict the class of cost functions considered, in order to
obtain analytical results on the cells themselves.

\begin{ass}[Cost function]\label{Rest_Cost1}
The cost function $c:\Rn\times \Rn \to \R^+$
satisfies the requirements of a distance function:
	\begin{equation*}\begin{split}
			\text{positivity:} \quad &
	\begin{cases}
	\, &	c(\x,\y) \geq 0 , \ \ \forall \x,\y\in\Rn,\\
	\, &	c(\x,\y) = 0\iff \x = \y, \ \ \forall \x,\y\in\Rn,
	\end{cases}\\
	{\text{symmetry:}}\quad &		c(\x,\y) = c(\y,\x),\ \ \forall \x,\y\in\Rn,\\
	{\text{triangular inequality:}}\quad &		c(\x,\y) \leq c(\x,z) + c(z,\y),\ \ \forall \x,\y,z\in\Rn, \\
\end{split}	\end{equation*}
and additionally we also assume that the cost function satisfies these
two properties:
	\begin{equation*}\begin{split}
	{\text{homogeneity:}}\quad &	c(t\x,t\y) = |t| c(\x,\y),\ \ \forall t\in\mathbb{R},\ \forall \x,\y\in\mathbb{R}^n,\\
	{\text{shift invariance:}}\quad &	c(\x+z,\y+z) = c(\x,\y),\ \ \forall \x,\y,z\in\mathbb{R}^n.\\
\end{split}	\end{equation*}
\end{ass}
\begin{rems}
\item[(a)]
It is an interesting theoretical consequence 
of the last two properties above 
that (e.g., see \cite[p.21]{bollobas1999linear}) 
there exists a vector norm $\| \cdot \|$ such that $c(\x,0)=||\x||$, 
$ \forall \x\in\mathbb{R}^n$, and therefore our cost is a continuous
function.
\item[(b)]
Below, see Restriction \ref{Rest_Cost2}, we will give very concrete examples
of the cost functions we consider in practice.  Presently, we 
remark that commonly adopted choices, like the Euclidean norm squared: 
$c(\x,\y)=\|\x-\y\|_2^2$, do not
satisfy Assumption \ref{Rest_Cost1}, and hence some of our results below
(in particular, Lemma \ref{Lem_yi} and Theorem \ref{Thm_Star})
do not apply to this case.
\end{rems}

\subsection{Star-Shapedness Results}\label{SubS_Star}
In this Section, we define Laguerre cells and tessellation and 
show some important geometrical properties of Laguerre cells, that are
relevant for our later algorithmic development to solve semi-discrete
optimal transport problems.  Our main result is Theorem \ref{Thm_Star},
that appears to be new, at least in the given generality.  The results in
this section are derived irrespective of the measure theoretic setting
of optimal transport and only require that the cost function satisfy
Assumption \ref{Rest_Cost1}.

\begin{defn}[Laguerre tessellation]\label{Def_Lag}
Given a cost function $c$ satisfying Assumption \ref{Rest_Cost1}, 
given a set of $N$ distinct 
points $\y_i\in\mathbb{R}^n$, $i=1,2,\dots, N$, $N\ge 2$,
and given a shift vector $w\in\mathbb{R}^N$, the Laguerre tessellation of
$\mathbb{R}^n$ associated to $c(\cdot, \cdot)$, 
to the $\y_i$'s, and to $w$, 
is given by the set of $N$ regions ({\emph{Laguerre cells}})
\begin{equation}\label{Eq_LagRd}
    \LL (\y_i) = \{\x\in\mathbb{R}^n\ \ | \ \ c(\x,\y_i) - w_i\leq c(\x,\y_j)-w_j, \ \ \forall \ j=1,\dots, N, \ j\neq i\}, \ \ i=1,2,\dots,N.
\end{equation}
The points $y_i$, $i=1,2,\dots, N$, are called {\emph{target points}}. 
See Figure \ref{Laguerre}. \\
Moreover, let $\Omega$ be as in Assumption \ref{WhatOmega} and let
$\y_i\in\Omega^\mathrm{o}$, $i=1,2,\dots, N$. Then, the Laguerre 
tessellation of $\Omega$ is the set of {\emph{Laguerre cells}} (see Figure \ref{Laguerre})
\begin{equation}\label{LagOmega}
    A(\y_i) = \LL(\y_i)\cap \Omega = \{\x\in\Omega \ \ | \ \ c(\x,\y_i) - w_i\leq c(\x,\y_j)-w_j, \ \ j\neq i\}, \ \ i=1,2,\dots,N.
\end{equation}
\end{defn}

\begin{figure}[ht]
	\centering
	\begin{subfigure}[b]{0.45\linewidth}
		\centering
		\includegraphics[width=\linewidth]{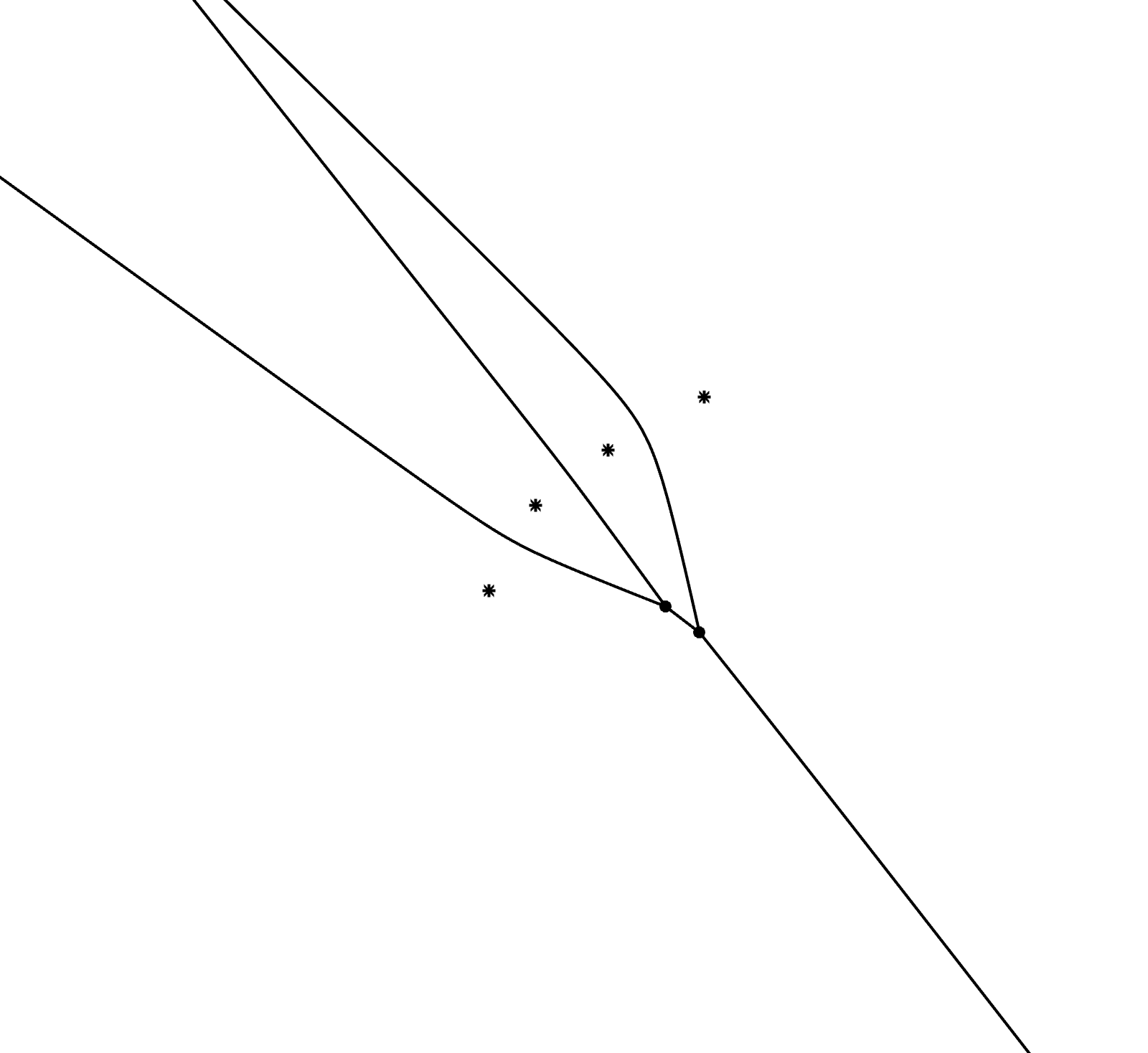}
		\caption{}
	\end{subfigure}
	\hfill
	\begin{subfigure}[b]{0.45\linewidth}
		\centering
		\includegraphics[width=\linewidth]{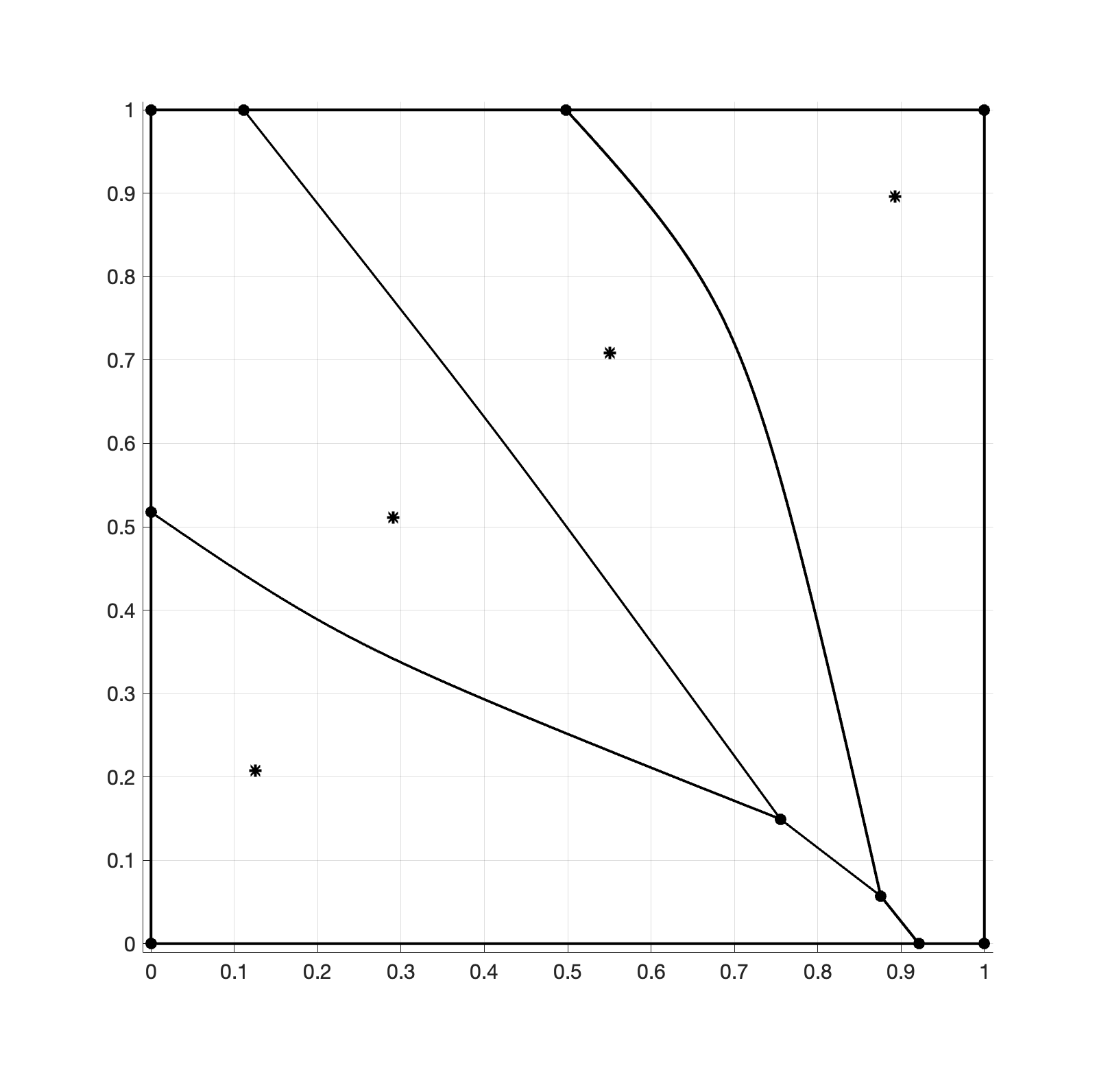}
		\caption{}
	\end{subfigure}
	\caption{For the same target points, $w$, and cost:  (A)
		Laguerre cells in $\R^2$ and (B) Restricted to the unit square.}
	\label{Laguerre}
\end{figure}

\begin{rem}
The special case of $w=0$ gives the tessellation in terms of {\emph{Voronoi cells}}. 
This is the standard {\emph{proximity}} problem. 
\end{rem}

\begin{lem}[Every $x\in \R^n$ is in some $\LL(\y_i)$]\label{Lem_Allx}
For any given $\x\in\Rn$, there exist an index $i$ such that $\x\in\LL(\y_i)$.
\end{lem}
\begin{proof}
This result follows from the fact that $\Rn = \bigcup_{i=1}^N\LL(\y_i)$.
\end{proof}

\begin{defn}[Active boundary and interior]\label{BdryLag}
The {\emph{active boundary}}, or simply {\emph{boundary}}, 
between two Laguerre cells $\LL(\y_i)$ and $\LL(\y_j)$, $i\ne j$, 
such that
$\LL(\y_i)\cap \LL(\y_j) \ne \emptyset$, 
is indicated with $\LL_{ij}$, and it is given by the set of
all $\x\in \Rn$ satisfying the relation
$$ c(\x,\y_i) - w_i = c(\x,\y_j) - w_j\implies w_{ij} = c(\x,\y_i) - c(\x,\y_j)\ ,
\,\ w_{ij}=w_i-w_j \ , $$
as well as the $(N-2)$ inequalities
$$ c(\x,\y_i) - w_i \le c(\x,\y_k) - w_k \ ,\,\, k\ne i,j\ . $$
The boundary of a Laguerre cell $\LL(\y_i)$, indicated with
$\partial \LL(\y_i)$, is given by
$$\partial \LL(\y_i)=\bigcup_{\begin{smallmatrix} j=1,\  j\ne i: \\ \LL(\y_i)\cap \LL(\y_j) \ne \emptyset 
\end{smallmatrix}}^N \LL_{ij}\ .$$
Accordingly, we define the interior of $\LL(\y_i)$, and indicate it
with $\LL(\y_i)^\mathrm{o}$, to be the set
$$ \LL(\y_i)^\mathrm{o}= \LL(\y_i)\backslash\partial\LL(\y_i) \ .$$
\end{defn}

\begin{lem}[Existence of boundary between cells: necessary condition]\label{Lem_wij}
Let the cost function $c$ satisfy Assumption \ref{Rest_Cost1}.
The necessary condition for having an active boundary  $\LL_{ij}$
in $\Rn$ between $\LL(\y_i)$ and $\LL(\y_j)$ is that
$$|w_i - w_j|\leq c(\y_i,\y_j) = c(\y_j,\y_i)\ .$$
\end{lem}

\begin{proof}
Let $w_{ij} = w_i - w_j$. 
By Definition \ref{BdryLag}, 
the values $x\in \LL_{ij}$ must satisfy the following relation:
\begin{equation*}
    c(\x,\y_i) - w_i = c(\x,\y_j) - w_j\implies w_{ij} = c(\x,\y_i) - c(\x,\y_j).
\end{equation*}
From the triangular inequality:
\begin{equation*}
    c(\x,\y_i)\leq c(\x,\y_j) + c(\y_j,\y_i)\implies w_{ij}\leq c(\x,\y_j) + c(\y_j,\y_i) - c(\x,\y_j) = c(\y_j,\y_i).
\end{equation*}
Similarly, with $w_{ji} = w_j - w_i = -w_{ij}$, one has
$$w_{ji} = c(\x,\y_j) - c(\x,\y_i)\le c(\x,\y_i) + c(\y_i,\y_j)- c(\x,\y_i)=
c(\y_i,\y_j)$$
and since $c(\y_i,\y_j) = c(\y_j,\y_i)$, we obtain 
\begin{equation*}
    \begin{Bmatrix}w_{ij}\leq c(\y_j,\y_i)\\-w_{ij}\leq c(\y_j,\y_i)\end{Bmatrix}\implies |w_{ij}|\leq c(\y_j,\y_i)
\end{equation*}
and the result follows.
\end{proof}

\begin{lem}[$\y_i\in \LL(\y_i)$]\label{Lem_yi}
Let the cost function $c$ satisfy Assumption \ref{Rest_Cost1}.
If $\LL(\y_i)\ne \emptyset$, each target point $\y_i$ is contained in its own
Laguerre cell $\LL(\y_i)$.
\end{lem}
\begin{proof}
Assume $\y_i\not\in \LL(\y_i)$. Then, because of Lemma \ref{Lem_Allx}, 
$\y_i\in \LL(\y_j)\backslash \partial \LL(\y_i)$, for some $j\neq i$.\\
\textit{Case 1.} Let $\LL(\y_i)\cap \LL(\y_j)\neq \emptyset$. Then, by construction of Laguerre cells, it follows that:
\begin{equation*}
    \y_i\not\in \LL(\y_i),\ \y_i\in \LL(\y_j) 
    \implies c(\y_i,\y_i) - w_i > c(\y_i,\y_j) - w_j\implies - w_i + w_j > c(\y_i, \y_j).
\end{equation*}
It follows that $-w_{ij} > c(\y_i,\y_j)$. But this contradicts Lemma \ref{Lem_wij}.\\

\noindent\textit{Case 2.} Let $\LL(\y_i)\cap \LL(\y_j)= \emptyset$. Then,
since $\y_i\in \LL(\y_j)\backslash \partial \LL(\y_i)$:
\begin{equation*}
    c(\y_i, \y_j) - w_j < c(\y_i,\y_i) - w_i = -w_i\implies w_j - w_i > c(\y_i, \y_j).
\end{equation*}
Next, let $\x_0\in \LL(\y_i)\backslash \partial \LL(\y_j)$. Such point $x_0$ exists since $\LL(\y_i)\neq \emptyset$.  Then:
\begin{equation*}
    c(\x_0, \y_i) - w_i < c(\x_0, \y_j) - w_j\implies w_j - w_i < c(\x_0, \y_j) - 
    c(\x_0, \y_i).
\end{equation*}
Combining the above two equations:
\begin{equation*}
    \begin{Bmatrix}
        w_j - w_i > c(\y_i, \y_j)\\
        w_j - w_i < c(\x_0, \y_j) - c(\x_0, \y_i)
    \end{Bmatrix}\implies c(\y_i,\y_j)<c(\x_0,\y_j)-c(\x_0,\y_i).
\end{equation*}
Using the triangular inequality:
\begin{equation*}
    c(\x_0,\y_j)\leq c(\x_0,\y_i)+c(\y_i,\y_j)\implies c(\y_i,\y_j)<c(\x_0,\y_i)+c(\y_i,\y_j)-c(\x_0,\y_i)=c(\y_i,\y_j),
\end{equation*}
which is the sought contradiction.
\end{proof}

\begin{lem}\label{Lem_yi_bound}
Let the cost function $c$ satisfy Assumption \ref{Rest_Cost1}.
If $\y_i\in\partial\LL(\y_i)$, then $\LL(\y_i)$ is degenerate, i.e. 
$\LL(\y_i)^\mathrm{o} = \emptyset$.
\end{lem}
\begin{proof}
W.l.o.g., we can assume that $\y_i\in\LL_{ij}$, for some $j\neq i$.
Then, it follows that:
\begin{equation*}
    c(\y_i,\y_i) - w_i = c(\y_i,\y_j) - w_j \implies w_j - w_i = c(\y_i,\y_j).
\end{equation*}
Next, assume by contradiction that there exist some point $\x_0\in\LL(\y_i)^\mathrm{o}$. Then $x_0$
needs to satisfy the following relation:
\begin{equation*}
    c(\x_0,\y_i) - w_i < c(\x_0,\y_j) - w_j\implies w_j - w_i< c(\x_0,y_j) - c(\x_0, \y_i) \implies c(\y_i,\y_j) < c(\x_0,y_j) - c(\x_0, \y_i).
\end{equation*}
But this contradicts the triangular inequality.
\end{proof}

\begin{thm}[$\LL(\y_i)$ is star shaped]\label{Thm_Star}
Let the cost function $c$ satisfy Assumption \ref{Rest_Cost1}.
Let $\x\in \LL(\y_i)$ and 
$S(t) = \begin{Bmatrix}t(\x-\y_i)+\y_i,\hspace{0.5em}0\leq t\leq 1\end{Bmatrix}$. Then, for all $t\in [0,1]$, 
$S(t)\in \LL(\y_i)$.  \\
As a consequence,
each Laguerre cell $\LL(\y_i)$ is star-shaped in $\Rn$
with respect to its own target point $\y_i$.
\end{thm}
\begin{proof}
The case $t = 0$ follows from Lemma \ref{Lem_yi} and the
case $t = 1$ holds since $\x\in \LL(\y_i)$. 
Next, by contradiction, assume that there exists $t_0\in(0,1)$ such that 
the point $\x_0 = t_0(\x-\y_i) + \y_i$ is not in $\LL(\y_i)$. 
W.l.o.g., let $\x_0\in \LL(\y_j)\backslash \partial \LL(\y_i)$, for some
$j\neq i$.  Then, it follows that, since
\begin{equation*}\begin{split}
         & \x\in \LL(\y_i) \implies c(\x,\y_i) - w_i\leq c(\x,\y_j) - w_j,
 \\
& \x_0=S(t_0) \implies c(\x, \y_i) = c(\x, \x_0) + c(\x_0, \y_i),
\end{split}\end{equation*}
where the second equality comes about from the following
\begin{equation*}\begin{split}
&  c(\x, \x_0) + c(\x_0, \y_i) = c(\x,\y_i+t_0(\x-\y_i))+c(\y_i+t_0(\x-\y_i),\y_i) 
= \\
& c((1-t_0)(\x-\y_i),0)+c(t_0(\x-\y_i),0)=(1-t_0)c(\x-\y_i,0)+t_0c(\x-\y_i,0)=
c(\x-\y_i,0)=c(\x,\y_i)\ .
\end{split}\end{equation*}
As a consequence, we get
$$c(\x, \x_0) + c(\x_0, \y_i)-w_i \leq c(\x,\y_j) - w_{j}\implies c(\x_0,\y_i) - w_i\leq c(\x,\y_j) - c(\x,\x_0) -w_j.
$$
From the triangular inequality, it follows that:
\begin{align}\label{Eq_Star}
    &c(\x_0,\y_i) - w_i\leq c(\x,\y_j) - c(\x,\x_0) -w_j\leq c(\x,\x_0) + c(\x_0,\y_j) - c(\x,\x_0) -w_j =  c(\x_0,\y_j) - w_j,\nonumber\\
    &\hspace{12em}\implies c(\x_0,\y_i) - w_i\leq c(\x_0,\y_j) - w_j.
\end{align}
In addition, since $\x_0\in \LL(\y_j)\backslash \partial \LL(\y_i)$,
it holds that
\begin{equation*}
    \x_0\in \LL(\y_j) \implies c(\x_0,\y_j) - w_j< c(\x_0,\y_i) - w_i\implies c(\x_0,\y_i) - w_i>c(\x_0,\y_j) - w_j.
\end{equation*}
But this inequality contradicts \eqref{Eq_Star}, and hence the entire
segment $S(t)\in \LL(\y_i)$.\\
The fact that $\LL(\y_i)$ is star-shaped with respect to $\y_i$ follows at once
by the fact that every point in $\LL(\y_i)$ can be reached via a line segment
from $\y_i$ entirely contained in $\LL(\y_i)$.
\end{proof}

\begin{cor}\label{Cor_Star}
Let $\Omega$ be as in Assumption \ref{WhatOmega}, let $\y_i\in\Omega^\mathrm{o}$,
$i=1,2,\dots,N$, and let the cost function $c$ satisfy Assumption \ref{Rest_Cost1}.
Then each Laguerre cell $A(\y_i)$ is star-shaped in $\Omega$ with respect 
to the point $\y_i$.
\end{cor}
\begin{proof}
This result follows from Theorem \ref{Thm_Star} and the fact that the 
intersection of a convex set with a star-shaped set is a star-shaped set (e.g., see \cite{bounceur2018boundaries}).
\end{proof}

\begin{rem}\label{RemStar}
For Voronoi diagrams, star-shapedness is simple to infer; 
e.g., see \cite{Gao_Voronoi} for results about even more general Voronoi diagrams.
But, for Laguerre cells, we have found much
fewer conclusive results, except for the standard Euclidean distance.
Indeed, that $\LL(\y_i)$ are star-shaped when the cost function is the
$2$-norm has been observed before, 
for example see \cite[Example 2.12]{Bourne2018SemidiscreteUO}, 
\cite[p.21]{Dobrin}.   In this case of  
$c(\x,\y)=\|\x-\y\|_2$, Laguerre tessellation are often called
{\emph{Apollonius diagrams}} or also {\emph{additively weighted Voronoi 
diagrams}}, and in 2-d it has been
observed many times that the boundaries between cells are arcs of 
hyperbolas (e.g., see \cite{Bourne2018SemidiscreteUO}, \cite{Dobrin},
\cite{Hartmann2020SemidiscreteOT}, or \cite{Sharir}), although we
are not aware of a similar result for cost functions being a $p$-norm, 
with $p\ne 2$.   
\end{rem}

\subsection{Smoothness results: what cost function?}\label{SubS_Smooth}
In this section we give analytical results on smoothness
of the boundaries $\LL_{ij}$.  To get these results, we further restrict the class
of appropriate cost functions, as well as an appropriate measure
$\mu$.  As said in the introduction,
$\mu$ will always be assumed to be 
equivalent to Lebesgue measure
and with bounded density $\rho$ in $\Omega$: 
$d\mu=\rho(\x)d\x$.
These requests on $\mu$ 
allow to infer, under mild conditions on the cost function
$c(x,y)$ (namely continuity, or even lower semi-continuity), that in 
\eqref{InfOT} the
``$\inf$'' is realized as a ``$\min$'', a fact noted  several times
(e.g., see \cite{Ambrosio2003, Cuesta1993a, Pratelli}).
However, we need that there is a unique partition (map), and 
to obtain this result we need to place further
restrictions on the cost function.  Indeed, a fundamental result 
of Cuesta-Albertos and Tuero-Diaz
\cite[Corollary 4]{Cuesta1993a} states that there exists a unique optimal 
partition (map) {\bf if} the cost function $c(x,y)$ gives $\mu$-measure
$0$ to the cells boundaries.  Because of this, the key concern 
will be to determine that the boundaries have indeed $0$-measure, in
particular under what conditions on the cost function this is possible.

Several authors have addressed the issue of selecting an appropriate
cost, and proposed 
restrictions on the cost function in order to obtain a desired end result.
For example, in \cite{geiss2013optimally} the authors restricted to
what they called {\emph{admissible costs}} (see below), 
in \cite{Bourne2018SemidiscreteUO} they restricted to
so-called  {\emph{radial cost}} functions, whereas in 
\cite{kitagawa2019convergence} the authors required the cost
to satisfy conditions that they called (Twist)-(Reg)-(QC), and in 
\cite{de2019differentiation} the authors required the cost
to satisfy their conditions (Diff-2) and (Cont-2); more on all these
cost functions is below.
We should note that the restrictions placed on the
cost in the works \cite{Bourne2018SemidiscreteUO},  
\cite{kitagawa2019convergence}, and \cite{de2019differentiation},
appear to have been motivated
by the cost given by the $2$-norm square, a case that
violates our original Assumption \ref{Rest_Cost1} and that therefore is
not of specific interest in the present work, since we may fail to have
star-shaped Laguerre cells, a fact which plays a very important role
in our algorithmic development.  A different approach was taken in
\cite{LucaJD} and in \cite{Hartmann2020SemidiscreteOT}.  Namely, in
\cite{LucaJD} , Dieci and Walsh considered cost given by
$p$-norms ($1<p<\infty$) and
in \cite{Hartmann2020SemidiscreteOT} Hartmann and Schuhmacher 
restricted to the $2$-norm cost, $\|\x-\y\|_2$; these choices
surely satisfies Assumption 
\ref{Rest_Cost1} and thus lead to star-shaped Laguerre cells, although
this fact was not noticed/used in \cite{LucaJD}
or \cite{Hartmann2020SemidiscreteOT}.
The choice we will make is just slightly more general than
the choice made in \cite{LucaJD}, though more general than that in
\cite{Hartmann2020SemidiscreteOT}: 
we will henceforth restrict to cost functions
satisfying the following Restriction \ref{Rest_Cost2}, which
satisfy of course our previous Assumption \ref{Rest_Cost1}.

\begin{Rest}\label{Rest_Cost2}
Hereafter, the cost functions considered are given by positive,
finite, combinations of $p$-norms, with $1<p<\infty$.  That is,
\begin{equation}\label{OurCost}
c(\x,\y)=\sum_{l=1}^m \alpha_l \|\x-\y\|_{p_l},\quad {\rm{where}}\quad
\alpha_l>0,\,\ p_l\in (1,\infty)\ .
\end{equation}
\end{Rest}

The relation between our class of cost functions, and the choices made
in previous works is elaborated on in Remark \ref{CompareCosts} below.

\begin{rem}\label{CompareCosts}
The radial cost functions of \cite{Bourne2018SemidiscreteUO} are
functions of the form $c(\x,\y)=g(\|\x-\y\|_2)$, where $g$ is an
increasing function, e.g., $\|\x-\y\|_2^2$; our cost functions do not
fit in this class.  Conditions (Reg) of 
\cite{kitagawa2019convergence} can be generally satisfied by our cost
functions, in light of Lemma \ref{SmothNorms}; however, 
conditions (Twist) and (QC) of \cite{kitagawa2019convergence} are more
difficult to verify, and it may 
possibly not hold for our cost functions.  Finally, the authors
of \cite{de2019differentiation} require sufficiently regular cost
functions (which we can require, see Lemma \ref{SmothNorms}), 
satisfying (in their notation) conditions (Diff-2-a), (Diff-2-b),
(Diff-2-c) and (Cont-2); it is not difficult to verify that, with our
cost functions, conditions (Diff-2-b)-(Diff-2-c)-(Cont-2) can be
enforced, but condition (Diff-2-a) is more problematic.   This condition
asks for the existence of $\epsilon>0$ such that
$\|\nabla_{\x} c(\x,\y_i)-\nabla_{\x} c(\x,\y_j)\| >\epsilon$
for all $\x$ in $\LL_{ij}$, and thus the
authors of \cite{de2019differentiation} ask for this condition to hold
in $\R^n$, and not just in a bounded set $\Omega$.  For completeness,
in Lemma \ref{Lem_NNZGrad} we show that this condition holds for us
for even values of $p$, with $c(\x,\y)=\|\x-\y\|_p$,
when we look at it on a bounded convex set $\Omega$, which is
what we do in practice in our numerical computations.
Finally, our choice of cost function always gives 
an admissible cost according to the definition of \cite{geiss2013optimally},
as we prove below.  
\end{rem}

\begin{defn}[Admissible cost; \cite{geiss2013optimally}]\label{Def_AdmCost}
Let $\Leb$ be Lebesgue measure, let $\Omega$ be as in Assumption 
\ref{WhatOmega}, and let $\y_1, \y_2$ be two distinct target points
in $\Omega^\mathrm{o}$.  A cost function satisfying Assumption
\ref{Rest_Cost1} is called {\emph{admissible}} if 
there exist values $w_{a}\leq w_{b}$ such that:
\begin{itemize}
\item 
$\Leb(A(\y_1))$ is continuously increasing from $0$ to $\Leb(\Omega)$ as
$w_{12} = w_1 - w_2$ increases from $w_{a}$ to $w_{b}$, where
$A(\y_1) = \{\x\in\Omega \ \ | \ \ c(\x,\y_1) - w_1 \leq c(\x,\y_2) - w_2\}$;
\item likewise, 
$\Leb(A(\y_2))$ is continuously decreasing from $\Leb(\Omega)$ to $0$ as 
$w_{12} = w_1 - w_2$ increases from $w_{a}$ to $w_{b}$, where
$	A(\y_2) = \{\x\in\Omega \ \ | \ \ c(\x,\y_2) - w_2 \leq c(\x,\y_1) - w_1\}$.
\end{itemize}
\end{defn}

\begin{lem}\label{Claim_pNorm}
All $p$-norms,  $1<p<\infty$, are admissible cost functions as in
Definition \ref{Def_AdmCost}, that is
$c(\x,\y) = ||\x-\y||_p$, for $1<p<\infty$, is an admissible cost function.
Moreover, a positive combination of admissible cost functions is an
admissible cost function, and in particular a cost function satisfying 
Restriction \ref{Rest_Cost2} is admissible according to
Definition \ref{Def_AdmCost}.
\end{lem}
\begin{proof}
For $1<p<\infty$, it is enough to observe that when 
$c(\x,\y) = ||\x-\y||_p$,  choosing
$w_a =-c(\y_1,\y_2)= -||\y_1-\y_2||_p$ and 
$w_b=c(\y_1,\y_2)=||\y_1-\y_2||_p$ will work in Definition \ref{Def_AdmCost}.  \\
Next, take a cost function 
$c(\x,\y)=\sum_{k=1}^m \alpha_k c_k(\x,\y)$ where
$\alpha_k>0$ and $c_k(\cdot,\cdot)$ are admissible cost functions.
Let $w_a^{(k)}$ and $w_b^{(k)}$ be values satisfying the requirements for
admissibility of the $c_k(\cdot,\cdot)$ functions in Definition \ref{Def_AdmCost}.
Then, in Definition \ref{Def_AdmCost}, we can take 
$w_a=\sum_{k=1}^m \alpha_k w_a^{(k)}$ and 
$w_b=\sum_{k=1}^m \alpha_k w_b^{(k)}$ 
to obtain that $c(\cdot,\cdot)$ is admissible. 
\end{proof}

An interesting an useful result for our cost functions is
the following.

\begin{lem}\label{Lem_Point}
	Let $\Omega$ be as in Assumption \ref{WhatOmega} and let $\y_1,\y_2\in\Omega^\mathrm{o}$ be two distinct target points. 
	Let $w_{12} = w_1 - w_2$ satisfy $|w_{12}|\leq c(\y_1,\y_2)$
	(see Lemma \ref{Lem_wij}). 
	If $c(\cdot,\cdot)$ is an admissible cost function as in Restriction 
	\ref{Rest_Cost2}, then $A(\y_1)\cap A(\y_2) \neq \emptyset$, where
	\begin{align*}
		&A(\y_1) = \{\x\in\Omega \ \ | \ \ c(\x,\y_1) - w_1 \leq c(\x,\y_2) - w_2\},\\
		&A(\y_2) = \{\x\in\Omega \ \ | \ \ c(\x,\y_2) - w_2 \leq c(\x,\y_1) - w_1\}.
	\end{align*}
\end{lem}
\begin{proof}
	By direct computation, it is straightforward to show that the point 
	$x_0 = \y_1 +\frac{k+1}{2}(\y_2-\y_1)$, 
	$k = \frac{w_{12}}{c(\y_1,\y_2)}\in[-1,1]$, is contained in the boundary 
	set $A(\y_1)\cap A(\y_2)$, where
	\begin{equation*}
		A(\y_1)\cap A(\y_2) =  \{\x\in\Omega \ \ | \ \ c(\x,\y_1) - w_1 = c(\x,\y_2) - w_2\}.
	\end{equation*}
	Hence the result follows.
\end{proof}

Next, in order to obtain $\mu$-measure $0$ for the boundaries between
Laguerre cells, we will use some results of differential geometric
flavor, and for these we need some preliminary results.

\begin{fact}\label{Claim_Limit}
Let $a\in \R$.  For $s<1$ and $r\ge 1$:
$$\lim_{a\to 0} \frac{a^r}{|a|^s} \ = \ 0\ .$$
\end{fact}
\begin{proof}
This follows from the fact that 
$\frac{a}{|a|^s} =|a|^{1-s} \text{sign}(a) \to 0$ since $\text{sign}(a)$ is 
bounded and $s<1$.
\end{proof}

\begin{lem}\label{SmothNorms}
Consider the function $f(\x) = ||\x||_p$, $1<p<\infty$, for $\x$ 
in an open subset $B$ of $\R^n$ not containing the origin.
Then, $f$ is a $\cont^k$ function in $B$, where 
$$k=\begin{cases}  \ \infty\ , & {\text{if}} \,\ p \,\ {\text{even}} \\
	 \ m \ , & \text{if} \,\ m<p\le m+1,  \,\, m\in \Z^+ . \end{cases} $$
\end{lem}
\begin{proof}
For all $\x$, the function $f(\x)$ can be rewritten as:
\begin{align*}
f(\x) = [|\x_1|^p + |\x_2|^p + \dots + |\x_d|^p]^{1/p}\ .
\end{align*}
Then, it is enough to check smoothness of the single term $|\x_j|^p$, 
$j\in[1,d]$, since $\x\neq 0$. 
Note that $|\x_j|^p = \x_j^p$ when $p$ is an even integer, and so the
result follows for $p$ even.
Otherwise, using Fact \ref{Claim_Limit} it is simple
to show that the function $|\x_j|^p$ is $t$ times continuously 
differentiable where $t<p\leq t+1$.
\end{proof}

Let $F(\x)$ be the function corresponding to the boundary set between
two Laguerre cells $L(\y_1)$ and $L(\y_2)$ in $\R^n$:
\begin{equation}\label{FunctionF}
	F(\x) = c(\x,\y_1) - c(\x,\y_2) \implies L(\y_1)\cap L(\y_2) = 
	\{\x\in\R^n \ \ | \ \ F(\x) = w_{12}\}.
\end{equation}

\begin{lem}\label{Lem_Smooth}
	Let 
	$c(\cdot,\cdot)$ be an admissible cost function as in Restriction 
	\ref{Rest_Cost2}, 
	$\y_1, \y_2\in \R^n$ be two distinct target points, and 
	$w_{12} = w_1 - w_2$ satisfy $|w_{12}|< c(\y_1,\y_2)$ (see 
	Lemma \ref{Lem_wij}).
	Assume that $w_{12}$ is a regular value for the function $F(\x)$,
	that is that $\nabla F(\x_0)\neq 0$, $\forall \x_0\in L(\y_1)\cap L(\y_2)$.
	Then, the boundary set $\{x: \ F(\x) = w_{12}\}$ is a
	$\cont^k$ submanifold of dimension $(n-1)$ in $\mathbb{R}^n$.
\end{lem}
\begin{proof}
Lemma  \ref{Lem_Point} gives the existence of at least one point
$\x_0: F(\x_0)=w_{12}$, where $F$ is given in \eqref{FunctionF}.
By our Restriction \ref{Rest_Cost2}, the function $F(\x)$
is a $\cont^k$ function, with $k\ge 1$, in an open neighborhood
of $\LL_{ij}$.
Since $\nabla F(\x_0)\neq 0$ is assumed, 
then the smooth manifold theorem tells us that
$S=F^{-1}(w_{12}) \subset \R^n$ is a
$(n-1)$-dimensional $\cont^k$ manifold embedded in $\R^n$.
\end{proof}

\begin{rem}\label{Rem_EvenP}
We notice that the request that $w_{12}$ is a regular value for $F$ is not a 
severe restriction, since the Morse-Sard Theorem guarantees
that almost all values $w_{12}$ are regular values for $F$.
Furthermore, in the Appendix, we show in Claim \ref{p-even} 
that $\nabla F(\x_0)\neq 0$, 
$\forall \x_0\in L(\y_1)\cap L(\y_2)$, whenever the cost function is a
$p-$norm with $p$ an even number.  As a consequence, for 
cost functions given by positive combination of $p$-norms with
$p$ even, there is no need to assume that $\nabla F(\x_0)\neq 0$, since
it is always satisfied.
\end{rem}

\begin{rem}\label{Curve2d}
In 2-d, Lemma \ref{Lem_Smooth} says that the set $F^{-1}(w_{12})$ is a 
smooth curve.   Cfr. with Remark \ref{RemStar} for $c(\x,\y)=\|\x-\y\|_2$.
\end{rem}

\begin{thm}\label{Thm_Smooth}
Let $\y_i\in\R^n$, $i=1,2,\dots,N$, be a set of $N$
distinct target points, 
and $c(\cdot,\cdot)$ be an admissible cost function as in Restriction 
\ref{Rest_Cost2}.  In addition, assume that $\y_i\not\in\partial L(\y_i)$, 
$i=1,2,\dots,N$.  Then, the boundary of each Laguerre cell $L(y_i)$
consists of at most $(N-1)$ 
sections of $\cont^k$ $(n-1)$-dimensional manifolds of $\R^n$.

Let $\Omega$ be as in Assumption \ref{WhatOmega}, and let 
$\y_i\in\Omega^\mathrm{o}$, $i=1,2,\dots,N$, be a set of $N$
distinct target points, such that $\y_i\not\in\partial A(\y_i)$, 
$i=1,2,\dots,N$.  Then, the boundary of each Laguerre cell $A(y_i)$
consists of 
sections of $\cont^k$ $(n-1)$-dimensional manifolds of $\R^n$.
\end{thm}

\begin{proof}
The first statement follows from Lemma \ref{Lem_Smooth}
and the form of $\partial \LL(\y_i)$:
$$\partial \LL(\y_i)=\bigcup_{\begin{smallmatrix} j=1,\  j\ne i: \\ \LL(\y_i)\cap \LL(\y_j) \ne \emptyset 
\end{smallmatrix}}^N \LL_{ij}\ .$$
	
The second statement follows from the fact that
\begin{equation*}
\partial A(\y_i) = \partial (\LL(\y_i)\cap \Omega), \, 
\ \ i = 1,2,\dots,N,
\end{equation*}
from Lemma \ref{Lem_Smooth},
and from the assumed piecewise smoothness of $\partial \Omega$.	
\end{proof}


\begin{minipage}[l]{.40\textwidth}
\begin{rem}\label{2Sections}
In general, the union of several pieces
of the same smooth $(n-1)$-dimensional manifold can be part of
$A(\y_i)\cap A(\y_j)$.   See the Figure on the right.
It is guaranteed to be a single smooth section when there are 
exactly two target points.
\end{rem}
\end{minipage}\hfill
\begin{minipage}[c]{.60\textwidth}
	\hskip1cm
	\includegraphics[width=0.83\linewidth]{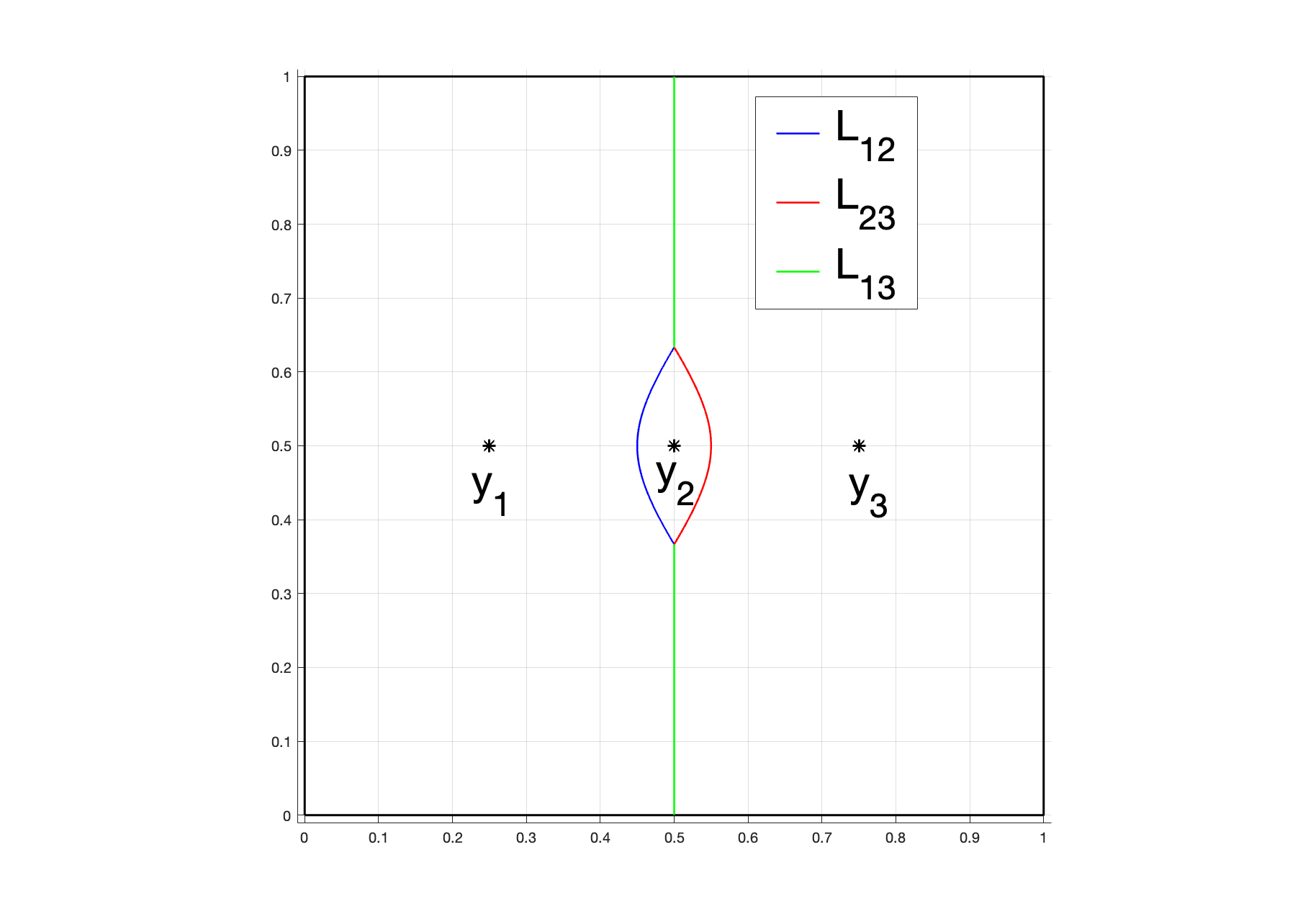}
\end{minipage}

\begin{lem}\label{Lem_ZeroMeas}
Let $\Leb$ be Lebesgue measure.
Let $\Omega\subset\Rn$ be as in Assumption \ref{WhatOmega},
and let $\y_1,\y_2\in\Omega^\mathrm{o}$ be two distinct target points.
Let $c(\cdot,\cdot)$ be a cost function satisfying
Restriction \ref{Rest_Cost2}.
Then, $\Leb(A(\y_1)\cap A(\y_2)) = 0$  holds, where
\begin{align*}
    &A(\y_1) = \{\x\in\Omega \ \ | \ \ c(\x,\y_1) - w_1 \leq c(\x,\y_2) - w_2\},\\
    &A(\y_2) = \{\x\in\Omega \ \ | \ \ c(\x,\y_2) - w_2 \leq c(\x,\y_1) - w_1\}.
\end{align*}
\end{lem}
\begin{proof}
From Lemma \ref{Lem_Smooth}, $\Leb (L_{ij})=0$, and
thus the result follows.
\end{proof}

\begin{cor}\label{Cor_ZeroMeas}
Let $\Leb$ be Lebesgue measure.
Let $\Omega\subset\Rn$ be as in Assumption \ref{WhatOmega},
and let $\y_i\in\Omega^\mathrm{o}$, $i=1,2,\dots,N$, be $N$ distinct 
target points. If $c(\cdot,\cdot)$ is an admissible cost function, then 
$\Leb(\partial A(\y_i)) = 0$, $i=1,2,\dots,N$.
\end{cor}
\begin{proof}
This result follows from Lemma \ref{Lem_ZeroMeas}, and the facts that
$\displaystyle{\partial \LL(\y_i)=\bigcup_
{\begin{smallmatrix} j=1,\  j\ne i: \\ \LL(\y_i)\cap \LL(\y_j) \ne \emptyset 
\end{smallmatrix}}^N \LL_{ij}}$, and that
$A(\y_i)=\LL(\y_i)\cap \Omega$.
\end{proof}

\begin{exm}\label{Rem_CountExamp}
We show next why the $1$-norm and the $\infty$-norm must be
excluded from consideration in our restriction on allowed costs.  The reason is
that they violate Lemma \ref{Lem_ZeroMeas} and lead to lack of uniqueness, in
general.  This fact is actually well understood, 
and already true for Voronoi diagrams, but we give the 
following counterexamples because they are simple and
can be worked out by hand.
\begin{itemize}
\item 
{\emph{Counterexample for $1$-norm}}. Let $\Omega = [0,3]^2$, 
$\y_1 = \begin{pmatrix}1\\2\end{pmatrix}$, 
$\y_2 = \begin{pmatrix}2\\1\end{pmatrix}$, $w_1 = w_2 = 0$. 
The boundary between the two Laguerre cells with respect to the $1$-norm will 
satisfy the following equation:
\begin{align*}
    \x = \begin{pmatrix}x_1\\x_2\end{pmatrix}\in A(\y_1)\cap A(\y_2)& \implies||\x-\y_1||_1 = ||\x-\y_2||_1 \\
    & \implies |x_1-1| + |x_2-2| = |x_1-2| + |x_2-1|.
\end{align*}
But, to illustrate,
any point in the square $0\leq x_1\leq 1$ and $0\leq x_2\leq 1$ will satisfy 
the above relation:
$(1-x_1) + (2-x_2) = (2-x_1) + (1-x_2)$.  
As a result, $\mu(A(\y_1)\cap A(\y_2)) \neq 0$.
See Figure \ref{CounterExs}: any point in the shaded regions is both
in $A(\y_1)$ and $A(\y_2)$.
In particular, 
we cannot satisfy $\mu(A(\y_1))+\mu(A(\y_2))=1$. 

\item 
{\emph{Counterexample for $\infty$-norm}}. Let  $\Omega = [0,3]^2$, 
$\y_1 = \begin{pmatrix}1\\2\end{pmatrix}$, $\y_2 = 
\begin{pmatrix}3\\2\end{pmatrix}$, $w_1 = w_2 = 0$. The boundary 
between the two Laguerre cells with respect to $\infty$-norm will satisfy the 
following equation:
\begin{gather*}
\x= \begin{pmatrix}x_1\\x_2\end{pmatrix}\in A(\y_1)\cap A(\y_2)\implies||\x-\y_1||_{\infty} = ||\x-\y_2||_{\infty},\\
\max\{|x_1-1|, |x_2-2|\} = \max\{|x_1-3|, |x_2-2|\}.
\end{gather*}
Now, take a point $\x$ in the square $1\leq x_1\leq 2$ and $0\leq x_2\leq 1$.
Then:
\begin{align*}
    \max\{(x_1-1), (2-x_2)\} = \max\{(3-x_1), (2-x_2)\}\implies (2-x_2) = \max\{(3-x_1), (2-x_2)\}.
\end{align*}
But, to illustrate,
any point in the region $x_1 - x_2 > 1$ will satisfy this relation and
thus $\mu(A(\y_1)\cap A(\y_2)) \neq 0$. 
See Figure \ref{CounterExs}: any point in the shaded regions is both
in $A(\y_1)$ and $A(\y_2)$. 
In particular, 
we cannot satisfy $\mu(A(\y_1))+\mu(A(\y_2))=1$. 
\end{itemize}
\begin{figure}[ht]
	\centering
	\begin{subfigure}[b]{0.45\linewidth}
		\centering
		\includegraphics[width=\linewidth]{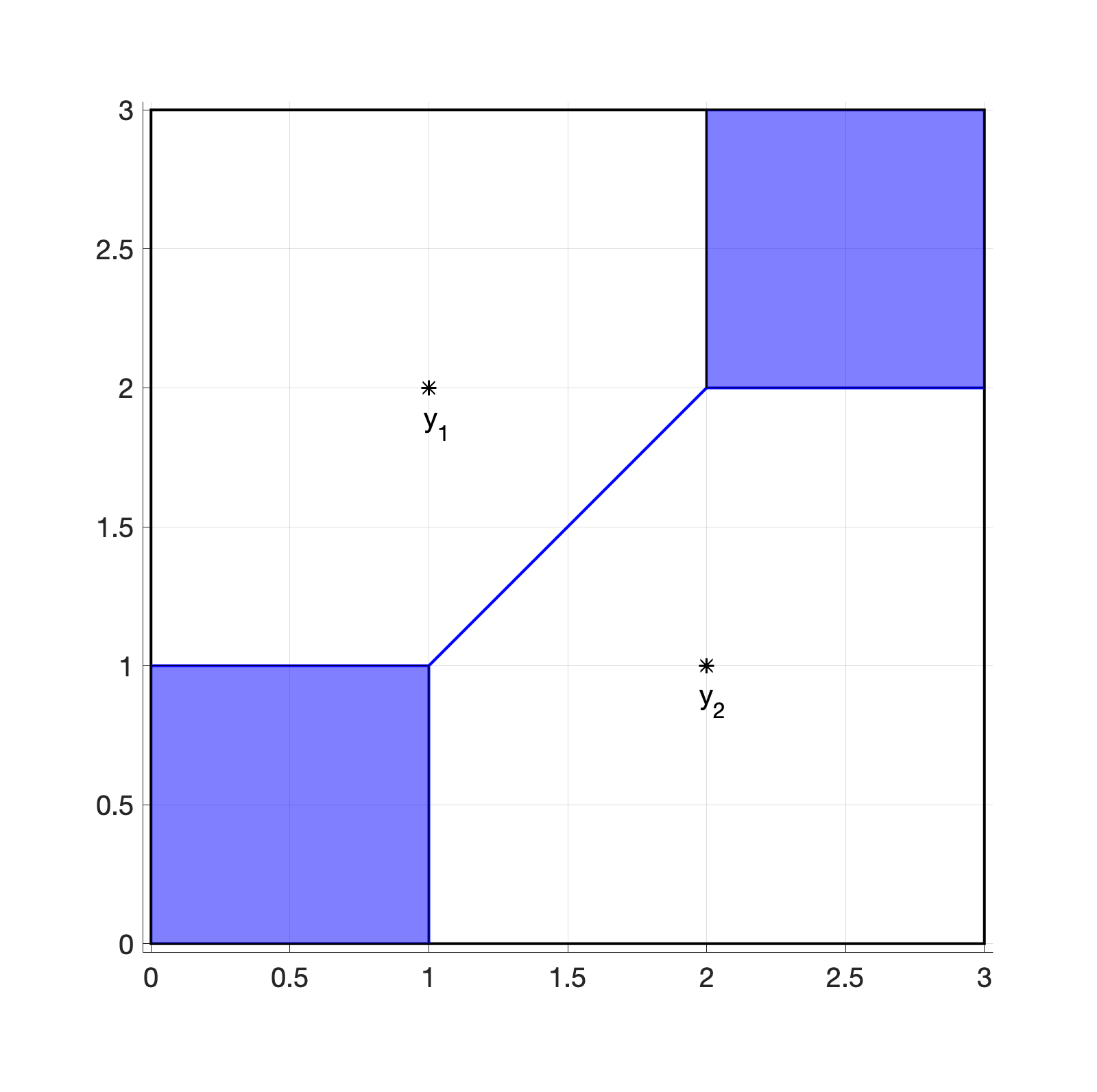}
		\caption{$c(\x,\y)=\| \x-\y\|_1$}
	\end{subfigure}
	\hfill
	\begin{subfigure}[b]{0.45\linewidth}
		\centering
		\includegraphics[width=\linewidth]{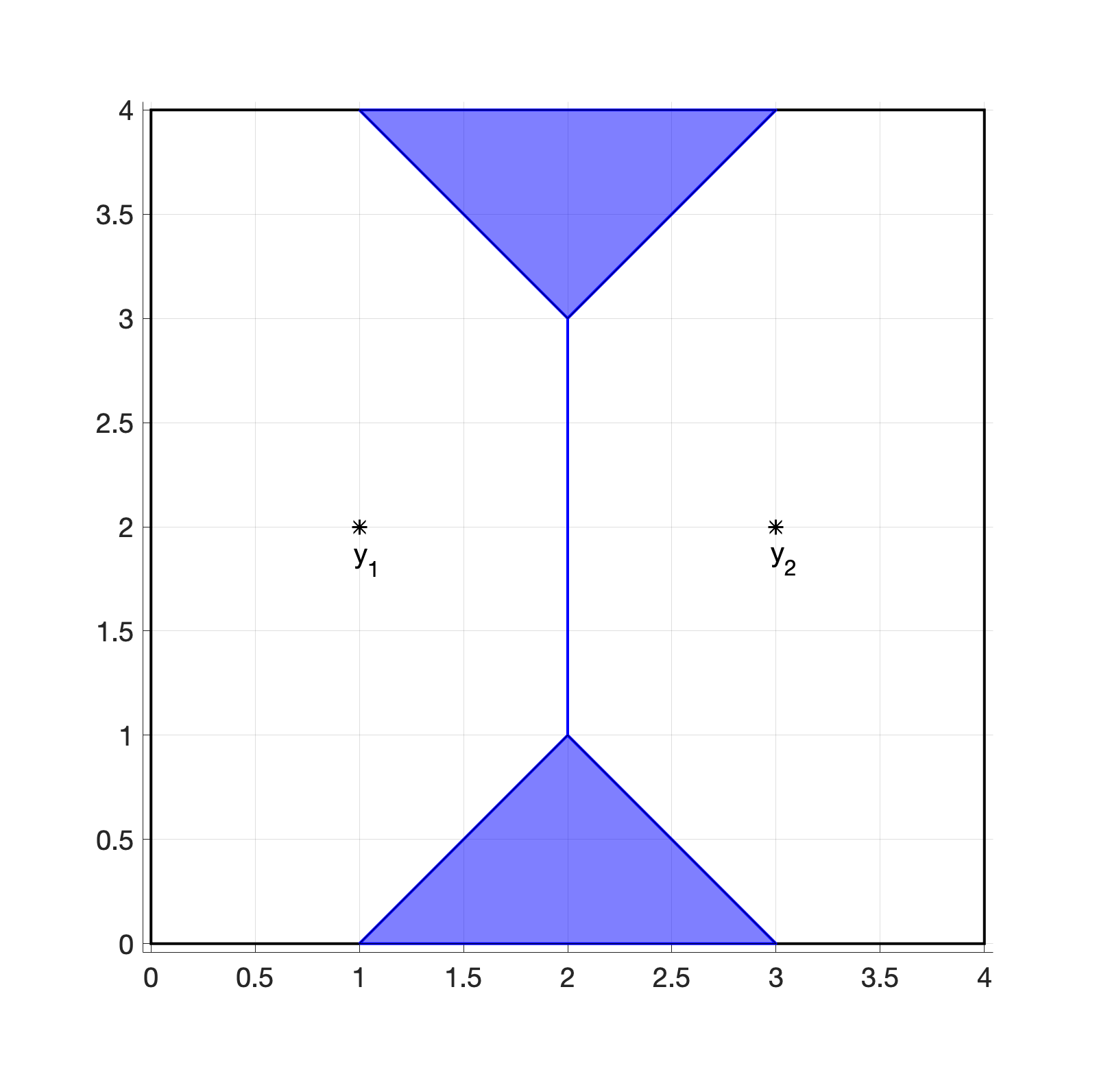}
		\caption{$c(\x,\y)=\| \x-\y\|_\infty$}
	\end{subfigure}
	\caption{Counterexamples to uniqueness $w=0$  The shaded regions
	are in both $A(\y_1)$ and $A(\y_2)$.}
	\label{CounterExs}
\end{figure}
\end{exm}

Finally, the following elementary result is useful to justify some steps of our
algorithm (see Section \ref{Shooting}).

\begin{cor}\label{Cor_Inter}
	Let $\Omega\subset\Rn$ be as in Assumption \ref{WhatOmega},
	let $c$ satisfy Restriction \ref{Rest_Cost2},
	and let $\y_i\in\Omega^\mathrm{o}$, $i=1,2,\dots,N$, be $N$ distinct 
	target points.
	Then, any open line segment from a 	point on $\partial A(\y_i)$ 
	to the point $\y_i$ is contained in the interior of $A(\y_i)$.
\end{cor}
\begin{proof}
This result follows from Corollary \ref{Cor_Star}, piecewise smoothness of
$\partial A(\y_i)$ and the fact that $\Leb(\partial A(\y_i))=0$.
\end{proof}

\subsection{Minimization Problem Results}\label{SubS_MinProb}
The results in this section are needed to justify use of our numerical
method for finding the OT partition.  With the exception of Theorem
\ref{feas-w-Thm}, and possibly point \textit{(2)} of Theorem 
\ref{HessProp}, 
the results below are known in the literature on semi-discrete
optimal transport.

\begin{thm}\label{Thm_ExistUniq}
Let $\Omega\subset\Rn$ be as in Assumption \ref{WhatOmega},
$\mu$ be a probability measure with support on $\Omega$, equivalent to 
Lebesgue measure, let 
$\y_i\in\Omega^\mathrm{o}$, $i=1,2,\dots,N$, be a set of $N$ 
distinct target points, 
and $c(\cdot,\cdot)$ be an admissible cost function as in Restriction
\ref{Rest_Cost2}.  Given a discrete measure $\nu$ supported at
the $\y_i$'s, that is $\nu(\x)=\nu_i\delta( \y_i)$, $i=1,\dots, N$,
$\nu_i>0$ and $\sum_{i=1}^N\nu_i = 1$, then 
there exists a shift vector $w$ (unique up to adding the same constant 
to all its entries), and a partition of $\Omega$: $\{A_i\}$, 
such that the partition is unique (up to sets of $\mu$-measure $0$), and 
such that $\mu(A(\y_i)) = \nu_i$ is satisfied for all
$i=1,2,\dots, N$.
\end{thm}

\begin{proof}
Since our cost function is an admissible cost (Lemma \ref{Claim_pNorm}), 
this result is in \cite[Theorem  1]{geiss2013optimally}.
\end{proof}

A very important fact, relative to the solution $w$ of Theorem 
\ref{Thm_ExistUniq}, is the following {\emph{feasibility result}},
which strengthens Lemma \ref{Lem_wij} and also has crucial 
algorithmic implications.

\begin{thm}[Feasibility of $w$]\label{feas-w-Thm}
With the notation of Theorem \ref{Thm_ExistUniq}, the vector $w$
must satisfy the following relation
\begin{equation}\label{w-feasible}
| w_i-w_j | < c(\y_i, \y_j)\ , \,\ i\ne j\ .
\end{equation}
\end{thm}
\begin{proof}
From Lemma \ref{Lem_yi_bound}, we have that if
$$w_j-w_i =c(y_i,y_j) \ \Rightarrow \ L^{\mathrm{o}}(y_i)=\emptyset \ .$$
But then, $A^{\mathrm{o}}(\y_i)=\emptyset$ and $\mu(A(\y_i))=0$ and we
cannot satisfy $\mu(A(\y_i))=\nu_i$ with $\nu_i>0$.
\end{proof}

\begin{thm}\label{Thm_Minimizer}
Let $\mu$ be a bounded probability measure on 
$\Omega\subset\Rn$, with $\Omega$
satisfying Assumption \ref{WhatOmega},
let $\y_i\in\Omega^\mathrm{o}$,
$i=1,2,\dots,N$, be a set of $N$ distinct target points, let
$c(\cdot,\cdot)$ be an admissible cost function satisfying Restriction 
\ref{Rest_Cost2}, and let $\nu$ be a discrete measure supported at
the $\y_i$'s, that is $\nu(\x)=\nu_i \delta(\y_i)$, $i=1,\dots, N$,
$\nu_i>0$ and $\sum_{i=1}^N\nu_i = 1$.
Then, the function $\Phi(w)$
\begin{equation}\label{ToMinimize}
    \Phi(w) = -\sum_i^N \int_{A(\y_i)} [c(\x,\y_i) - w_i]d\mu - \sum_i^N w_i\nu_i,
\end{equation}
is convex, $\mathcal{C}^1$-smooth, 
and has gradient
\begin{equation}\label{Gradient}
\nabla\Phi(w) = \{\mu(A(\y_i)) - \nu_i\}_{i=1:N}.
\end{equation}
Moreover, a minimizer of $\Phi$ gives the optimal transport map (partition),
up to sets of $\mu$-measure $0$.
\end{thm}

\begin{proof}
With our Restriction \ref{Rest_Cost2}, the result follows from 
from \cite[Proposition 1]{de2019differentiation}.  (See also
\cite[Theorem 1.1 and Corollary 1.2]{kitagawa2019convergence}.) 
\end{proof}

\begin{rems}\label{ShiftByConstant}
\item[(1)]
In agreement with the uniqueness statement of Theorem \ref{Thm_ExistUniq},
it is plainly apparent that in Theorem \ref{Thm_Minimizer} neither
$\Phi(w)$ nor $\nabla\Phi(w)$ change if a constant value is added to all
entries of $w$ (the same constant, of course).   For our problem,
this is the only lack of uniqueness we have.
\item[(2)]
It is worth emphasizing that the dependence on $w$ in the gradient
$\nabla \Phi$ is through the regions $A(\y_i)$, recall \eqref{LagOmega}.
\end{rems}

Theorem \ref{Thm_Minimizer} is the basis of minimization approaches
to find the OT partition; e.g., see \cite{Hartmann2020SemidiscreteOT}
for the $2$-norm cost.
Our numerical method  is based on seeking a root of the gradient in 
\eqref{Gradient}, and we will use Newton's method to this end.
Of course,  use of Newton's method requires the Jacobian of
$\nabla \Phi$ (the Hessian of $\Phi$), and we will use the
formulas for the Hessian  given in
\cite{de2019differentiation}, 
\cite{kitagawa2019convergence}, \cite{PeyreCuturi}.  Indeed, in spite
of slightly different notations, the expressions of the Hessian given
in these works are actually the same.
One of the merits of our work will be to show that our implementation
of Newton's method, whereby we exploit the geometrical features
of Laguerre's cells and their star-shapedeness, is very
efficient and --in principle-- it leads to arbitrarily accurate results,
at least in $\R^2$ and for  sufficiently smooth density function $\rho$.


\begin{thm}\label{Thm_Hessian}
The Hessian matrix of $\Phi(w)$ is given by:
\begin{equation}\begin{split}\label{Hessian}
\frac{\partial^2\Phi}{\partial w_i\partial w_j} &= \int_{A_{ij}=A(\y_i)\cap A(\y_j)} 
\frac{-\rho(x)}{||\hspace{0.25em}\nabla_{\x} c(\x,\y_i) - \nabla_{\x} c(\x,\y_j)\hspace{0.25em}||}d\sigma, \hspace{1em}j\neq i; \hspace{0.5em} i,j =1,2,\dots,N,\\
\frac{\partial^2\Phi}{\partial w_i^2} & = 
-\sum_{j\neq i} \frac{\partial^2\Phi}{\partial w_i\partial w_j}, 
\hspace{1em}i=1,2,\dots,N,
\end{split}\end{equation}
where $d\sigma$ is the $(n-1)$-dimensional infinitesimal surface element.
\end{thm}
\begin{proof}
See \cite[Theorem 1]{de2019differentiation}. 
\end{proof}

\begin{thm}\label{HessProp}
    Let $H$ be the Hessian matrix of $\Phi(w)$ from Theorem \ref{Thm_Hessian}
    associated to a feasible $w$.  Then, $H$ has the following properties:
    \begin{itemize}
        \item[(1)] $H$ is symmetric positive semi-definite, and
        the vector of all ones is in the nullspace of $H$;
        \item[(2)] $H$ has rank $(N-1)$. In particular, $0$ is a simple eigenvalue.
    \end{itemize}
\end{thm}

\begin{proof}
The fact that $H$ is symmetric positive semi-definite
follows immediately from the structure of the Hessian matrix in  
\eqref{Hessian} and the fact that the function $\Phi(w)$ is convex. 
Likewise, the fact that $H\begin{bmatrix}  1 \\ \vdots \\1 \end{bmatrix}=0$
is a simple verification. \\
Next, let $\lambda_1(H)=0\le \lambda_2(H) \le \dots \le \lambda_N(H)$ be
the eigenvalues of $H$.
To show statement \textit{(2)}, we need to argue that $\lambda_2(H)>0$.
Consider the undirected graph $\Gamma$ 
whose vertices are the target points $\y_i$, and where 
there is an edge between vertices $\y_i$ and $\y_j$ (because of
our assumptions on $\mu$) if 
$A(\y_i)\cap A(\y_j)\neq\emptyset$.   Then, this graph is strongly
connected by construction (there is a path from any vertex to any other 
vertex moving along the graph edges), since $w$ is feasible.
By construction of $H$, $H$ is a {\emph{generalized Laplacian}} 
associated to this graph (see \cite[Section 13.9]{godsil2001algebraic} for
the concept of generalized Laplacian).  Then, we can argue like in 
\cite[Lemma 13.9.1]{godsil2001algebraic} to infer that $\lambda_1(H)$ is
a simple eigenvalue.  Indeed, take the matrix $H-\eta I$ for some value of
$\eta$ so that the diagonal entries of $H-\eta I$ are non-positive and consider
the matrix $-(H-\eta I)$ which has nonnegative entries and it is irreducible,
since the graph associated to it is strongly connected (because $\Gamma$ 
was).  Then, by the Perron-Frobenius theorem,
the largest eigenvalue of  $-(H-\eta I)$ is simple, and it is its spectral radius
and therefore it is the positive value $\eta$, and 
we can take the vector of all $1$'s as the associated Perron eigenvector. 
Therefore, $0$  is a simple eigenvalue of $H$ and
the rank of the Hessian matrix is $(N-1)$.
\end{proof}

\section{Laguerre Cell Computation: 2D}\label{Lag2d}
Here we describe the numerical algorithms we implemented.
Since all of the computations
we will present are in $\R^2$, we will delay discussion of appropriate
techniques in $\R^3$ to a future work, and our 
algorithmic description below is done relative to 2-d only. 

The fundamental task is to compute $\nabla \Phi$ in \eqref{Gradient} and
then form the Hessian in \eqref{Hessian}.  In essence, this requires being
able to compute the $\mu(A_i)$'s, the boundaries $A_{ij}$'s, and the
line integrals in \eqref{Hessian}.
Below, we will describe fast and precise algorithm for these tasks, and for ease
of exposition,  we are going to make two simplifications in our presentation: 
(i) the cost function is a $p-$norm (for $1<p<\infty$), 
$c(\x,\y) = ||\x-\y||_p$, and (ii) the domain $\Omega$ is the square 
$[a,b]\times [a,b]$.
It is fairly straightforward to adapt the description of our algorithms to 
the case of cost given by a positive combination of $p$-norms, and
for several choices of $\Omega$.  Our computations in
Chapter \ref{CompExamp} will exemplify these two points; see
Sections \ref{DifferentCosts} and \ref{DifferentDomains}.

First of all, we note a main advantage of our technique: it is fully parallelizable.
To witness, the computation of $\mu(A(\y_i))$ is independent of the 
computation of $\mu(A(\y_j))$, $\forall j\neq i$. 
Hence, we can formulate the task as follows.  Let $Y=\y_i$, for some $i$, then:
\begin{align*}
	&\text{Compute } \mu(A(Y)) = \int_{A(Y)}d\mu,\\
	&\text{where } A(Y) = \{\x\in\Omega \ \ | \ \ ||\x-Y||_p - W \leq ||\x-\y_j||_p - w_j,\ \forall j : \ \y_j\ne Y \}\ .
\end{align*}
The second feature of our algorithm is that it exploits star-shapedness result 
from Section \ref{SubS_Star}. Thus:
\begin{gather*}
	A(Y) = \{\x\in\Omega \ \ | \ \ ||\x-Y||_p - W \leq ||\x-\y_j||_p - w_j,\ \ \forall jj : \ \y_j\ne Y \}
	\ ,\quad \text{or} \\
	A(Y) = \{\x = Y + z\in\Omega \ \ | \ \ ||z||_p - W \leq ||z+ Y-\y_j||_p - w_j,\ \ \forall j : \ \y_j\ne Y
	\}\ .
\end{gather*}
Since $A(Y)$ is star shaped with respect to $Y$, we can parametrize in a unique
way the values on a ray from $Y$ in terms of distance from $Y$ along the 
direction determined by an angle $\theta$; in other words, in the above we 
can write  
$z=D(\theta)\begin{pmatrix}\cos{(\theta)}\\\sin{(\theta)}\end{pmatrix}$,
and thus write for $A(Y)$:
\begin{equation}\label{PolarArea1}
	A(Y) = \{(\theta, D(\theta)), \theta \in[0, 2\pi] \ \ | \ \
	 ||D(\theta) \begin{pmatrix}\cos{(\theta)}\\\sin{(\theta)}\end{pmatrix}||_p - W \leq ||D(\theta)\begin{pmatrix}\cos{(\theta)}\\\sin{(\theta)}\end{pmatrix}+ Y-\y_j||_p - w_j,\ \ 
	 \forall j : \ \y_j\ne Y \}\ .
\end{equation}
Clearly, 
the boundary of the Laguerre cell, $\partial A(Y)$, is also 
parametrizable in terms of  the angle $\theta$.  Namely, reserving the
notation $r(\theta)$ for the points on the boundary, we have
\begin{equation}\label{PolarBdry1}
\x \in\partial A(Y) \iff  
x = Y + r(\theta)\begin{pmatrix}\cos{(\theta)}\\ \sin{(\theta)}\end{pmatrix} \,\
\text{for some unique} \,\ \theta\in[0,2\pi) .
\end{equation}
When determining $\partial A(Y)$, 
in addition to the value of $r$ corresponding to each angle $\theta$, we will 
monitor also an index $k$, which indicates the neighboring cell or the portion of
the physical boundary for that value of the angle $\theta$ :
\begin{equation}\begin{split}\label{index}
	x & = Y + r(\theta) \begin{pmatrix}\cos{(\theta)}\\\sin{(\theta)}\end{pmatrix}
	\in A(Y)\cap A(\y_j)\implies k = j \\
	x & = Y+r(\theta) \begin{pmatrix}\cos{(\theta)}\\\sin{(\theta)}\end{pmatrix}
	\in A(Y)\cap \partial \Omega\implies k = \begin{cases}
		-1, & \text{when $\x$ is on the bottom side of $\Omega$}\\
		-2, & \text{when $\x$ is on the right side of $\Omega$}\\
		-3, & \text{when $\x$ is on the top side of $\Omega$}\\
		-4, & \text{when $\x$ is on the left side of $\Omega$.}\\
	\end{cases}
\end{split}\end{equation}

Below, we give a high level description of the algorithms we used to compute
$\partial A(Y)$, $\mu(A(Y))$, and so forth, and then at the end we give a descriptions
of the algorithms we used in the form of easily replicable pseudo-codes.
Although we implemented our algorithms in {\tt Matlab}, by providing these
pseudo-codes in a language-independent manner, it will be possible for
others to implement the techniques in the environment of their
choosing.  

\subsection{Boundary Tracing Technique}\label{BoundTrace}
Here we describe how we find the boundary of $A(Y)$.  Our approach below
is very different from that used in \cite{LucaJD}.  Namely, in \cite{LucaJD},
for a given set of weights, the boundary of $A(Y)$ was found by a 
subdivision algorithm, whereby the boundary itself was covered by (smaller
and smaller) squares up to a desired resolution.  As explained in \cite{LucaJD},
this approach was inherently limited to a resolution of the boundary of
${\mathcal{O}}(\sqrt{{\mathtt{eps}}})$ accuracy, but did not need to distinguish
a-priori the different smooth arcs making up the boundary, which we instead
need to do, see below.  At the same time, our present technique is able to
find the boundary to arbitrary accuracy (in principle).

\noindent
\begin{minipage}[l]{.40\textwidth}
\begin{rem}\label{BoundTraceRem}
We have to track the boundary of $A(Y)$, by making sure we also
identify where changes can occur.  In essence, 
the basic task is encoded in the figure on the right, where we
identify the smooth arcs and the breakpoints making up $\partial A(Y)$.
\end{rem}
\end{minipage}\hfill
\begin{minipage}[c]{.60\textwidth}
	\hskip2cm
	\centering
	\includegraphics[width=0.5\linewidth]{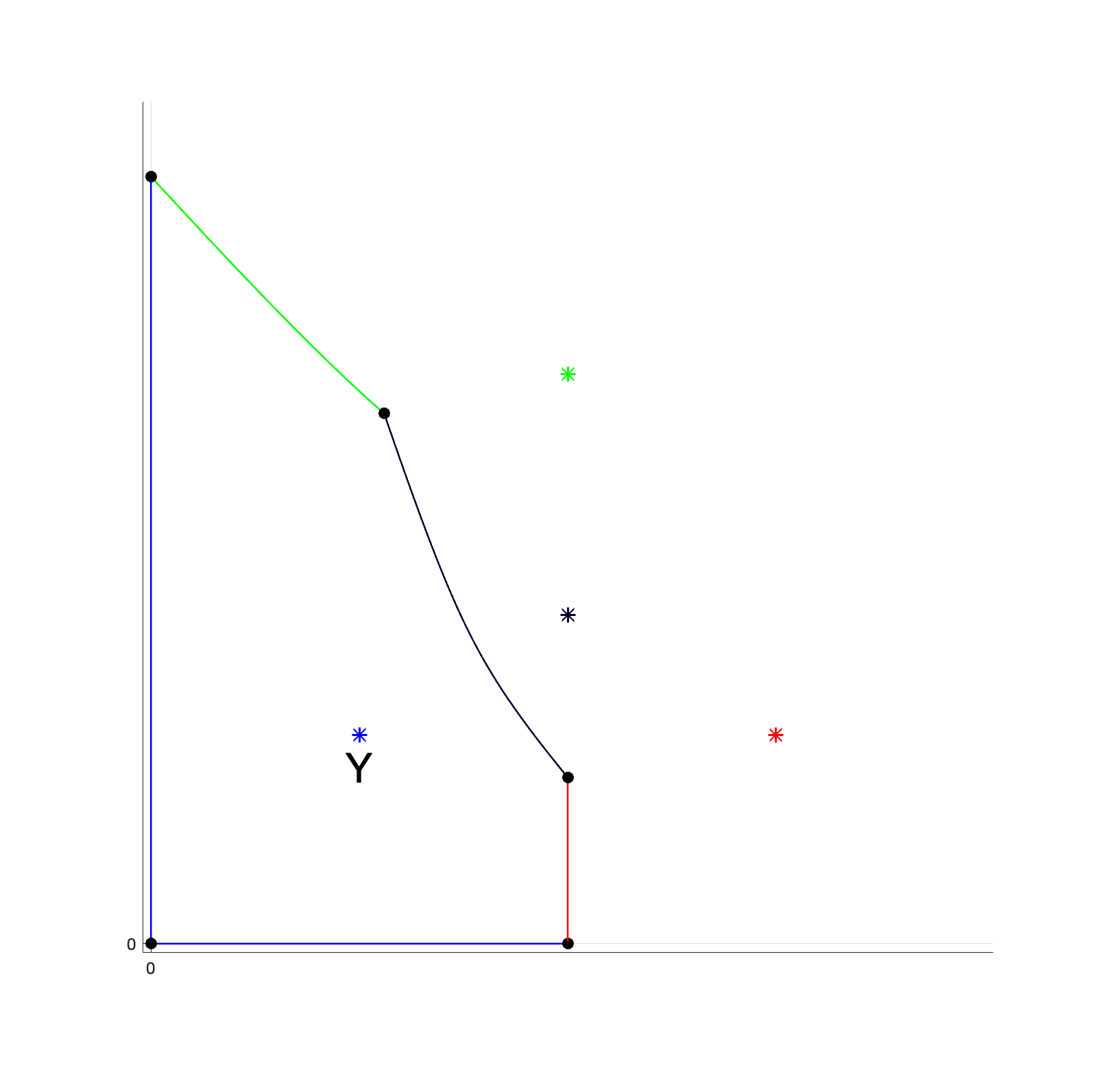}
\end{minipage}

%
%
%

\subsubsection{Shooting Step}\label{Shooting}
The goal of this step is to find a point on $\partial A(Y)$, 
for a given value $\theta_0$.  That is, we will identify
the value $r_0$, and the index $k$ in \eqref{index}, such that
\begin{align*}
	\x = Y + r_0\begin{pmatrix}\cos{(\theta_0)}\\\sin{(\theta_0)}\end{pmatrix}
	\in \partial A(Y)\ .
\end{align*}
The idea is simple: we move along the ray in the direction of the angle 
$\theta_0$ until we hit the boundary. 
Algorithm \ref{Alg_Shooting}, \textit{Shooting},
shows implementation details of this technique.
Algorithm \ref{Alg_BisectR}, \textit{BisectR}, finds the root of the boundary 
function $F(r,\theta_0)$ given in \eqref{Frtheta} below 
using the bisection method on an interval $[r_a, r_b]$ until 
$\frac{r_b-r_a}{2}$ is sufficiently small.  It is assumed/enforced
that $F(r_a,\theta_0)<0$ and $F(r_b,\theta_0)>0$.
Finally, Algorithm \ref{Alg_Bound}, \textit{Bound},
identifies the value on $\partial \Omega$ intersected by the ray from $Y$ in the 
direction of the angle $\theta_0$.
\medskip

Although with the shooting technique of Section \ref{Shooting}
we can find, for any value $\theta$, the value of $r$ on $\partial A(Y)$,
and the index $k$, it is more efficient to use a curve continuation technique 
when the index $k$ is not changing (see Section \ref{breakpoint} on this
last issue).  Our technique is a classical predictor-corrector approach,
with tangent predictor and Newton corrector, described next.

\subsubsection{Predictor-Corrector Method}\label{pred-corr}
Here, we exploit the fact that the value of $r$ on the boundary of $\partial A(Y)$
is a smooth function of $\theta$, and a solution of $F(r(\theta),\theta)=0$
for the function $F(r,\theta)$ below:
\begin{equation}\label{Frtheta}
F(r,\theta) = ||r\begin{pmatrix}\cos{(\theta)}\\\sin{(\theta)}\end{pmatrix}||_p - ||r \begin{pmatrix}\cos{(\theta)}\\\sin{(\theta)}\end{pmatrix}+ Y-\y_j||_p + w_j - W .
\end{equation}
Therefore, a simple application of the implicit function theorem gives that
$$r_\theta=-\frac{F_\theta}{F_r}\ ,$$
and this allows for an immediate computation of a tangent predictor, followed
by a Newton corrector. 

Namely, given $\theta_0$, $r(\theta_0)$, and an angular increment $\Delta$,
we compute $r(\theta+\Delta)$ by tangent predictor and Newton corrector
using $F_r$ and $F_\theta$ below:
\begin{align*}
	&F_r(r,\theta) = \frac{1}{p}||r \begin{pmatrix}\cos{(\theta)}\\\sin{(\theta)}\end{pmatrix}||_p^{1-p}\bigg{[}\frac{\partial}{\partial r}\big{|}r\cos(\theta)\big{|}^{p}+\frac{\partial}{\partial r}\big{|}r\sin(\theta)\big{|}^{p}\bigg{]}-\\
	&\hspace{3em}-\frac{1}{p}||Y - \y_j + r\begin{pmatrix}\cos{(\theta)}\\\sin{(\theta)}\end{pmatrix}||_p^{1-p}\bigg{[}\frac{\partial}{\partial r}\big{|}Y_1 - \y_{j,1}+r\cos(\theta)\big{|}^{p}+\frac{\partial}{\partial r}\big{|}Y_2 - \y_{j,2}+r\sin(\theta)\big{|}^{p}\bigg{]},\\
	&F_{\theta}(r,\theta) = \frac{1}{p}||r\begin{pmatrix}\cos{(\theta)}\\\sin{(\theta)}\end{pmatrix}||_p^{1-p}\bigg{[}\frac{\partial}{\partial\theta}\big{|}r\cos(\theta)\big{|}^{p}+\frac{\partial}{\partial\theta}\big{|}r\sin(\theta)\big{|}^{p}\bigg{]}-\\
	&\hspace{3em}-\frac{1}{p}||Y - \y_j + r \begin{pmatrix}\cos{(\theta)}\\\sin{(\theta)}\end{pmatrix}||_p^{1-p}\bigg{[}\frac{\partial}{\partial\theta}\big{|}Y_1 - \y_{j,1}+r\cos(\theta)\big{|}^{p}+\frac{\partial}{\partial\theta}\big{|}Y_2 - \y_{j,2}+r\sin(\theta)\big{|}^{p}\bigg{]}.
\end{align*}
Algorithm \ref{Alg_PredCorr}, \textit{PredCorr},
gives implementation details of the PC method. 
In addition to the value of $r_1 = r(\theta_0 + \Delta)$, this algorithm outputs 
also a new value for the angular increment $\Delta$, 
which reflects how close the approximation given by the tangent predictor was 
to the final value $r_1$.

\begin{rem}
We performed extensive experiments also with the trivial, the secant, and
the parabolic predictor.  But, in our experiments, the tangent predictor
outperformed these alternatives, not only in accuracy, but also in
overall computational time.
\end{rem}

Through use of the techniques of Sections \ref{Shooting} and \ref{pred-corr},
we compute smooth arcs of $\partial A(Y)$.  To guarantee that these are
indeed smooth arcs, we need to locate points where there is a change.

\subsubsection{Breakpoint Location}\label{breakpoint}
During the curve continuation process, we always monitor the index $k$
in \eqref{index}.  When this index changes 
in the interval $[\theta_0, \theta_0 + \Delta]$,
this means there is a {\emph{breakpoint}}, i.e. there is a value
$\theta\in[\theta_0, \theta_0 + \Delta]$ such that boundary is changing from
one branch to another. There are three things that can happen: (i) the boundary
changes from being portion of a curve to a portion of another curve
(i.e., the neighboring cell is changing  from $A(\y_j)$ to another cell $A(\y_i)$);
(ii) the boundary is changing from being
a portion of the curve between two/more cells to being a portion of the 
boundary of $\Omega$, or vice versa; and (iii) the boundary changes from
one smooth branch of $\partial \Omega$ to another branch of $\partial \Omega$
(presently, this means that two sides of $\partial \Omega$ are meeting at
a corner).  The goal is to identify the breakpoint in these three cases. 

In the first situation, we use bisection method on the interval 
$[\theta_0, \theta_0 + \Delta]$ (see Algorithm \ref{Alg_BisectTheta},
\textit{BisectTheta}). 
In the second situation, we use a standard root finding technique. 
In fact, since 
we can express $r$ in terms of  $\theta$,
we used the {\tt MATLAB} function {\tt fzero} to find the root,
 in the interval $[\theta_0, \theta_0 + \Delta]$, of the function
$F(r(\theta),\theta)$ in \eqref{Frtheta}; see Algorithm \ref{Alg_CurveBound},
\textit{CurveBound}.
Finally, for the last situation, when the breakpoint is at the intersection of two
branches of $\partial \Omega$, using the relation between indices it is
easy to identify exactly the intersection point of the two branches
(in our case, the corner of the square). See Algorithm \ref{Alg_BoundBound},
\textit{BoundBound}.

\subsection{2-d Integral Computation}\label{2dInt}
The other main task is to compute the integral, $\mu(A(Y))$.   This step is
common to all method that want to minimize the functional $\Phi(w)$ in
\eqref{ToMinimize} (see also \eqref{Gradient}), whether or not by using
Newton's method.  The approach taken in \cite{kitagawa2019convergence} and
\cite{de2019differentiation}, where Newton's method is used for
$c(\x,\y)=\|\x-\y\|_2^2$, is not detailed in their works, but it is most likely based on
triangulation of the cell $A(Y)$ in order to take advantage of the polygonal shape
of the latter.  In \cite{Hartmann2020SemidiscreteOT}, instead, where a quasi-Newton
approach is used for $c(\x,\y)=\|\x-\y\|_2$, the authors adopt a quadrature based
on subdivision of the given domain (the square $\Omega$) in a uniform 
way and approximate all integrals on this uniform grid.  No adaptivity,
error estimation, nor error control, is given in 
\cite{Hartmann2020SemidiscreteOT}.  
In our algorithm, instead, we will exploit the parametrization in $\theta$
expressed in
\eqref{PolarArea1}, in particular, points inside $A(Y)$ along a ray from $Y$
are parametrized by the distance from $Y$ and the angular direction $\theta$ 
of the ray.   The main advantages of our method are that it will provide
error estimates, work in a fully adaptive fashion, and be much more
efficient and accurate than competing techniques.  The main drawback
is the need to have a sufficiently smooth density $\rho$, $\cont^4$ to
be precise, since we use a composite Simpson's rule.
Conceivably, there may be situation when this is a stringent
requirement and alternatives will need to be adopted, either by using a
quadrature rule requiring less smoothness\footnote{We also
implemented a version of composite trapezoidal rule, but for smooth
$\rho$ it was much less efficient than Simpson's.}, or by accepting
answers which are accurate to just a few digits.

\smallskip
\noindent
\begin{minipage}[l]{.42\textwidth}
Consider a smooth arc of the boundary $\partial A(Y)$, comprised
between two breakpoints, say at the angles $\alpha$ and $\beta$, 
and let $r(\theta)$
be the smooth function expressing the boundary curve.  Then, we have that
the integral giving $\mu(A(Y))$ restricted to the values of $x$ falling in the sector
determined by $(\alpha, r(\alpha))$ and $(\beta, r(\beta))$, and center $Y$,
is given by (see the figure on the right)
$$\int_\alpha^\beta \int_0^{r(\theta)} d\mu=\int_\alpha^\beta \int_0^{r(\theta)}
\rho(\x)d\x =\int_\alpha^\beta \int_0^{r(\theta)} \rho(D,\theta)D d Dd\theta\ .$$
\end{minipage}\hfill
\begin{minipage}[c]{.58\textwidth}
	\vskip-1cm
	\hskip2cm
	\includegraphics[width=0.55\textwidth]{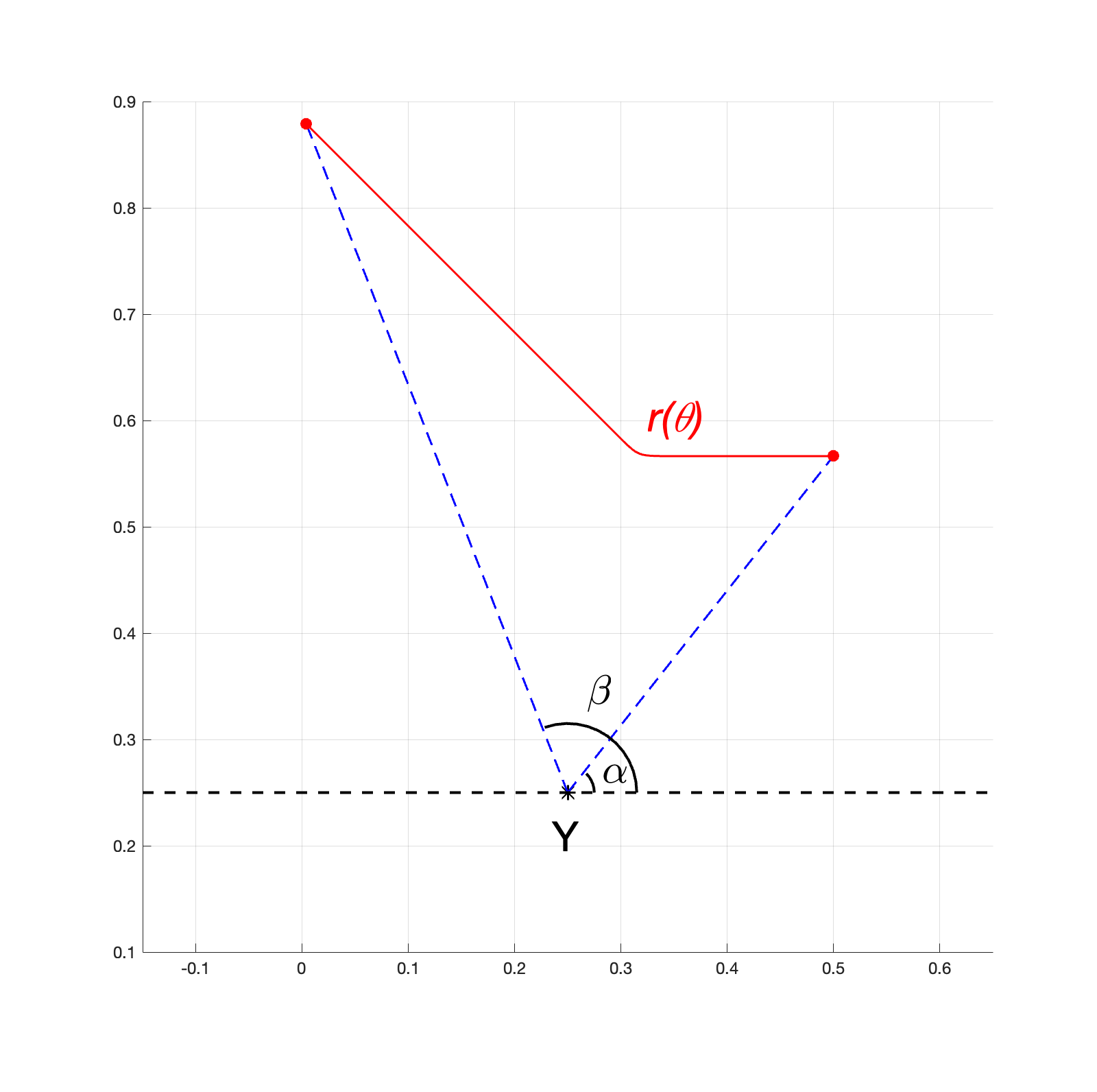}
\end{minipage}
\medskip


As a consequence, for $\mu(A(Y))$, we have
\begin{equation}\label{PolarArea2}\begin{split}
	\mu(A(Y)) & = \int_{A(Y)}d\mu = \int_{A(Y)}\rho(\x)d\x = 
	\sum_{l=1}^{N_l}\int_{\theta_{l,0}}^{\theta_{l,1}}\int_0^{r(\theta)}
	\rho(D,\theta)DdDd\theta \\
	& = \sum_{l=1}^{N_l}\int_{\theta_{l,0}}^{\theta_{l,1}}F(\theta) d\theta\ ,\quad\text{with}
	\quad F(\theta)= \int_0^{r(\theta)}\rho(D,\theta)DdD,
\end{split}\end{equation}
where $N_l$ is the number of breakpoints and $\theta_{l,0}$ and $\theta_{l,1}$ 
are the angular values corresponding to consecutive (in $\theta$) breakpoints.
As a result, the integration is always performed along smooth sections of the 
boundary $\partial A(Y)$ and it involves 1-d integrals (in nested fashion).
Because of this fact, 
to evaluate each integral we implemented an ``adaptive composite Simpson's rule ''
with error control, 
see Algorithm \ref{Alg_AdaptSimpson}, \textit{AdaptSimpson}, and
Algorithm \ref{Alg_CompSimpson}, \textit{CompSimpson}.
Recall that, on an interval $\theta \in [\alpha, \beta]$, where $r(\theta)$
is smooth, the basic composite Simpson's rule  with $n$ subdivisions reads
$$\int_{\alpha}^\beta F(\theta) d\theta =  \frac{h}{3}\bigg{[}F(t_0) + 4\sum_{i=1}^{n/2}F(t_{2i-1}) + 
2\sum_{i=1}^{n/2-1}F(t_{2i})+F(t_n)\bigg{]}+E(F)\ ,
$$
where $h = \frac{\beta-\alpha}{n}$ and $t_i = \alpha + ih$. When $n = 2k$ for some 
$k\in\mathbb{N}$, the error term $E(F)$ has the form:
$$
E(F) = -\frac{h^4}{180}(\beta-\alpha)F^{(4)}(\xi)\ ,
$$
justifying why we are requiring $\rho$ to be $\cont^4$.  In the nested form of
\eqref{PolarArea2}, the function $F(\theta)$ is of course approximated by
adaptive composite Simpson's rule as well.

\begin{rem}\label{RemSimpson}
Adaptivity with error control turned out to be essential in order to resolve
a lot of problems, since the boundary of the Laguerre cells often have
large, localized, curvature values, and local refinement turned out
to be appropriate.  Of course, the caveat to use an accurate
integral approximation, as we did, is that we are asking that
the density $\rho$ is $\cont^4$.
See Example \ref{DifferentSources}-(D) for a case where it is
needed to smooth out a given non-smooth density. 
\end{rem}

%
%

\section{Implementation of Newton's Method and Conditioning of the OT 
problem}\label{Newton}
In the context of semi-discrete optimal transport problem, both 
\cite{de2019differentiation} and \cite{kitagawa2019convergence}
give implementation and testing of 
Newton's method when the cost function is the 2-norm square;
of course, in this case the effort to compute Laguerre cells is greatly
simplified since they happen to be polygons.
We are not aware of any effort for the cost functions we considered
in this work, in particular for cost given by a $p$-norm, $1<p<\infty$.
Indeed, even in the important case of the $2$-norm, 
Hartmann and Schumacher in 
\cite{Hartmann2020SemidiscreteOT} used a quasi-Newton
minimization technique 
(LBFGMS) and, referring to the expression of the Hessian in \eqref{Hessian},
they remarked
 ``{\sl  It remains to be seen, however, if the advantage from using the Hessian 
rather than performing a quasi Newton method outweighs the considerably 
higher computational cost due to computing  the above integrals numerically
 ... }''.  The main computational goal of ours in this work is
to show that Newton's method can be implemented very efficiently, 
and we will see in Section \ref{CompExamp} that is quite a bit more effective
than quasi-Newton approaches.  In fact, computation of the Hessian itself
using \eqref{Hessian} turns out to be inexpensive (often much less 
expensive than computing the gradient), ultimately
because it requires 1-d integrals to be approximated, rather than 2-d integrals;
this is particularly apparent in the increase in cost when approximating the
Hessian by divided differences rather than by using the analytic expression
\eqref{Hessian}, though of course the divided difference approximations
lead to just as accurate end results.  To witness, see Table \ref{Tab_Hesscomp}
below, which refers to Examples (E1)-(E4) of Section \ref{CompExamp}; 
in Table \ref{Tab_Hesscomp}, ``Forward Difference'' refers to the 
approximate Hessian whose columns are obtained as standard difference 
quotient: $\frac{\nabla \Phi(w+e_ih)-\nabla \Phi(w)}{h}$, $h=\sqrt{\mathtt{eps}}$.
``Centered Difference'' refers to the approximate Hessian where
each column is obtained as centered difference quotient:
$\frac{\nabla \Phi(w+e_ih)-\nabla \Phi(w-e_ih)}{2h}$, $h=({\mathtt{eps}})^{1/3}$.
When using these divided difference approximation, we
explicitly enforce symmetry and the row-sum condition in \eqref{Hessian}.

\begin{table}[ht]
	\centering\scalebox{.75}{\renewcommand*{\arraystretch}{2}
		\begin{tabular}{||c||c|c||c|c|c||c|c|c||}\hline\hline
			& \multicolumn{2}{|c||}{\textbf{Analytic \eqref{Hessian}}} & \multicolumn{3}{|c||}{\textbf{Forward Difference}} & \multicolumn{3}{|c||}{\textbf{Centered Difference}}\\\hline\hline
			Example & \textbf{Iterations} & \textbf{Time } & \textbf{Iterations} & \textbf{Time} & \textbf{Difference} & \textbf{Iterations} & \textbf{Time} & \textbf{Difference} \\\hline\hline
			\eqref{E_1} & 2 & 6.2872 & 2 & 9.5086 & $2.5724*10^{-9}$ & 2 & 13.268 & $1.1513*10^{-10}$\\\hline
			\eqref{E_2} & 6 & 38.559 & 6 & 66.805 & $5.3836*10^{-8}$ & 6 & 96.908 & $2.5479*10^{-8}$\\\hline
			\eqref{E_3} & 2 & 8.8570 & 2 & 17.115 & $2.6139*10^{-8}$ & 2 & 26.294 & $2.4851*10^{-10}$\\\hline
			\eqref{E_4} & 3 & 23.908 & 3 & 61.79 & $1.2420*10^{-7}$ & 3 & 109.24 & $1.4503*10^{-7}$\\\hline
			\hline
	\end{tabular}}
	\caption{Iterations refer to Newton's steps, Difference is the $\infty$-norm of 
	the difference between analytic and divided differences 
	Hessians.}\label{Tab_Hesscomp}
\end{table}

As usual, an implementation of Newton's method requires three main ingredients:
forming the Hessian, solving for the updates and iterate, and
of course providing a good initial guess.  We look at these issues
next.

\subsection{Hessian Computation}\label{SubS_HessianComp}
To compute the Hessian, we exploit the parametrization
in terms of the angular variable $\theta$.  Namely, looking at
the expression of $\frac{\partial^2\Phi}{\partial w_i\partial w_j}$ in
\eqref{Hessian}, we note that $x$ is restricted to the boundary $A_{ij}$
between $A(\y_i)$ and $A(\y_j)$, and thus can be given as
\begin{equation*}
	\x(r(\theta),\theta) = \y_i + r(\theta)\begin{pmatrix}\cos{(\theta)}\\\sin{(\theta)}\end{pmatrix}
\end{equation*}
for an appropriate range of $\theta$-values, and 
where $r(\theta)$ refers to the associated part of the boundary curve,
that is the distance from $\y_i$ along the boundary curve.
Thus, the term $\frac{\partial^2\Phi}{\partial w_i\partial w_j}$ can be
rewritten as an arc-length integral:
\begin{equation}\begin{split}\label{HessianPolar}
	& \frac{\partial^2\Phi}{\partial w_i\partial w_j} = 
	\sum_{\text{smooth arcs}} \int_{\theta_0}^{\theta_1} 
	\frac{-\rho(r(\theta),\theta)}{||\hspace{0.25em}\nabla_{\x} c(\x(r(\theta),\theta),\y_i) - \nabla_{\x} 
		c(\x(r(\theta),\theta),\y_j)\hspace{0.25em}||} 
		\sqrt{r^2(\theta) + (r'(\theta)^2)}d\theta \\
& \hspace{1em}j\neq i; 
\hspace{0.5em} i,j 	=1,2,\dots,N,
\end{split}\end{equation}
where $[\theta_0, \theta_1]$ is an interval corresponding to a smooth portion of
the boundary $A(\y_i)\cap A(\y_j)$ and the summation is done along all
smooth arcs making up $A_{ij}$ (recall the Figures in Remarks \ref{2Sections}
and \ref{BoundTraceRem}).
Note that, just like in Section \ref{pred-corr}, $r'(\theta)$ can be obtained
using the Implicit Function theorem on the boundary equation
$F(r(\theta),\theta) = 0$, where $F$ is given in \eqref{Frtheta}, that is:
$r_\theta=-\frac{F_\theta}{F_r}$.
The 1-d integrals in \eqref{HessianPolar} are again computed by the composite adaptive
Simpson's rule.
	
%

Of course, due to the symmetric structure of the Hessian matrix, and the
special form of its diagonal (see \eqref{Hessian}), only the 
strictly upper-triangular part of $H$ needs to computed.
Finally, the Hessian will generally be a  very sparse matrix (which is especially 
relevant for large number of target points), since if 
$A(\y_i)\cap A(\y_j) = \emptyset$, then $\frac{\partial^2\Phi}{\partial w_i\partial w_j}=0$.

\subsection{Taking a Newton step: solving a system with $H$}
\label{HessianSolve}
Assume we have a (guess) $w^0$, satisfying the feasibility condition
\eqref{feas-w-Thm} and also such that its component add up to $0$.

Now, a typical Newton step would read
\begin{equation*}
\text{Given}\,\ w^0, \,\  \text{form} \,\ H(w^0)\,\ 
\text{solve for $s$}\,\ H(w^0)s=-\nabla \Phi(w^0),\,\ \text{update}
\,\ w^1=w^0+s\ .
\end{equation*}
However, there are at least two obvious issues that prevent implementing this
step as naively as above.  

For one thing, $H(w^0)$ is singular, see Theorem \ref{HessProp}.  
As said before, this is a reflection of the fact that every feasible vector
$w$ can be shifted by a multiple of $e$ without changing the
associated Laguerre cells $A(y_i)$'s.   Singularity
of the Hessian has of course been recognized by those working with Newton's
method, though dealt with in different ways; for example, in 
\cite{kitagawa2019convergence} the authors use the pseudo-inverse of $H$,
and in \cite{santambrogio2015optimal} an approach is adopted by holding
at $0$ one component of $w$ and eliminating the row and column of
that index from $H$.
However, we favor an approach that exploits that $H$ has
a 1-dimensional nullspace spanned by the vector 
$e=\tiny{\begin{bmatrix} 1 \\ \vdots \\ 1 \end{bmatrix}}\in \R^N$.
So, we select $w^0$ so that its components add up to $0$, 
and we will keep this property during Newton iteration.

To achieve the above, our approach is to seek the update $s\in \R^N$ by ``solving'' 
$H(w^0)s=\nabla \Phi(w^0)$ restricting the solution for $s$ in 
the (orthogonal) complement to $e$.  
In formulas, 
we form once and for all the orthogonal matrix $U$ below
(which is independent of $w$)
\begin{equation}\begin{split}\label{OrthComp}
& U=\begin{bmatrix} q & Q \end{bmatrix}\,\ \text{with}\,\ 
q=\frac{1}{\sqrt{N}}e ,\,\ 
Q= \{I_N - uu^T\}_{\{2,3,\dots,N\}}\in\mathbb{R}^{N\times N-1},\\
& \text{and where} \,\  u = \frac{1}{\sqrt{\sqrt{N}-1}}
(\frac{1}{\sqrt{N}}e - e_1)\in\mathbb{R}^N.
\end{split}\end{equation}
Then, letting $H^0:=H(w^0)$, and $b^0=-\nabla \Phi(w^0)$, 
we rewrite the problem as
$$H^0 s = b^0 \ \to \ 
(U^T H^0U) (U^T s)\ = \ U^T b^0 \ \to \ \begin{bmatrix}
0 & 0 \\ 0 & Q^TH^0Q \end{bmatrix} U^Ts=\begin{bmatrix}
q^T b \\ Q^Tb \end{bmatrix} ,
$$ 
since $U^T H(w^0)=\tiny{\begin{bmatrix} 0 \\ \vdots \\ 0 
\end{bmatrix}}$ and $H^0=(H^0)^T$.  Therefore, we must have
$q^Tb=0$ and thus
$$
\begin{bmatrix}	0 & 0 \\ 0 & Q^TH^0Q \end{bmatrix} 
\begin{bmatrix}	q^T s \\ Q^Ts \end{bmatrix}=
\begin{bmatrix}	0 \\ - Q^T\nabla \Phi(w^0)  \end{bmatrix}$$
from which multiplying both sides by the invertible matrix
$\begin{bmatrix}	1 & 0 \\ 0 & (Q^TH^0Q)^{-1} \end{bmatrix}$, we
obtain $\begin{bmatrix} q^T \\ Q \end{bmatrix} s=
\begin{bmatrix}  0 \\ -[Q^TH^0Q]^{-1}Q^T \nabla \Phi(w^0)
\end{bmatrix}$ and eventually our update $s$ is given by
\begin{equation}\label{NewtUpdate}
	s=-(H^0)^{I}\nabla \Phi(w^0), \,\ \text{where}\,\ 
	(H^0)^{I}=Q[Q^TH^0Q]^{-1}Q^T .
\end{equation}
Since $q^Ts=0$, then $e^Ts = 0$, and so the components of $w^1=w^0+s$
indeed will add up to $0$.

\subsubsection{Conditioning of semi-discrete OT problems}\label{condOT}
The second concern is that the updated shift values needs to 
be feasible,  in other words, see Theorem \ref{feas-w-Thm}, they
must satisfy the condition
\begin{equation}\label{NeedFeas}
\min_{i,j} c(\y_i,\y_j)-|w^1_i - w^1_j|> 0\ ,\,\ \forall i,j=[1,N],\, i\ne j\ .
\end{equation}
If this is satisfied, then the step is accepted, otherwise we reduce
by a factor of $1/2$ the update $s$ until the criterion \eqref{NeedFeas}
is satisfied (similar to damped Newton method, see below)
$$s\leftarrow s/2\ ,\,\ w^1\leftarrow w^0+s \ .$$
\begin{lem}
With feasible $w^0$, the above updating strategy 
will eventually lead to a feasible $w^1$.
\end{lem}
\begin{proof}
We rewrite the feasibility requirement in matrix form as the component-wise
inequality $|Aw|<b$, where the matrix $A\in \R^{N(N-1)/2\times N}$ and
the vector $b\in \R^{N(N-1)/2}$ are given by
$$A=\tiny{\begin{bmatrix} 1 & -1 & 0 & 0 & \cdots & 0 \\
1 & 0 & -1 & 0 & \cdots & 0 \\
1 & 0 & 0 & -1 & \cdots & 0 \\
\vdots & \vdots & \vdots & \vdots & \vdots & \vdots \\
1 & 0 & 0 & 0 & \cdots & -1\\
0 & 1 & -1 & 0 & \cdots & 0\\
0 & 1 & 0 & -1 & \cdots & 0\\
\vdots&\vdots&\vdots&\vdots&\vdots&\vdots
\end{bmatrix}}, \hspace{1em}
b = \tiny{\begin{pmatrix}
c(\y_1, \y_2)\\
c(\y_1, \y_3)\\
c(\y_1, \y_4)\\
\vdots\\
c(\y_1, \y_N)\\
c(\y_2, \y_3)\\
c(\y_2, \y_4)\\
\vdots
\end{pmatrix}}
$$
According to the strategy above, 
consider $w^{1} = w^0 + \frac{1}{2^k}s$, where 
$s = -H(w^0)^{I}\nabla\Phi(w^0)$, and recall that $w^0$ is feasible
(hence, it satisfies \eqref{NeedFeas}).  We claim that $w^1$ is
guaranteed to be feasible for 
$k = \max{\{0, \lceil\max{(\log_2{|As|} - \log_2|\epsilon|)}\rceil\}}$, where 
$\epsilon = b - |Aw^0|$.  In fact:
\begin{equation*}\begin{split}
	& |A(w^0 + \frac{1}{2^k}s)| =  |Aw^0 + \frac{1}{2^k}As| \leq |Aw^0| + \frac{1}{2^k}|As| = b - \epsilon + \frac{1}{2^k}|As| \quad \text{and so} \\
	& \frac{1}{2^k}|As| - \epsilon<0 \iff |As| < 2^k\epsilon \iff \log_2|As|<k+\log_2\epsilon\iff k> \max{(\log_2{|As|} - \log_2|\epsilon|)}\ 
\end{split}\end{equation*}
as claimed.	
\end{proof}
\begin{rem}\label{FeasRem}
Of course, it is possible that $w^0$, albeit feasible, is far from the solution
of the nonlinear system; in this case, it is possible that the neighborhood
of feasibility of $w^0$ is small and
that to obtain $w^1$ by the above updating strategy
will lead to very tiny steps, and ultimately numerical failure.  If this is the case,
there are effectively two possibilities: (i) we should give an improved initial guess 
$w^0$, or (ii) the problem is inherently ill-conditioned, in that the neighborhood
of feasibility of the exact solution $w$ is just too small, and in finite precision
one cannot obtain arbitrary accuracy.  Below, we elaborate on both of these
aspects.
\end{rem}

Motivated by \ref{NeedFeas} and Remark \ref{FeasRem}, 
we introduce the {\emph{feasibility coefficient}} relative to a given
feasible $w$.  Namely, we let it be the quantity:
\begin{equation}\label{FeasCoeff}
\kappa(w) = \min_{i,j} \left[ 1-\frac{|w_{ij}|}{c(y_i,y_j)}\right] \ , 
\end{equation}
and of course $\kappa(w)$ has to be positive.
In our computations, we always monitor this quantity $\kappa(w)$
(during and at the end of the Newton iterations), and we have 
consistently observed that smaller values of $\kappa(w)$ correspond to more
difficult problems; e.g., see Example \ref{ImpactFeas}.
We also note that $\kappa(w)$ is a scaling-invariant quantity (i.e., it is not
changed by uniformly dilating or restricting the domain $\Omega$, since $w$ rescales
in the same way as the cost), 
and it only depends on the problem itself, that is the location
of the target points, the cost function, and the value of $w$, which ties
together the continuous and discrete probability densities.  
Thus, our proposal is: 

\centerline{\bf Regard $1/\kappa$ as 
condition number of semi-discrete optimal transport problems.}

\subsection{Getting an initial guess}\label{InitialGuess}
We have implemented and experimented with several different strategies
to obtain a good and feasible initial guesses for solving $\nabla \Phi(w)=0$.

\begin{itemize}
\item[(i)] {\sl Trivial initial guess}. This is the simplest strategy, 
we take $w^0 = 0$ as an initial guess. 
This corresponds to a classical Voronoi diagram (proximity problem), 
and we call $V_i$ the resulting value of $\mu(A(\y_i))$ obtained when
using $w=0$.  In spite of its simplicity, this approach often worked
remarkably well, especially in conjunction with damped Newton's
method (see below).
\item[(ii)] {\sl Grid Based}.  This technique is based on the work
\cite{Merigot2013a}, and it is quite similar to the well known Lloyd algorithm.  
First, consider a grid of meshsize  $h=(b-a)/N_h$ to discretize $\Omega$,
so that there are $N_h^2 = \frac{(b-a)^2}{h^2}$ squares on the grid.
Compute the coordinates of the $N_h^2$ centers of these squares.
For given shift values $w$, approximate $\mu(A(\y_i))$ as
\begin{gather*}
	\mu(A(\y_i)) \approx \sum_{l=1}^{M_i}\rho(x_l)h^2\ ,
\end{gather*}
where $x_l$, $l=1,2,\dots,M_i$, are the square centers 
which are contained in the Laguerre cell $A(\y_i)$.
Then, using this approximation we can adjust the values $\mu(A(\y_i))$ 
to satisfy $\mu(A(\y_i)) = \nu_i$ up to some tolerance by changing shift values.
For example, if the approximated $\mu(A(\y_i))$ value is greater (smaller) than 
the required $\nu_i$ value, then we force a decrease (increase) in the value of
$\mu(A(\y_i))$ by increasing (decreasing) $w_k$, $\forall k\neq i$;
we do this by  the same increment for all $k\ne i$. 
Note that the adjustment of $\mu(A(\y_i))$ is done one index at a time. 
In practice, we choose the tolerance to be the area of a square, 
$|\mu(A(\y_i))-\nu_i|<h^2$, and an initial increment to be 
$\sqrt{h}$. Depending on the problem, the desired tolerance might not be 
reachable for all $\mu(A(\y_i))$ values. Then, one can either increase the 
maximum  number of allowed iterations, or decrease the mesh size $h$.
See Algorithm \ref{Alg_GridInitialGuess}, \textit{GridInitialGuess},
and Algorithm \ref{Alg_GridMeasure}, \textit{GridMeasure},
for implementation details. 
\begin{itemize}
The next two approaches are classical homotopy techniques, which
we implemented as commonly done when solving nonlinear
systems (e.g., see \cite{Keller}).
\item[(iii)] {\sl Trivial Homotopy}.
We want to solve $g_1(w(\alpha))=0$, where the components of $g_1$ are
given by
\begin{equation*}
(g_1(w))_i = \mu(A(\y_i)) - [(1-\alpha)\mu(V_i) + \alpha \nu_i], \ \ i=1,\dots, N, \,\
0\le \alpha \le 1\ .
\end{equation*}
At $\alpha = 0$, $w = 0$ is a solution, which corresponds to the Voronoi 
diagram $\{V_i\}$ (see (i)), and at $\alpha=1$ we obtain the problem we are
trying to solve.  We observe that the problem can be viewed, for each 
$\alpha$, as finding $w$ such that
$$\mu(A(\y_i)) = \hat \nu_i \equiv
\alpha \nu_i+(1-\alpha)\mu(V_i) , \ \ i=1,\dots, N,$$
that is we have a semi-discrete problem where the target density is
given by $\sum_i \delta(\y_i)\hat \nu_i$, which --by virtue of
Theorem \ref{Thm_ExistUniq}-- is solvable giving a solution
$w(\alpha)$ (unique up to adding a constant to all
entries)\footnote{This homotopy path is the counterpart to what
was used in \cite{MEYRON201913}, where the homotopy was done to move
from an ``easy'' source density to the desired source density. }.
We use Newton's method to solve from $\alpha=0$ to $1$ choosing the
steps adaptively beginning with $\Delta^0=1/10$ (or $1/100$) and updating
it based on the 
number of Newton iterations $k$ required for convergence at the previous step:
\begin{align*}
\Delta^i = 2^{\frac{4-k}{3}}\Delta^{i-1}, \hspace{1em}\alpha^{i} = \alpha^{i-1} + \Delta^i\ .
\end{align*}
We note that the Hessian resulting from using Newton's method is  always 
the one in \eqref{Hessian}, hence it is always singular along
the homotopy path, with null vector given by $e$.

\item[(iv)] {\sl Non-Singular Homotopy}.
Here, we solve $g_2(w(\alpha))=0$, where the components of $g_2$ are
\begin{equation*}
(g_2(w(\alpha)))_i = (1-\alpha)w + \alpha(\mu(A(\y_i)) - \nu_i), \ i=1,\dots, N\ ,
\,\ 0\le \alpha \le 1\ .
\end{equation*}
At $\alpha=0$, we have the solution $w=0$, and for $\alpha>0$, we can
look at the problem as finding $w$ such that $\sum w_i=0$ and
\begin{equation*}
\mu(A(\y_i))=  \hat \nu_i \equiv \nu_i-\frac{1-\alpha}{\alpha}w_i , \ i=1,\dots, N\ ,
	\,\ 0< \alpha \le 1\ ;
\end{equation*}
however, although $\sum_i \hat\nu_i=1$, we now have no guarantee that
$\hat \nu_i>0$, and in fact we occasionally had difficulties with this
homotopy choice precisely because of this fact.
Otherwise, for as long as $\hat \nu_i>0$, 
we can use Newton's method to solve for $w(\alpha)$ from 
$0$ to $1$, and a main advantage now is that 
the Hessian (Jacobian of $g_2$ with respect to $w$)
is non-singular for $\alpha\in [0,1)$, since it is
$(1-\alpha)I +\alpha H(w(\alpha))$ and $H(\alpha)$ is positive semi-definite
from Theorem \ref{Hessian}.
\end{itemize}
\end{itemize}

On the same four problems (E1)-(E4) of Section \ref{CompExamp}, in
Table \ref{Tab_IGcomp} we show relative performance of the above
strategies to provide an initial guess $w^0$.  For the homotopy
techniques, ``Nsteps'' refer to the number of homotopy steps,
and ``Iterations'' to the total number of Newton's iterations.
Clearly, the ``Grid Based'' approach worked much better, and
for this reason all results we report on in Section \ref{CompExamp}
were obtained by using the
``Grid Based'' approach to provide the initial guess $w^0$.

\begin{table}[ht]
	\centering\scalebox{.75}{\renewcommand*{\arraystretch}{2}
		\begin{tabular}{||c||c|c||c|c||c|c||}\hline\hline
			& \multicolumn{2}{|c||}{\textbf{Grid Based} $h=0.05$} & \multicolumn{2}{|c||}{\textbf{Trivial Homotopy}} & \multicolumn{2}{|c||}{\textbf{Non-Singular Homotopy}}\\\hline\hline
			Example & \textbf{Iterations} & \textbf{Time}& \textbf{Nsteps}, \textbf{Iterations} &\textbf{Time} & \textbf{Nsteps}, \textbf{Iterations} &\textbf{Time}\\\hline\hline
			\eqref{E_1} & 2 & 7.6242 & 6, 16 & 50.767 & 7, 22 & 66.755\\\hline
			\eqref{E_2} & 7 & 40.951 & 7, 24 & 126.39 & 8, 31 & 148.63\\\hline
			\eqref{E_3} & 2 & 9.0191 & 5, 11 & 44.867 & 6, 18 & 68.335\\\hline
			\eqref{E_4} & 3 & 22.317 & 6, 19 & 154.22 & 7, 25 & 186.74\\\hline
			\hline
	\end{tabular}}
	\caption{Different strategies to give an initial guess for Newton's 
	method}\label{Tab_IGcomp}
\end{table}

 A standard modification of Newton's method is the well known ``Damped
 Newton's Approach''.    This can be viewed as a mean by which we improve on
 a feasible initial guess $w^0$ , e.g. $w^0=0$, to bring it closer to the solution
 of $\nabla \Phi(w)=0$.  As a general technique, damped Newton's 
 is routinely used in solving nonlinear
 systems, and was also used in \cite{kitagawa2019convergence} for solving
semi-discrete optimal transport problems for the 2-norm square cost.
Below, we simply give its justification for our case of singular Hessian,
particularly given the way we find the Newton's update.

\subsubsection{Damped Newton Method}\label{damped}
When one has to solve a nonlinear system $F(w)=0$, and the Jacobian at $w^0$
is invertible, then it is well known that the Newton direction
is a direction of descent for the function
$g(w):= \|F(w)\|_2^2=F(w)^TF(w)$,
in the sense that $g(w^0+\alpha s)<g(w^0)$ for some $0<\alpha \le 1$ and
$s$ is the solution of $DF(w^0)s=-F(w^0)$. 
Below, we show that also in our case the Newton direction from
\eqref{NewtUpdate} is an appropriate direction of descent in the subspace
complementary to the null-space of $H$.

\begin{lem}\label{DampedDescent}
If $\nabla \Phi(w)\ne 0$, then 
the direction $s$ given by \eqref{NewtUpdate} is a direction of descent
for the functional given by 
\begin{equation}\label{WhatDecreases}
g(w)= (Q^T \nabla \Phi(w))^T (Q^T\nabla \Phi(w))\ ,
\end{equation}
where $U=\begin{bmatrix} q & Q \end{bmatrix}$ is the orthogonal
matrix defined in \eqref{OrthComp}.  That is, there exists $0<\alpha\le 1$ 
such that $g(w^0+\alpha s)<g(w^0)$.
\end{lem}
\begin{proof}
Expand:
\begin{equation*}\begin{split}
g(w+\alpha s)& =g(w)+\alpha (\nabla g(w))^Ts+ \mathcal{O}(\alpha^2) \\
& =g(w)+2\alpha (\nabla \Phi(w))^TQQ^TH(w)s+ \mathcal{O}(\alpha^2)\ .
\end{split}\end{equation*}
Then, using \eqref{NewtUpdate}, we have
$$g(w+\alpha s)-g(w)=-2\alpha (\nabla \Phi(w))^TQ(Q^TH(w)Q)
(Q^TH(w)Q)^{-1}Q^T\nabla \Phi(w)+ 
\mathcal{O}(\alpha^2)= -2\alpha g(w)+\mathcal{O}(\alpha^2) .$$
Therefore, the claim follows if $g(w)>0$.  Next, we show that $g(w)=0$
only if $\nabla \Phi(w)=0$ and the result will follow.
In fact:
$$g(w)=0\ \longleftrightarrow \ Q^T\nabla \Phi(w)=0 \ \longleftrightarrow \ 
\nabla \Phi(w)=ce\ ,\,\ c\in \R\ ,$$
since the kernel of $Q^T$ is 1-dimensional and spanned by $e$.  Now,
since $\nabla \Phi(w)=\left(\mu(A_i)-\nu_i\right)_{i=1,\dots, N}$, then
$\sum_{i=1}^N (\nabla \Phi(w))_i =0$ and therefore we must have
$c=0$.
\end{proof}
Lemma \ref{DampedDescent} justifies use of a full damped Newton's method
to solve the problem $\nabla \Phi (w)=0$.  And, 
in the end, using damped Newton's method revealed to be the best method
(among those we implemented) to perform the full iteration in order to solve
$\nabla \Phi(w)=0$.  Namely, we always attempt to take the full Newton step, 
and damp it --if needed-- to force a decrease in the functional $g(w)$ in
\eqref{WhatDecreases}. 
Details of our implementation are provided in Algorithm 
\ref{Alg_DampedNewton}, \textit{DampedNewton}.  That said, most of the time
we found no need to damp the Newton step, see Section \ref{CompExamp}.

\section{Computational Examples}\label{CompExamp}

Here we present results of our algorithms.  We pay special attention to
give Examples that can be replicated by other approaches, and give
quantitative results, rather than just pictures.

\renewcommand\theequation{E\arabic{equation}} \setcounter{equation}{0}

All of our computations were performed with {\tt Matlab}.  The key quantities
we monitor are the {\bf Error}, which is $\|\nabla \Phi(w)\|_\infty$ at convergence,
the number of Newton {\bf Iterations} and further the
number of {\bf Damped} steps if required, the execution {\bf Time}, 
and the {\bf Feasibility Coefficient} $\kappa$ as needed.  The key
quantities by which we control accuracy are two tolerance values, one to
control accuracy of the computation of $\mu(A(\y_i))$ (and, therefore of the
gradient $\nabla \Phi$), the other to measure convergence of the Newton's
process, as assessed by either $\|\nabla \Phi\|_\infty$ or the $\infty$-norm
of the Newton's updates.  The default values are $10^{-12}$ for the approximation
of $\mu(A(\y_i))$ and $10^{-8}$ for Newton's convergence.

The next four examples have been used to produce Tables \ref{Tab_Hesscomp}
and \ref{Tab_IGcomp} as well as Table \ref{Tab_Algcomp} below.  They are
fairly easy problems, mostly used for testing purposes.  For all of them,
the domain is $\Omega = [0,1]\times [0,1]$, the cost is
the $2$-norm, $c(\x,\y) = ||\x-\y||_2$, and $\rho$, $\nu$,  and the locations
$\y_i$'s are given below.  See Figure below for their solution.

\begin{gather}\label{E_1}
\rho(x) = 1,\ \nu = \frac{1}{2}\small{\begin{pmatrix}1\\1\end{pmatrix}},\ 
\y = \bigg{\{}\small{\begin{pmatrix}0.125\\0.125\end{pmatrix},\ 
\begin{pmatrix}0.5\\0.5\end{pmatrix}}\bigg{\}},\ .
\end{gather}
\begin{gather}\label{E_2}
	\rho(x) = 4x_1x_2,\ \nu= \frac{1}{4} 
	\small{\begin{pmatrix}1\\\vdots\\1\end{pmatrix}},\
\y = \bigg{\{}\small{\begin{pmatrix}0.25\\0.25\end{pmatrix},\ 
	\begin{pmatrix}0.75\\0.25\end{pmatrix},\ 
	\begin{pmatrix}0.5\\0.25(1+\sqrt{3})\end{pmatrix},\ 
	\begin{pmatrix}0.5\\0.25(1+\frac{\sqrt{3}}{3})\end{pmatrix}}\bigg{\}}.
\end{gather}
\begin{gather}\label{E_3}
\rho(x) = 1,\ \nu = \frac{1}{5}\small{\begin{pmatrix}1\\\vdots\\1\end{pmatrix}},\ 
		\y = \frac{1}{4096}\bigg{\{}\small{\begin{pmatrix}646\\3491\end{pmatrix},\ \begin{pmatrix}3480\\3686\end{pmatrix},\ \begin{pmatrix}1364\\2737\end{pmatrix},\ \begin{pmatrix}609\\857\end{pmatrix},\ \begin{pmatrix}2967\\509\end{pmatrix}}\bigg{\}}.
\end{gather}
\begin{gather}\label{E_4}
\rho(\x) = 1,\ \nu = \frac{1}{8}\small{\begin{pmatrix}1\\ \vdots\\1\end{pmatrix}},\ 
\y = \text{8 random uniform points using seed $\mathtt{rng(2)}$}.
\end{gather}

\begin{figure}[ht]
	\centering
	\begin{subfigure}[b]{0.245\linewidth}
		\centering
		\includegraphics[width=\linewidth]{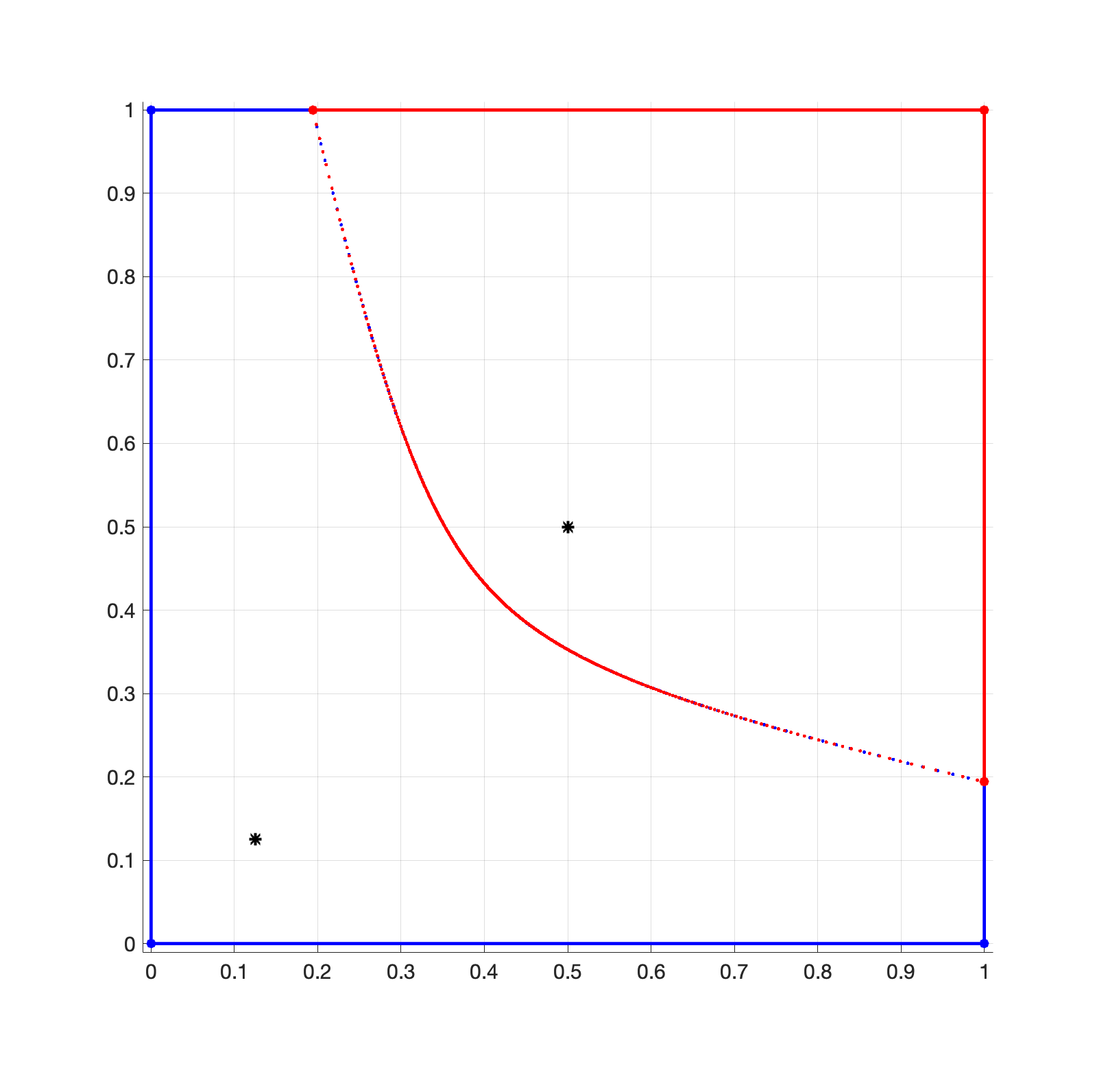}
		\caption{Example \eqref{E_1}}\label{Fig_E_1}
	\end{subfigure}
	\hfill
	\begin{subfigure}[b]{0.245\linewidth}
		\centering
		\includegraphics[width=\linewidth]{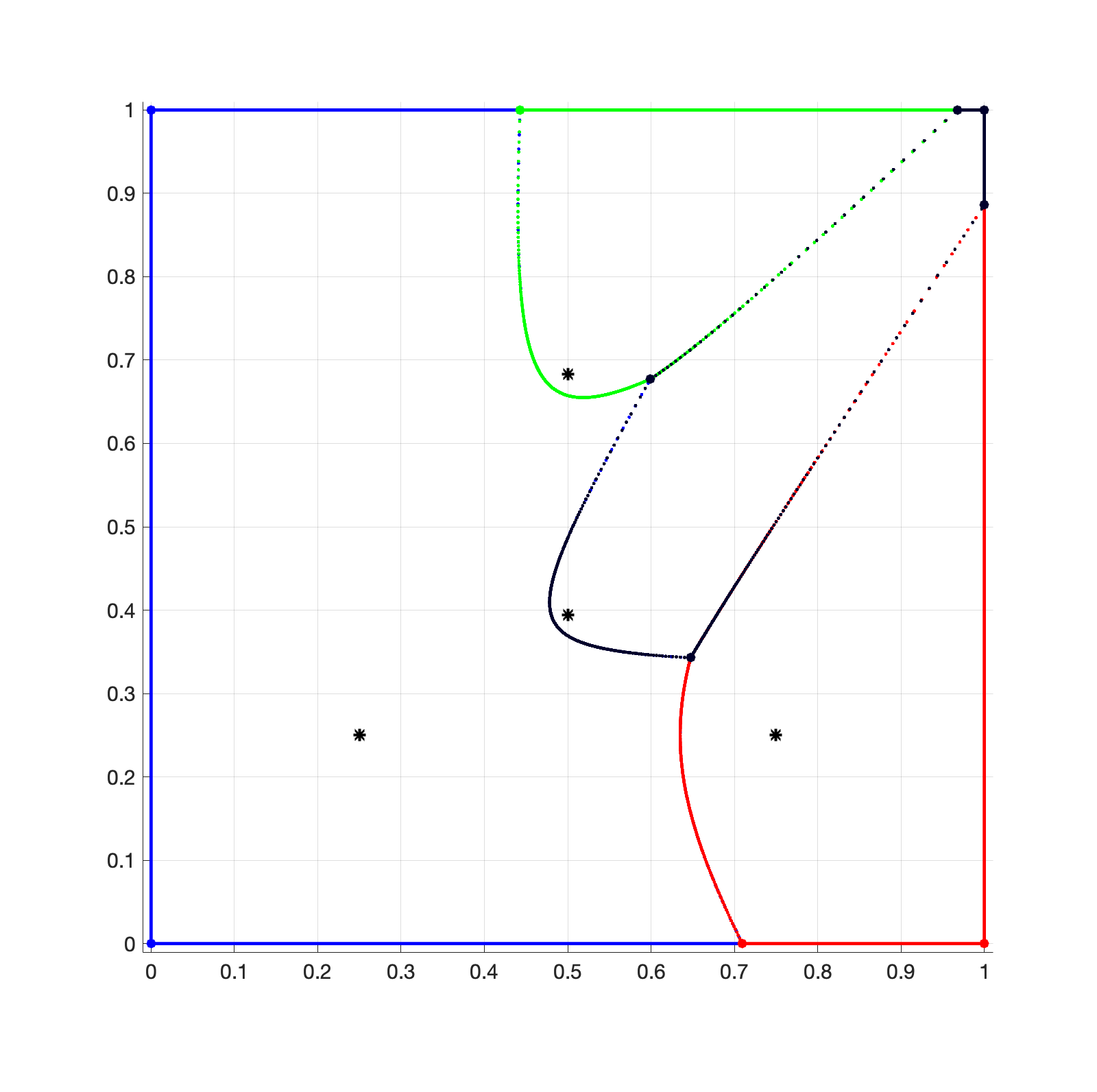}
		\caption{Example \eqref{E_2}}\label{Fig_E_2}
	\end{subfigure}
	\hfill
	\begin{subfigure}[b]{0.245\linewidth}
		\centering
		\includegraphics[width=\linewidth]{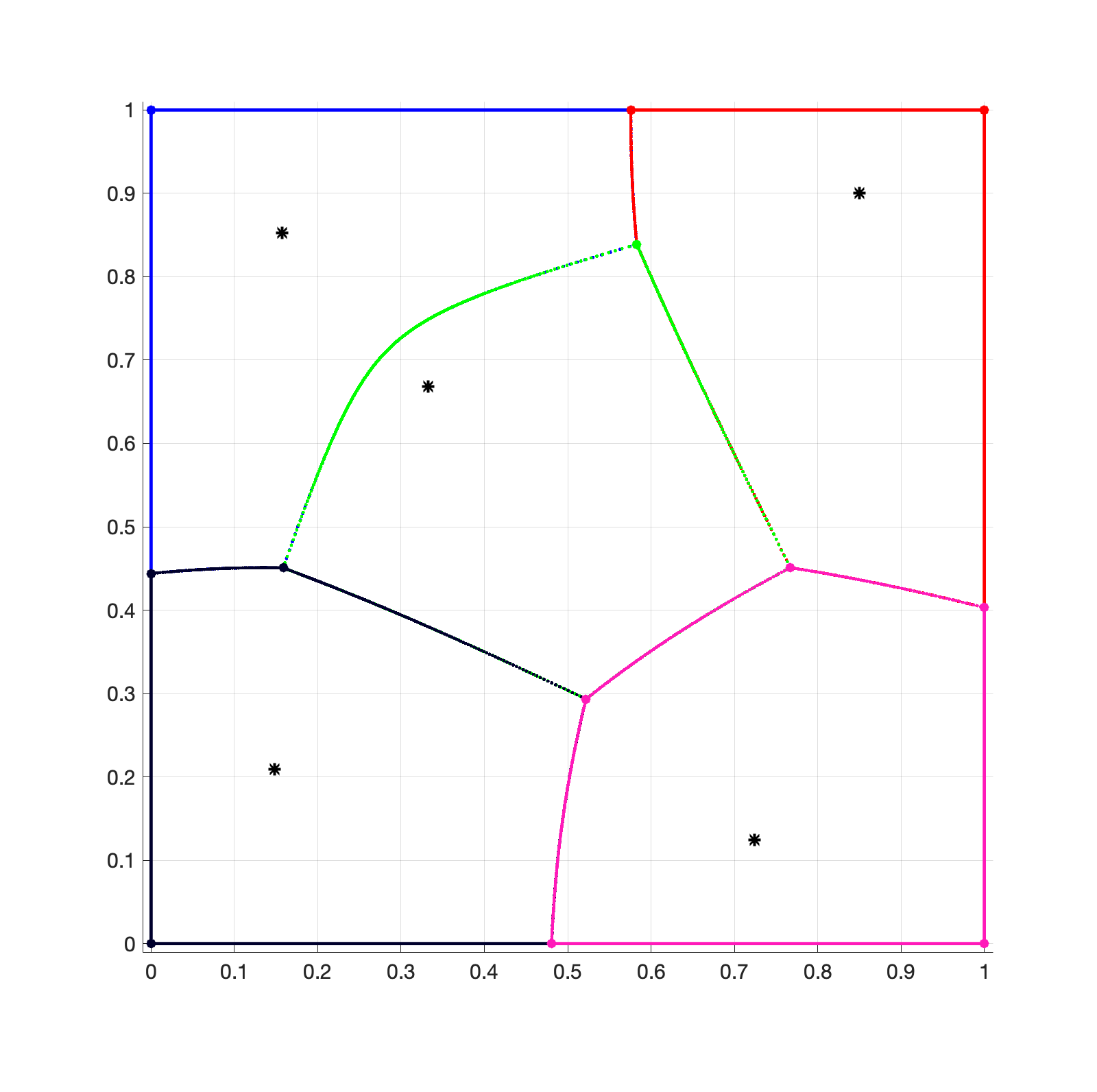}
		\caption{Example \eqref{E_3}}\label{Fig_E_3}
	\end{subfigure}
	\hfill
	\begin{subfigure}[b]{0.245\linewidth}
		\centering
		\includegraphics[width=\linewidth]{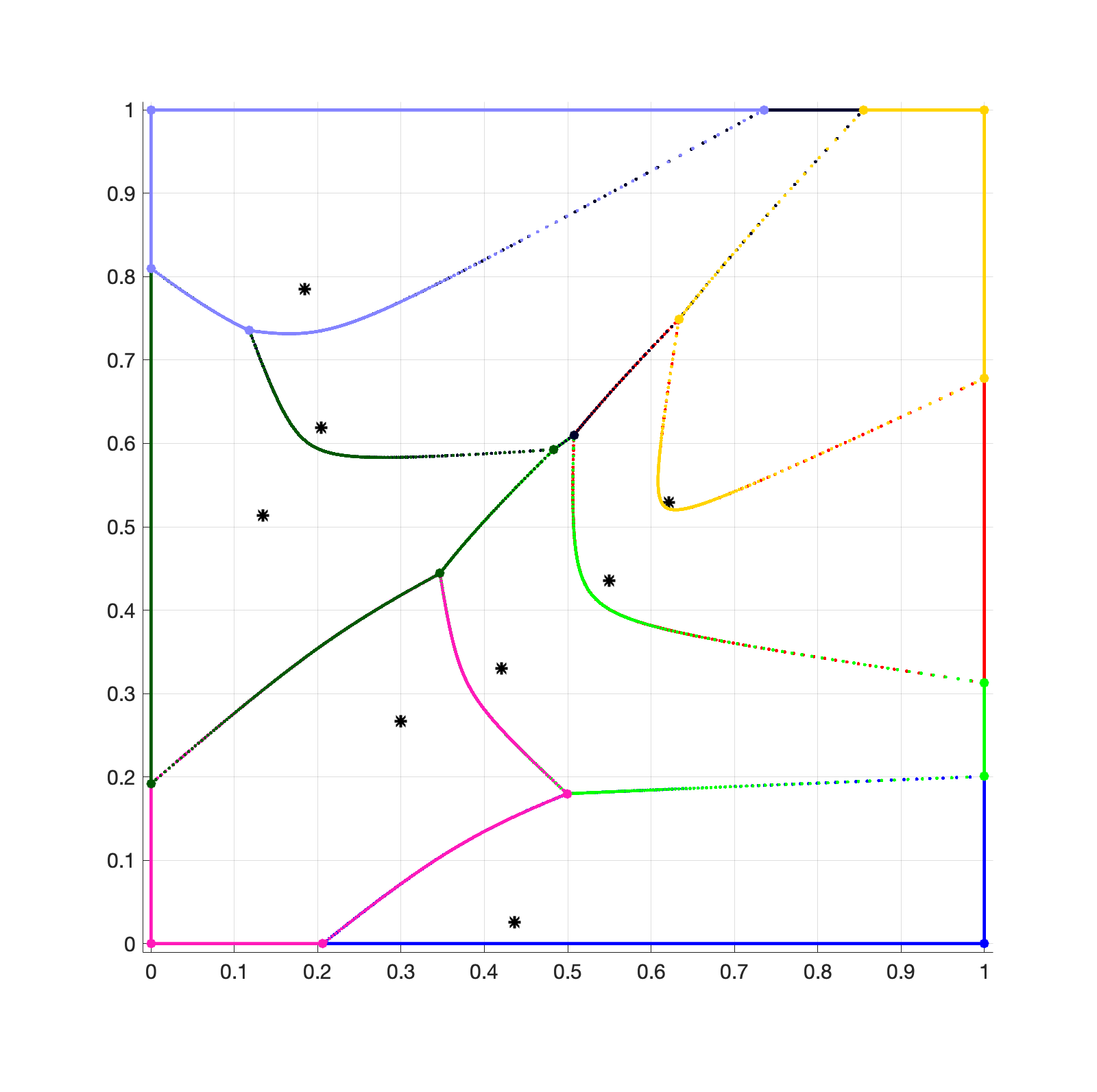}
		\caption{Example \eqref{E_4}}\label{Fig_E_4}
	\end{subfigure}
\end{figure}

\begin{exm}\label{CompareOthers}
In Table \ref{Tab_Algcomp},  we compare performance of our method with 
two other techniques, the ``boundary method'' of \cite{LucaJD} and the
minimization approach via use of a quasi-Newton's method 
as in \cite{Hartmann2020SemidiscreteOT}.  In the former case,
the results are those obtained by the C++-code kindly provided
by J.D. Walsh and the {\bf Time} in this case refers to that
of the C++-code, neglecting the final time required to measure
the {\bf Error} since the final areas computation is done by
our approach.  In the case of the BFGS quasi-Newton method
we used the tried and true {\tt Matlab} routine {\tt fminunc}.
First of all, with the boundary method we could not reach the required 
accuracy of ${\mathtt{TOL}}=10^{-8}$ and thus had to lower
that to $10^{-5}$, and even so we had to decrease considerably the
minimum allowed grid-size.  Secondly, with BFGS we could never obtain
answers more accurate than a couple of digits and the code 
{\tt fminunc} always returned a message that was not able
to achieve further reduction in the error.

\rm{
\begin{table}[ht]
	\centering\scalebox{.75}{\renewcommand*{\arraystretch}{2}
		\begin{tabular}{||c||c|c|c||c|c|c||c|c|c||}\hline\hline
			& \multicolumn{3}{|c||}{\textbf{Our method}} & \multicolumn{3}{|c||}{\textbf{Boundary method of \cite{LucaJD}}} & \multicolumn{3}{|c||}{\textbf{Quasi-Newton (BFGS) using \tt{fminunc}}}\\\hline\hline
			Example & \textbf{Error} & \textbf{Iterations} & \textbf{Time} & \textbf{Error} & \textbf{Grid size} & \textbf{Time} & \textbf{Error} & \textbf{Iterations} & \textbf{Time}\\\hline\hline
			\eqref{E_1} & $4.1959*10^{-11}$ & 2 & 7.9786 & $7.0657*10^{-6}$ & $2^{-15}$ & 1.01 & $1.0365*10^{-2}$ & 1 & 645.81 \\\hline
			\eqref{E_2} & $1.5987*10^{-9}$ & 6 & 37.807 & $7.8626*10^{-6}$ & $2^{-18}$ & 34.5 & $5.0087*10^{-2}$ & 7 & 591.15 \\\hline
			\eqref{E_3} & $7.2028*10^{-6}$ & 1 & 6.2608 & $5.3557*10^{-6}$ & $2^{-16}$ & 6.06 & $1.3930*10^{-3}$ & 3 & 500.25 \\\hline
			\eqref{E_4} & $3.3980*10^{-7}$ & 2 & 16.170 & $8.3461*10^{-7}$ & $2^{-20}$ & 772  & $7.6671*10^{-3}$ & 1 & 430.38\\\hline
			\hline
	\end{tabular}}
	\caption{Different Methods}\label{Tab_Algcomp}
\end{table}
}
\end{exm}

\begin{exm}[Example with Different Source Densities]\label{DifferentSources}
Here we keep $\Omega = [0,1]\times [0,1]$, the cost is $c(\x,\y) = ||\x-\y||_2$,
the target points are located at
$\bigg{\{}\small{\begin{pmatrix}0.25\\0.25\end{pmatrix},\ \begin{pmatrix}0.5\\0.75\end{pmatrix},\ \begin{pmatrix}0.75\\0.25\end{pmatrix},\ \begin{pmatrix}0.5\\0.3\end{pmatrix}}\bigg{\}}$, $\nu$ is uniform, and we take four different source densities to illustrate the
impact of $\rho$ on the overall error and execution time.  The source
densities are
\begin{gather*}
	\rho_1(x) = 1,\,\ \text{uniform},\,\ \ 
	\rho_2(x) = 4x_1x_2,\,\ \text{non-uniform},\,\ 
	\rho_3(x) = \gamma e^{-10(x_1-0.5)^2 - 
		10(x_2 - 0.5)^2},\,\ \text{Gaussian} \\
	\rho_4(x) =  \begin{cases} 
		\frac{1}{2} & 0 \leq x_1 \leq 0.3 \ ,\, \forall x_2 \\
		g(x) & 0.3 < x_1 < 0.7\ ,\, \forall x_2 \\
		\frac{3}{2} & 0.7 \leq x_1 \leq 1 \ ,\, \forall x_2 
	\end{cases}
\,\,\, \text{non smooth, smoothed out}
\end{gather*}
where $\gamma$ is the normalization constant and 
\begin{gather*}
	g(x_1,x_2) = \frac{1}{2} + \frac{(500 x_1 (4 x_1 (175 (x_1-3) x_1+594)-1203)+115173) (10 x_1-3)^5}{131072}\ .
\end{gather*}
We notice that $g(x)$ is obtained as a $\cont^4$ interpolant on a
non-smooth density given by $1/2$ before $x_1=1/2$ and by $3/2$
past it.  For all problems, 
the initial guess is obtained with grid size $h=0.05$.

\rm{
\begin{table}[ht]
	\centering\scalebox{.75}{\renewcommand*{\arraystretch}{2}
		\begin{tabular}{||c|c|c|c|c|c||}\hline\hline
			\textbf{Source Density} & \textbf{Error} &  \textbf{Iterations} & \textbf{Damping} & \textbf{Time} & \textbf{Feasibility Coefficient} $\mathbf{\kappa}$ \\ \hline\hline
			Uniform & $3.7612*10^{-9}$ & 2 & 0 & $12.806$ & $0.45594$ \\\hline
			Non-Uniform & $3.4750*10^{-14}$ & 7 & 2 & $46.448$ & $0.13112$ \\\hline
			Gaussian & $7.3751*10^{-11}$ & 4 & 0 & $53.691$ & $0.66334$ \\\hline
			Non smooth & $4.9682*10^{-14}$ & 6 & 0 & $119.99$ & $0.34405$ \\\hline
			\hline
	\end{tabular}}
	\caption{Example with Different Source Densities}\label{Tab_DiffSource}
\end{table}
}

\begin{figure}[ht]
	\centering
	\begin{subfigure}[b]{0.24\linewidth}
		\centering
		\includegraphics[width=\linewidth]{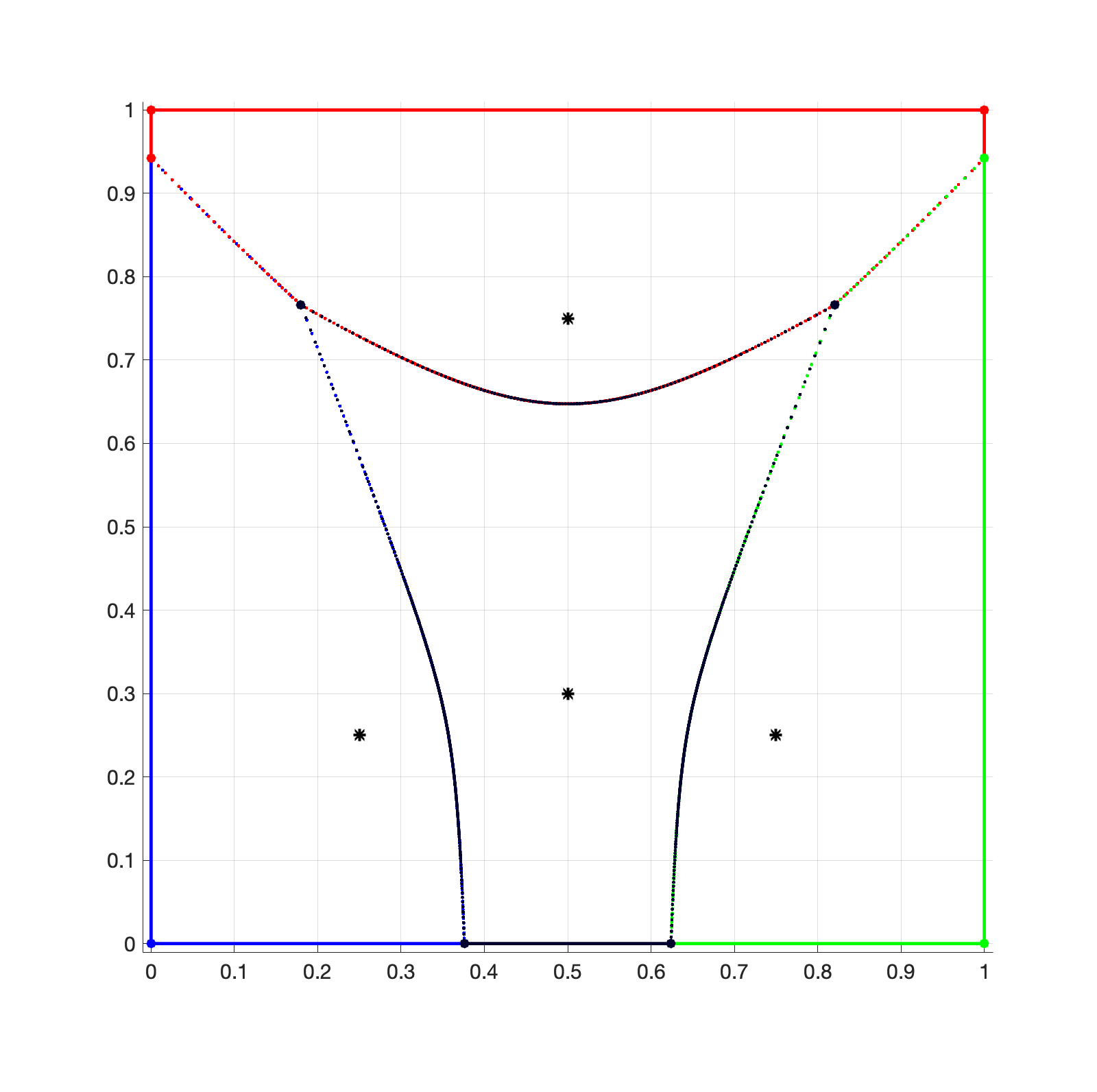}
		\caption{Uniform}
	\end{subfigure}
	\hfill
	\begin{subfigure}[b]{0.24\linewidth}
		\centering
		\includegraphics[width=\linewidth]{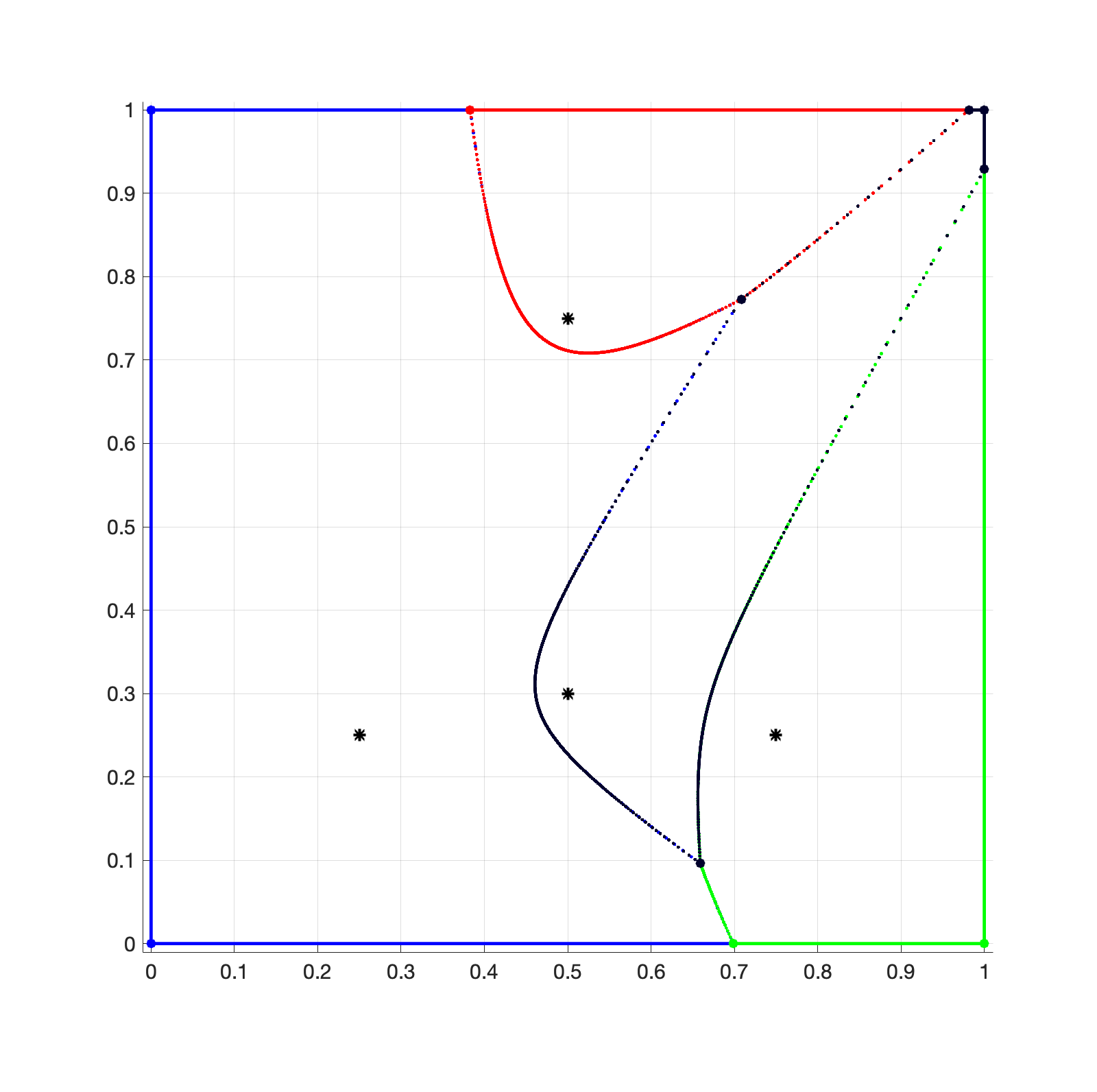}
		\caption{Non-Uniform}
	\end{subfigure}
	\hfill
	\begin{subfigure}[b]{0.24\linewidth}
		\centering
		\includegraphics[width=\linewidth]{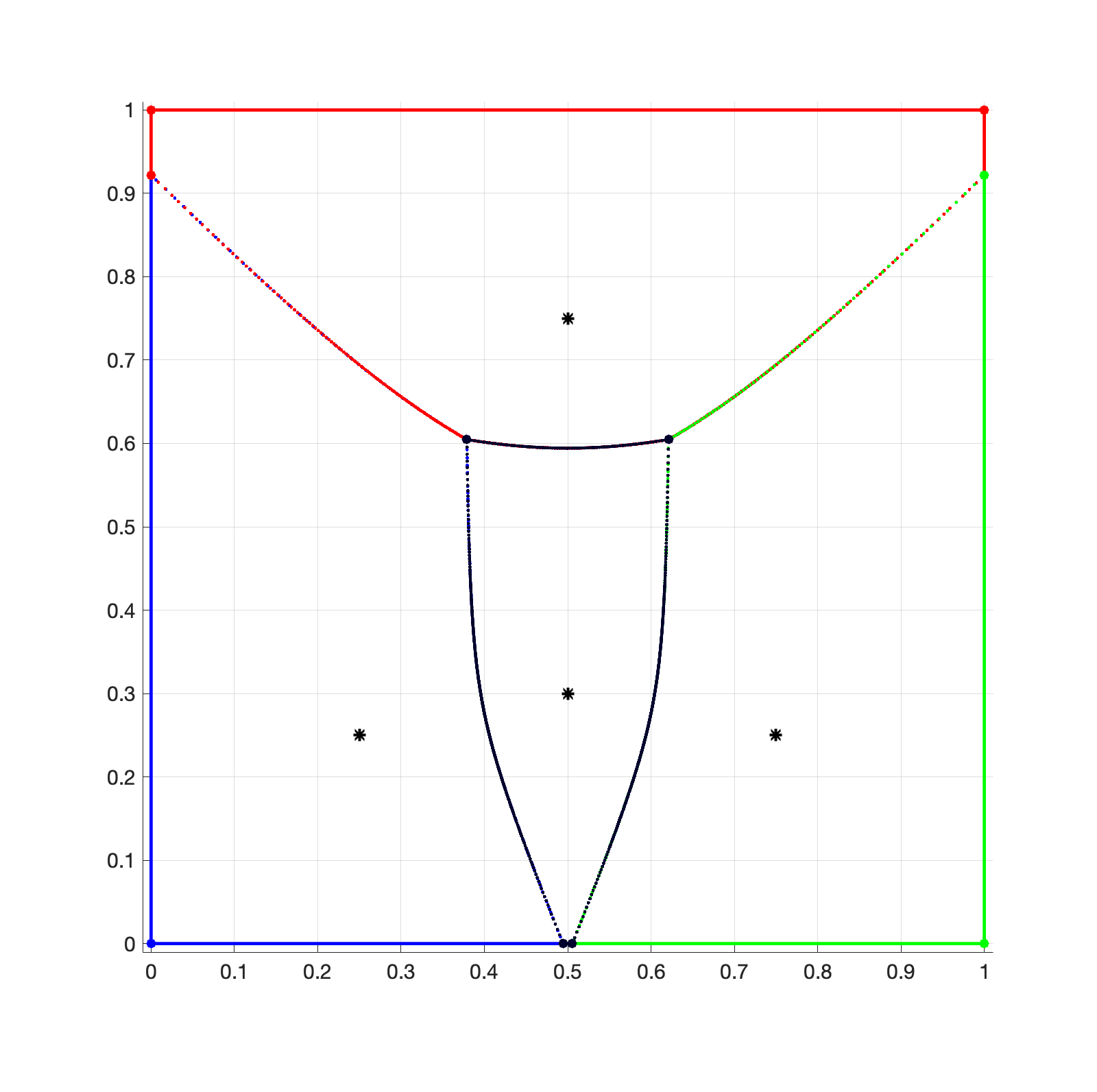}
		\caption{Gaussian}
	\end{subfigure}
	\hfill
	\begin{subfigure}[b]{0.24\linewidth}
		\centering
		\includegraphics[width=\linewidth]{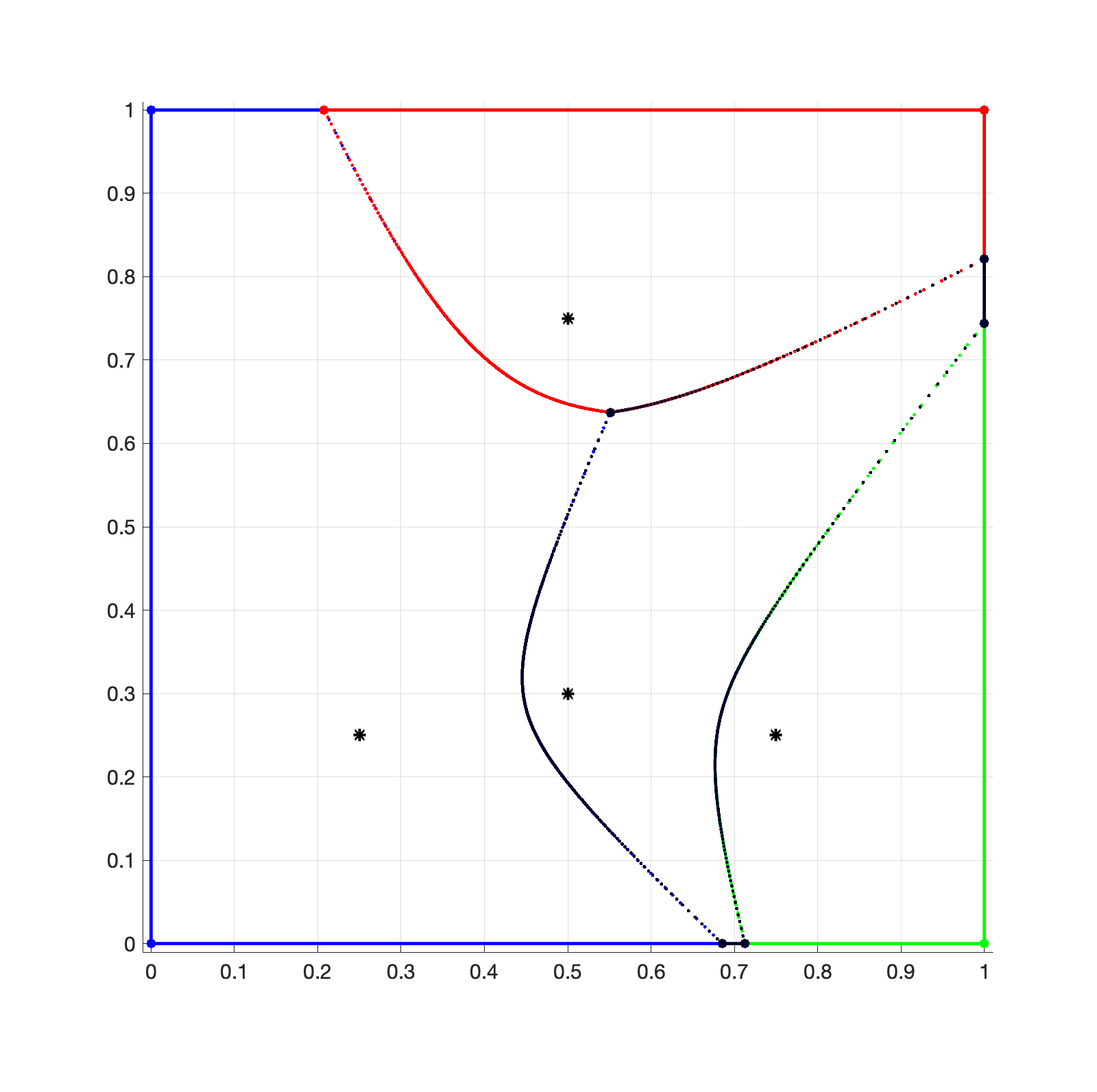}
		\caption{Non smooth}
	\end{subfigure}
\end{figure}
{\it
Clearly, the non-smooth case is the most computationally expensive due to the
very large values of the derivatives of $\rho$.
}
\end{exm}

\begin{exm}[Different $p$-norms]\label{DifferentCosts}
Here we want to assess the impact of different $p$-norms for the cost,
in particular we are interested in assessing the impact of using odd
value of $p$ (which gives reduced smoothness with respect
to the even values of $p$) and of ``approaching'' the $\infty$-norm
and the $1$-norm.  Although with costs given by the 
$\infty$-norm and the $1$-norm the problem is not well posed
(see Example \ref{Rem_CountExamp}), we will see that as we approach
these values the algorithm selects clearly defined Laguerre
tessellations, and we conjecture that this is a typical scenario
and deserves further study.  (Incidentally, also for Example \ref{Rem_CountExamp}, the
approach just outlined select a well defined tessellation).

We fix $\Omega = [0,1]\times [0,1]$, 
$\y = \bigg{\{}\small{\begin{pmatrix}0.25\\0.25\end{pmatrix},\ 
\begin{pmatrix}0.5\\0.75\end{pmatrix},\ 
\begin{pmatrix}0.75\\0.25\end{pmatrix}}\bigg{\}}$, take $\rho$ and $\nu$
uniform, and we consider the following cost functions: 
Cost function:
\begin{enumerate}
	\item $c(\x,\y)=\|\x-\y\|_3$; $c(\x,\y) = \frac{1}{2}(||\x-\y||_2 +||\x-\y||_4)$); $c(\x,\y) = ||\x-\y||_3 + ||\x-\y||_5 + ||\x-\y||_7$;
	\item toward the $\infty$-norm: $c(\x,\y)=\|\x-\y\|_{2^k}$, $k = 1, 2, 3, 4, 5$;
	\item toward the $1$-norm: $c(\x,\y)=\|\x-\y\|_{1+2^{-k}}$, 
	$k = 1, 2, 3, 4, 5$.
\end{enumerate}
All results obtained with initial guess from the grid based approach 
with $h=0.05$.

\rm{
\begin{table}[ht]
	\centering\scalebox{.75}{\renewcommand*{\arraystretch}{2}
		\begin{tabular}{||c|c|c|c|c|c||}\hline\hline
			\textbf{Cost Function} & \textbf{Error} &  \textbf{Iterations} & \textbf{Damping} & \textbf{Time} & \textbf{Feasibility Coefficient} $\mathbf{\kappa}$\\\hline\hline
			$\|\cdot \|_3$ & $1.9444*10^{-10}$ & 2 & 0 & $8.6622$ & $0.74508$ \\\hline
			$\frac{1}{2}(||\cdot ||_2 +||\cdot ||_4)$ & $1.0988*10^{-12}$ & 2 & 0 & $10.756$ & $0.74652$\\\hline
			$\|\cdot\|_3+\|\cdot\|_y+\|\cdot\|_7$ & $8.1481*10^{-10}$ & 2 & 0 & $10.742$ & $0.74023$\\\hline
			\hline
	\end{tabular}}
	\caption{Example with Different $p$-Norm Cost Functions}\label{Tab_DiffNorm}
\end{table}
}

\it{
\noindent
No appreciable difference is seen with the above three different costs, as
confirmed by Figure \ref{Fig_DiffNorm}.}

\rm{
\begin{table}[ht]
	\centering\scalebox{.75}{\renewcommand*{\arraystretch}{2}
		\begin{tabular}{||c|c|c|c|c|c||}\hline\hline
			\textbf{} & \textbf{Error} & \textbf{Iterations} & \textbf{Damping} & \textbf{Time} & \textbf{Feasibility Coefficient} $\mathbf{\kappa}$\\\hline\hline
			$k=1$ & $1.5999*10^{-10}$ & 2 & 0 & $6.539$ & $0.74940$ \\\hline
			$k=2$ & $7.9542*10^{-10}$ & 2 & 0 & $7.6051$ & $0.74083$ \\\hline
			$k=3$ & $4.1729*10^{-9}$ & 2 & 0 & $7.8891$ & $0.73576$ \\\hline
			$k=4$ & $7.0427*10^{-9}$ & 2 & 0 & $8.1167$ & $0.73452$ \\\hline
			$k=5$ & $6.1689*10^{-11}$ & 3 & 0 & $320.97$ & $0.73414$ \\\hline\hline
	\end{tabular}}
	\caption{Approaching the $\infty$-norm.  Cost is 
		$\|\x-\y \|_{2^k}$}\label{Tab_LargeNorm}
\end{table}
}

\it{\noindent
For larger values of $p$, computation of the integrals becomes expensive
because of large derivative values.  Figure \ref{Fig_LargeNorm} clearly
shows the decrease in smoothness. }
\rm{
\begin{table}[ht]
	\centering\scalebox{.75}{\renewcommand*{\arraystretch}{2}
		\begin{tabular}{||c|c|c|c|c|c||}\hline\hline
			\textbf{} & \textbf{Error} & \textbf{Iterations} & \textbf{Damping} & \textbf{Time} & \textbf{Feasibility Coefficient} $\mathbf{\kappa}$\\\hline\hline
			$k=1$ & $5.1627*10^{-9}$ & 2 & 0 & $46.634$ & $0.74426$ \\\hline
			$k=2$ & $1.6436*10^{-11}$ & 3 & 0 & $214.18$ & $0.73291$ \\\hline
			$k=3$ & $3.0620*10^{-9}$ & 2 & 0 & $165.3$ & $0.7261$ \\\hline
			$k=4$ & $3.4457*10^{-10}$ & 2 & 0 & $249.13$ & $0.72406$\\\hline
			$k=5$ & $2.6122*10^{-11}$ & 2 & 0 & $283.11$ & $0.72312$\\\hline
			\hline
	\end{tabular}}
	\caption{Approaching the $1$-norm. Cost is 
		$\|\x-\y \|_{1+2^{-k}}$}\label{Tab_SmallNorm}
\end{table}
}

\it{
\noindent
Also for costs approaching $1$-norm, there is a clear increase in the
computation time, again due to the decreased smoothness
of the boundary.  Figure \ref{Fig_SmallNorm} shows this clearly.

It is very interesting that the value of $\kappa$ is effectively the same
across all of the above experiments, and does not betray ill-conditioning
of the problem.  This is reflected in the fact that
only 2 (or occasionally 3) Newton iterations are required for all of
the above costs, and that the value of the Error is small.  The difficulties
here are entirely due to the decrease in smoothness of the boundary.}

\begin{figure}[ht]
	\centering
	\begin{subfigure}[b]{0.3\linewidth}
		\centering
		\includegraphics[width=\linewidth]{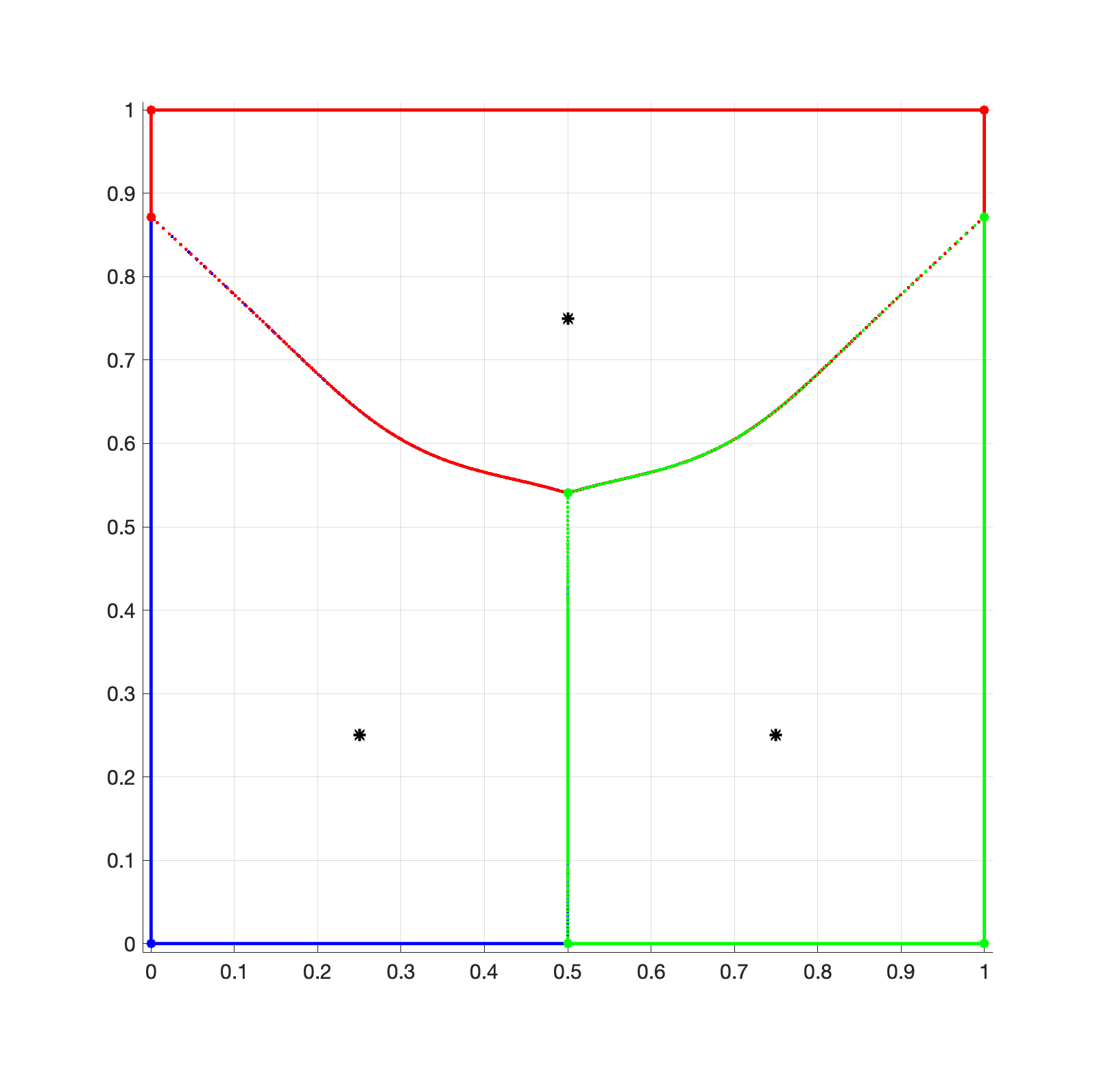}
		\caption{$\|\cdot\|_3$}\label{Fig_DiffNorm_3}
	\end{subfigure}
	\hfill
	\begin{subfigure}[b]{0.3\linewidth}
		\centering
		\includegraphics[width=\linewidth]{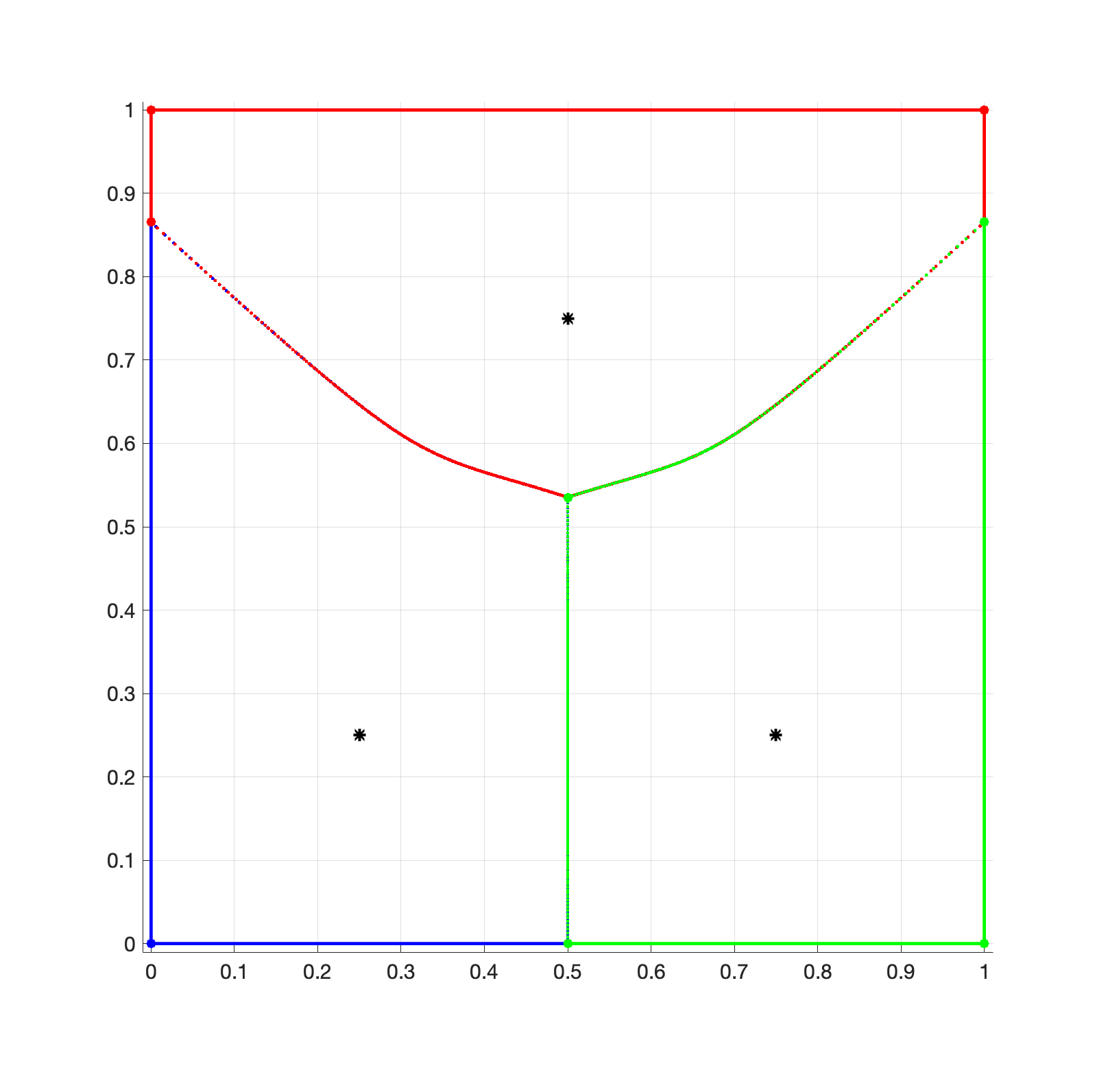}
		\caption{$\frac{1}{2}(\|\cdot\|_2+\|\cdot\|_4)$}\label{Fig_SmallNorm_2-4}
	\end{subfigure}
	\hfill
	\begin{subfigure}[b]{0.3\linewidth}
		\centering
		\includegraphics[width=\linewidth]{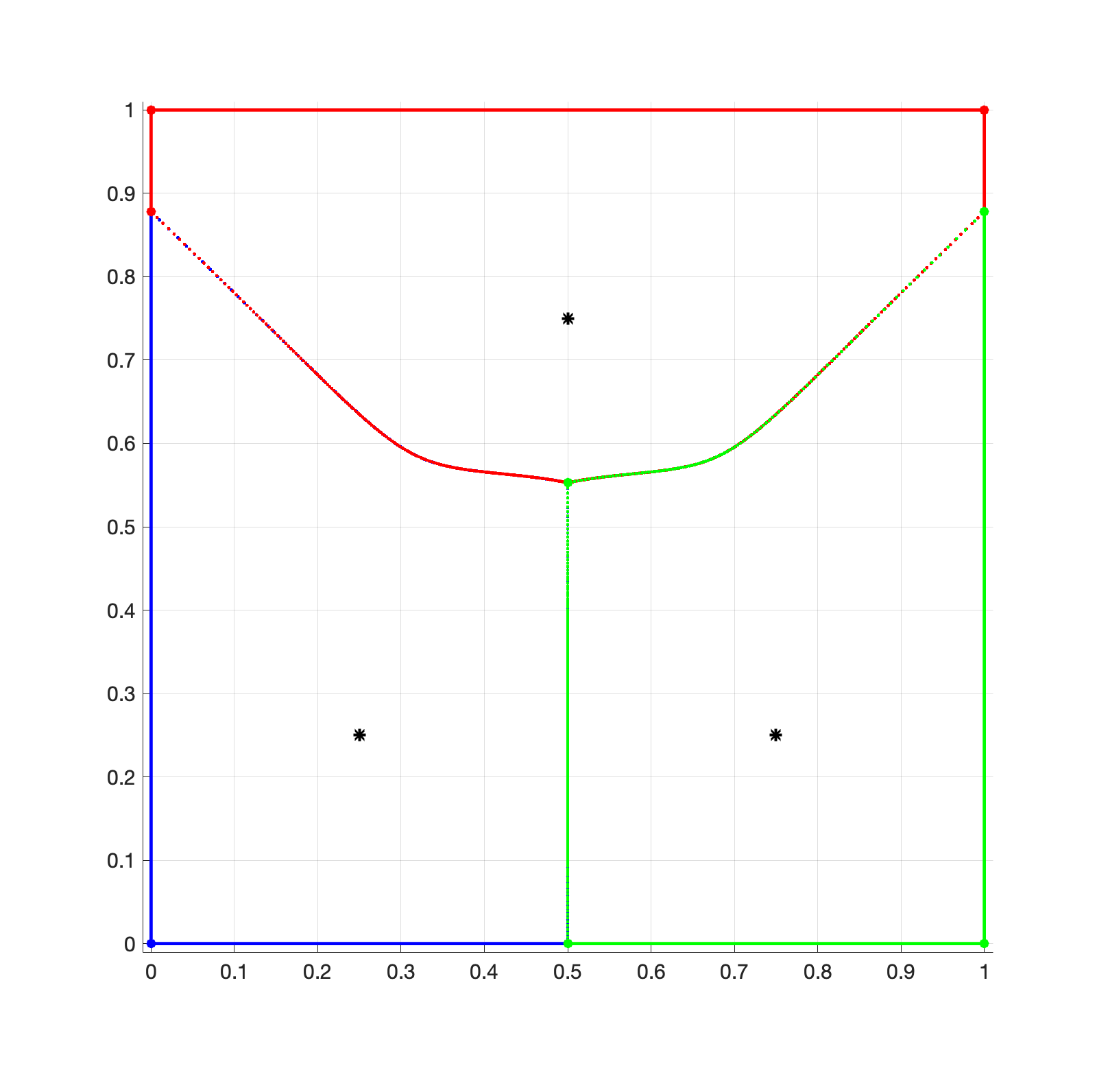}
		\caption{$\|\cdot\|_3+\|\cdot\|_5+\|\cdot\|_7$}\label{Fig_SmallNorm_3-5-7}
	\end{subfigure}
	\caption{Different $p$-norm costs}\label{Fig_DiffNorm}
\end{figure}

\begin{figure}[ht]
	\centering
	\begin{subfigure}[b]{0.3\linewidth}
		\centering
		\includegraphics[width=\linewidth]{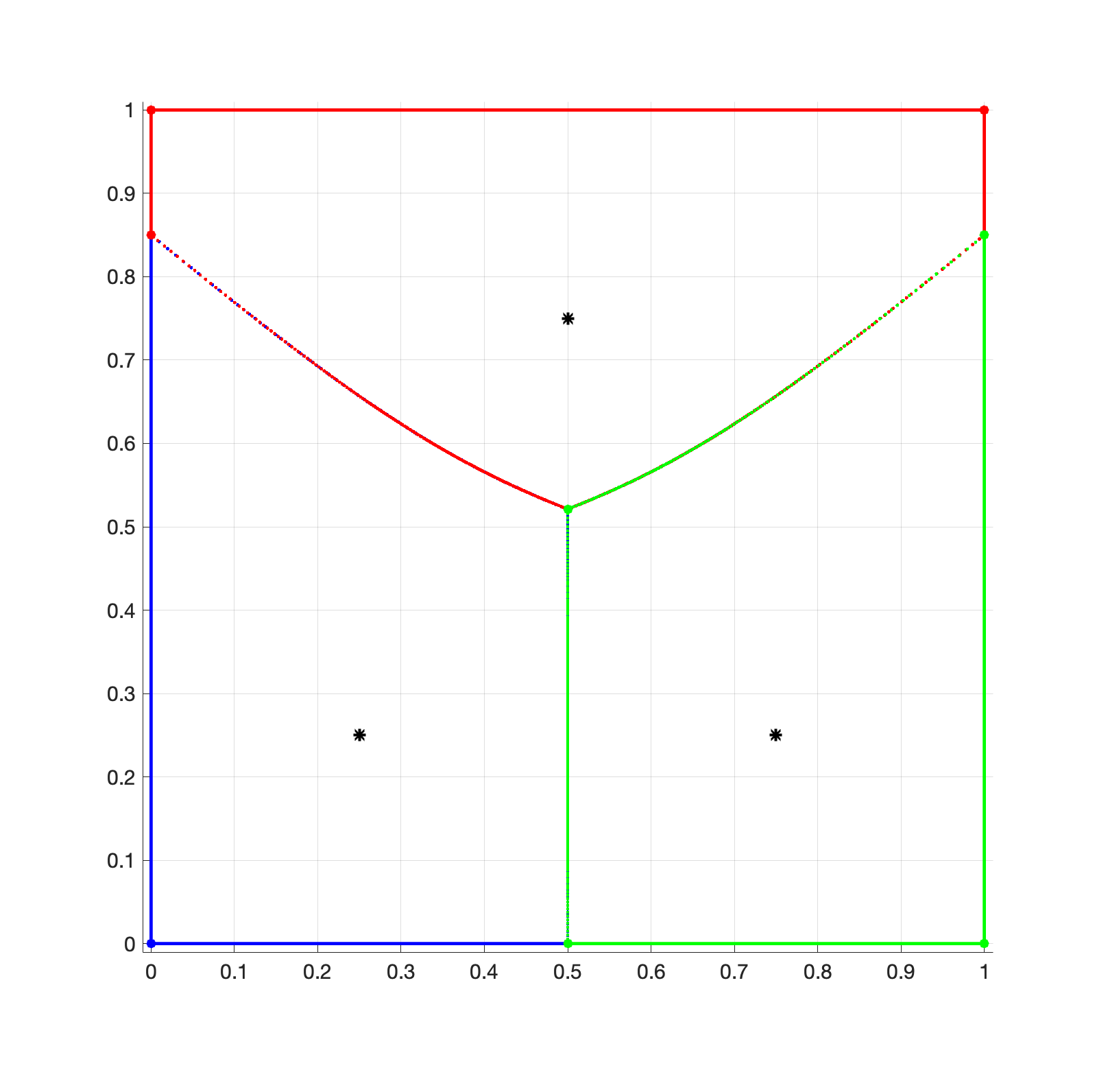}
		\caption{$\|\cdot \|_2$}\label{Fig_DiffNorm_2}
	\end{subfigure}
	\hfill
	\begin{subfigure}[b]{0.3\linewidth}
		\centering
		\includegraphics[width=\linewidth]{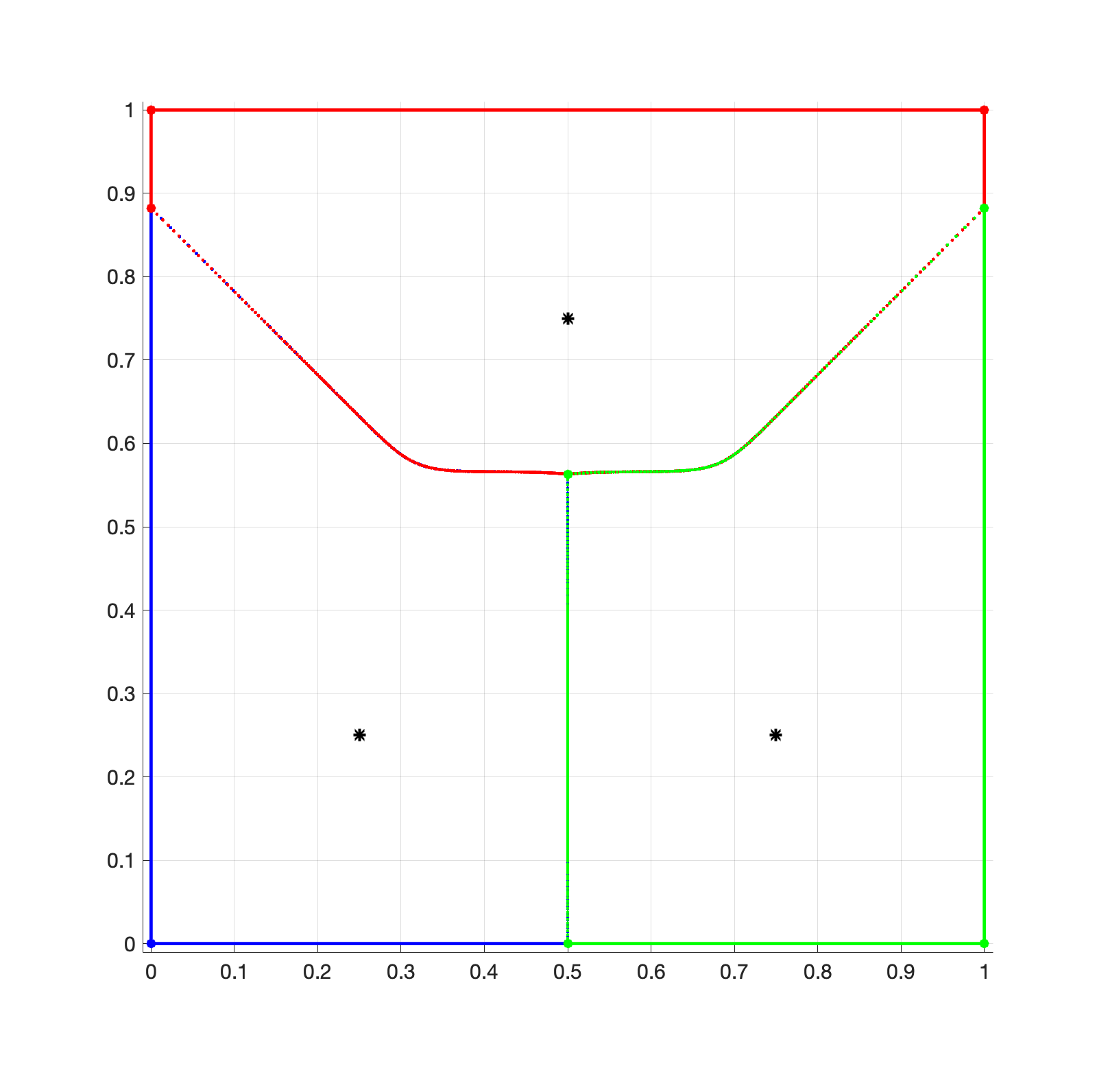}
		\caption{$\|\cdot \|_8$}\label{Fig_LargeNorm_8}
	\end{subfigure}
	\hfill
	\begin{subfigure}[b]{0.3\linewidth}
		\centering
		\includegraphics[width=\linewidth]{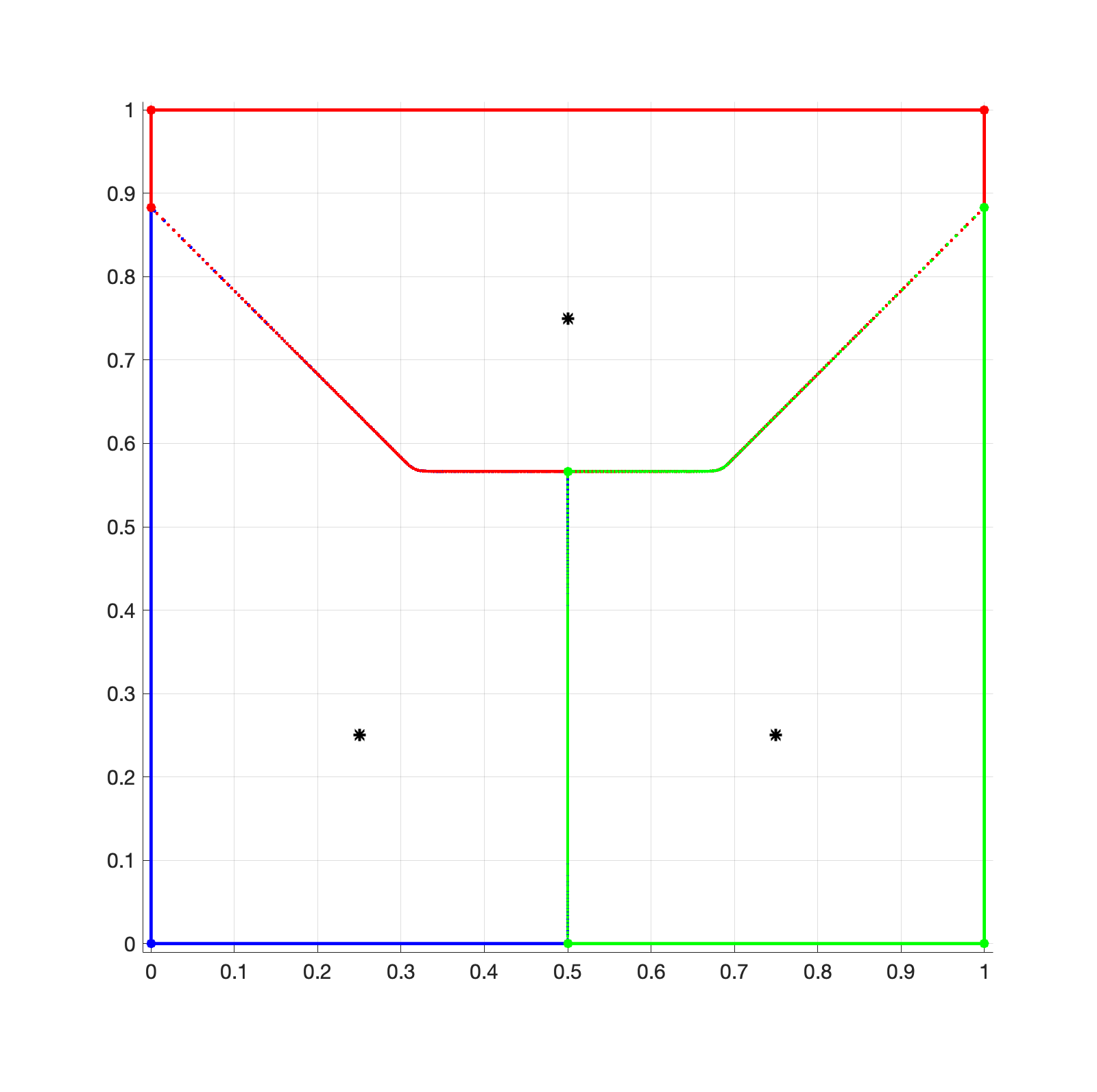}
		\caption{$\|\cdot \|_{32}$}\label{Fig_LargeNorm_32}
	\end{subfigure}
	\caption{$p$-norm costs for growing $p$}\label{Fig_LargeNorm}
\end{figure}

\begin{figure}[ht]
	\centering
	\begin{subfigure}[b]{0.3\linewidth}
		\centering
		\includegraphics[width=\linewidth]{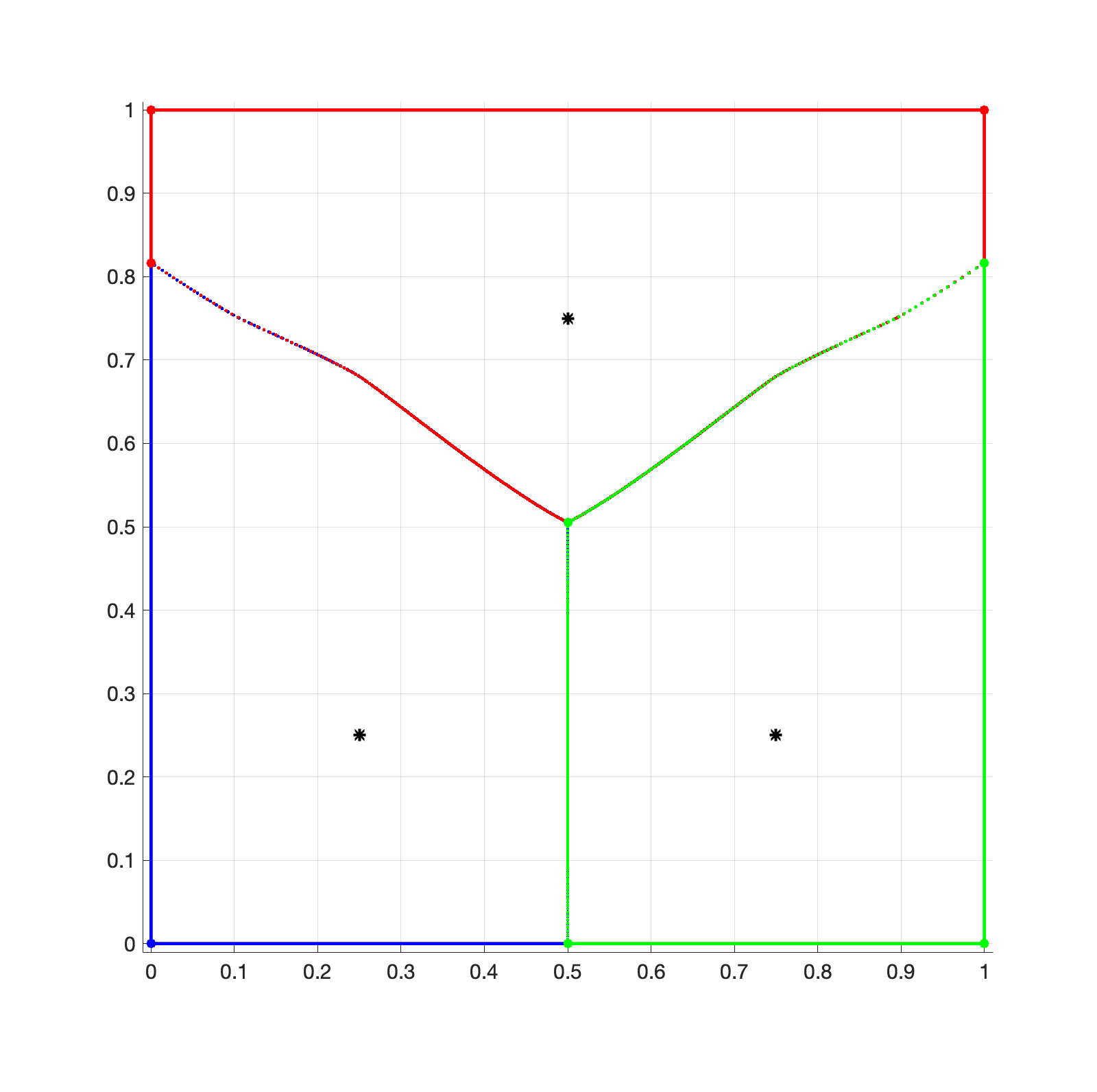}
		\caption{$\|\cdot \|_{1+\frac{1}{2}}$}\label{Fig_SmallNorm_1.5}
	\end{subfigure}
	\hfill
	\begin{subfigure}[b]{0.3\linewidth}
		\centering
		\includegraphics[width=\linewidth]{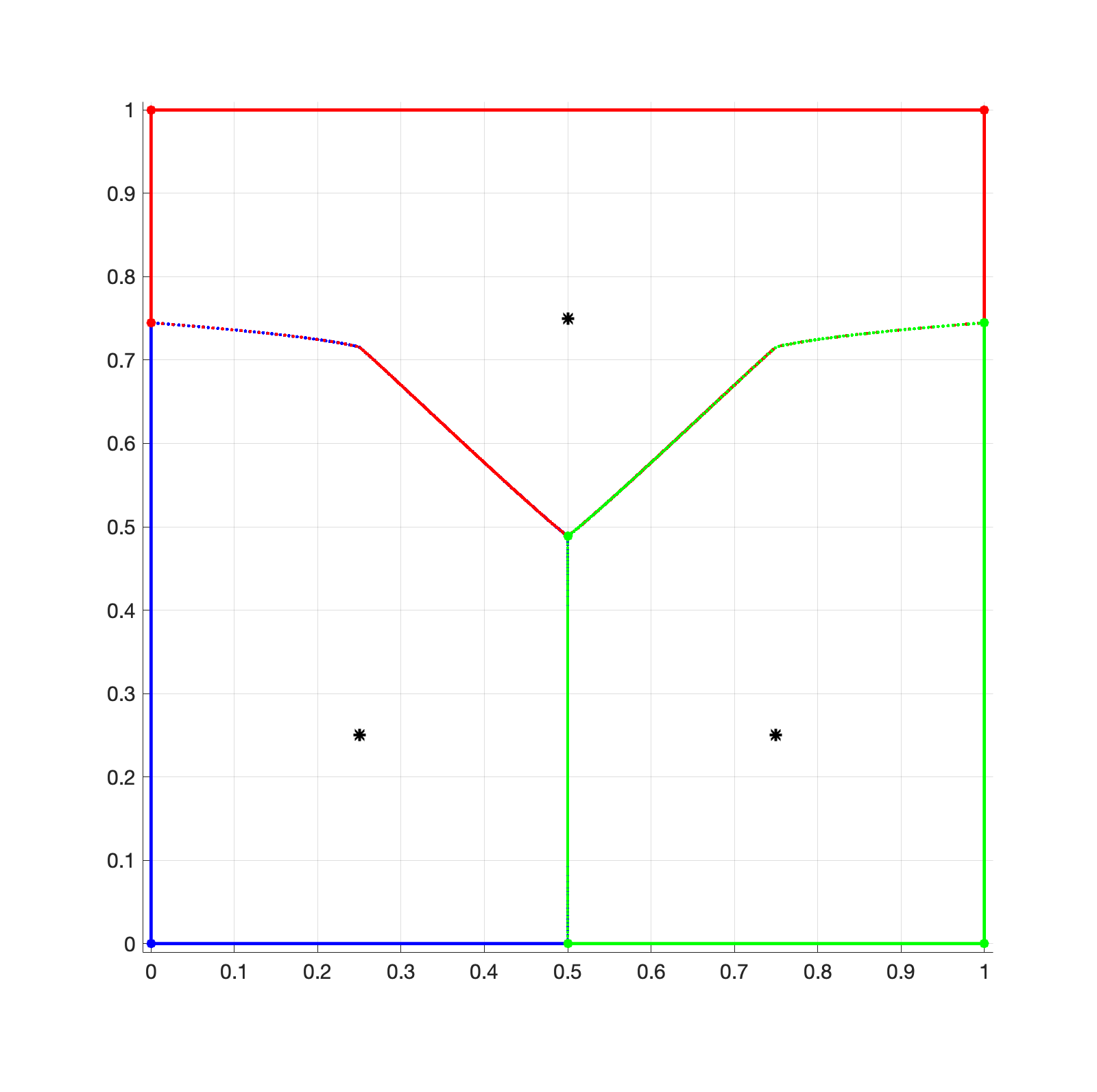}
		\caption{$\|\cdot \|_{1+\frac{1}{8}}$}\label{Fig_SmallNorm_1.0125}
	\end{subfigure}
	\hfill
	\begin{subfigure}[b]{0.3\linewidth}
		\centering
		\includegraphics[width=\linewidth]{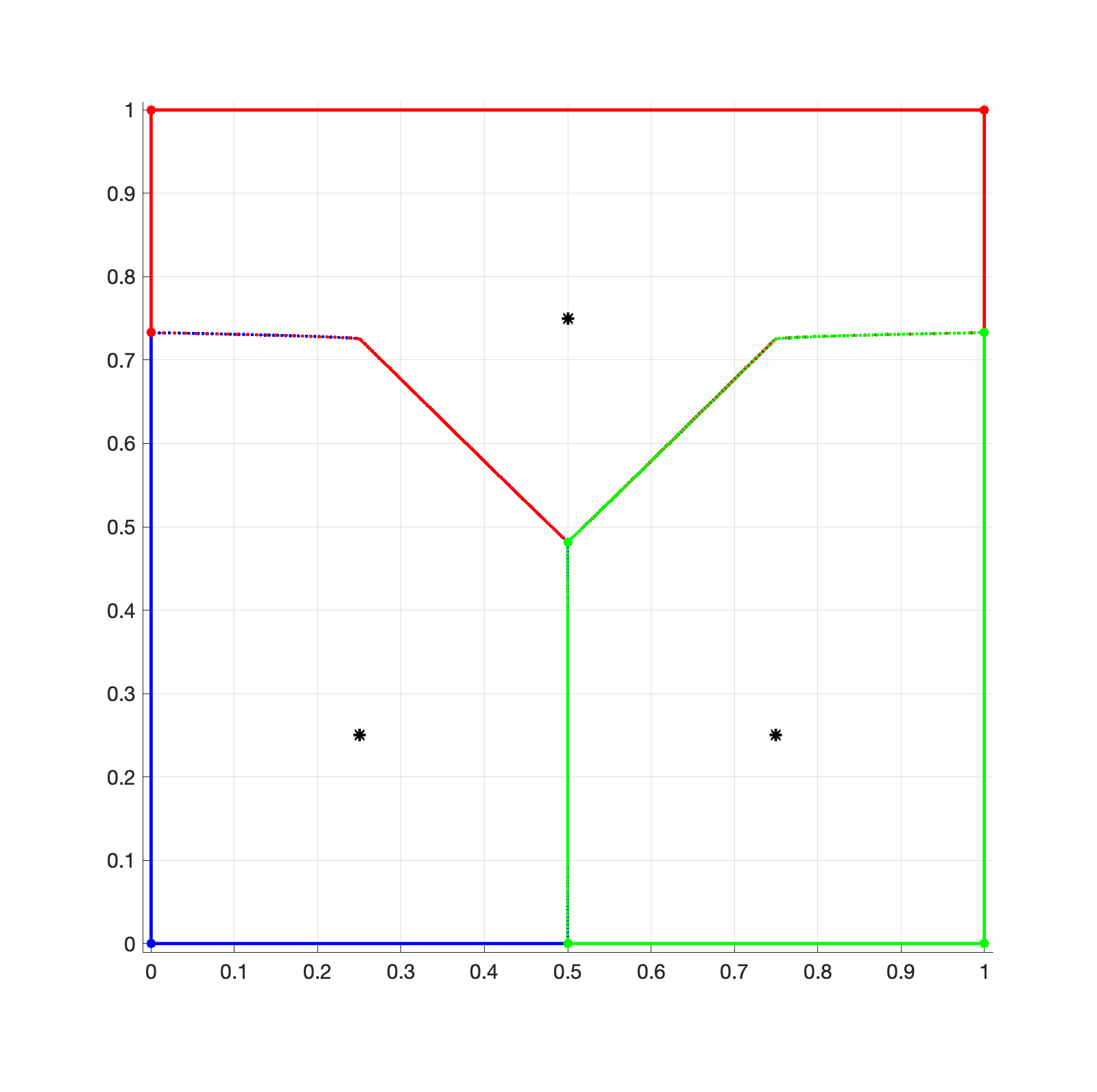}
		\caption{$\|\cdot \|_{1+\frac{1}{32}}$}\label{Fig_SmallNorm_1.003125}
	\end{subfigure}
	\caption{$p$-norm costs for $p$ decreasing toward $1$}\label{Fig_SmallNorm}
\end{figure}
\end{exm}

\begin{exm}[Impact of Feasibility]\label{ImpactFeas}
This is a difficult problem for all methods we tested, in that it
becomes arbitrarily ill-conditioned .  We have
\begin{gather*}\
	\Omega = [0,1]\times [0,1],\ c(\x,\y) = ||\x-\y||_2,\  \y = \bigg{\{}\small{\begin{pmatrix}0.25\\0.5\end{pmatrix},\ \begin{pmatrix}0.75\\0.5\end{pmatrix}}\bigg{\}},\ \rho(x) = 1,
\end{gather*}
and a non-uniform target density 
$\nu = (\frac{1}{2^k}, 1 - \frac{1}{2^k})$, $k = [1,\dots, 10]$).
Results in Table \ref{Tab_Feasibility} are all obtained with initial guess from the grid-based 
approach with $h=0.1$.  The tessellation is shown in Figure \ref{Fig_Feas}.

\rm{
\begin{table}[ht]
	\centering\scalebox{.6}{\renewcommand*{\arraystretch}{2}
		\begin{tabular}{||c|c|c|c|c|c||}\hline\hline
			\textbf{Target Density} & \textbf{Error} & \textbf{Iterations} & \textbf{Damping} & \textbf{Time} & \textbf{Feasibility Coefficient} $\mathbf{\kappa}$\\\hline\hline
			$(\frac{1}{2},\ 1 - \frac{1}{2})$ & $3.3307*10^{-14}$ & 0 & 0 & $2.2765$ & $1$ \\\hline
			$(\frac{1}{2^{2}},\ 1 - \frac{1}{2^{2}})$ & $4.6629*10^{-14}$ & 3 & 0 & $11.363$ & $4.0243*10^{-1}$ \\\hline
			$(\frac{1}{2^{3}},\ 1 - \frac{1}{2^{3}})$ & $5.6483*10^{-14}$ & 2 & 0 & $9.6461$ & $2.0029*10^{-1}$ \\\hline
			$(\frac{1}{2^{4}},\ 1 - \frac{1}{2^{4}})$ & $2.9511*10^{-10}$ & 2 & 0 & $10.27$ & $7.9527*10^{-2}$ \\\hline
			$(\frac{1}{2^{5}},\ 1 - \frac{1}{2^{5}})$ & $4.1078*10^{-14}$ & 4 & 0 & $17.766$ & $2.4611*10^{-2}$ \\\hline
			$(\frac{1}{2^{6}},\ 1 - \frac{1}{2^{6}})$ & $9.0354*10^{-10}$ & 3 & 0 & $20.834$ & $6.6039*10^{-3}$ \\\hline
			$(\frac{1}{2^{7}},\ 1 - \frac{1}{2^{7}})$ & $2.7367*10^{-11}$ & 9 & 0 & $59.625$ & $1.6834*10^{-3}$ \\\hline
			$(\frac{1}{2^{8}},\ 1 - \frac{1}{2^{8}})$ & $2.7842*10^{-12}$ & 10 & 0 & $120.6$ & $4.2294*10^{-4}$ \\\hline
			$(\frac{1}{2^{9}},\ 1 - \frac{1}{2^{9}})$ & $3.9706*10^{-12}$ & 10 & 0 & $201.91$ & $1.0587*10^{-4}$ \\\hline
			$(\frac{1}{2^{10}},\ 1 - \frac{1}{2^{10}})$ & $1.2863*10^{-10}$ & 12 & 4 & $1759$ & $2.6475*10^{-5}$ \\\hline\hline
	\end{tabular}}
	\caption{Impact of $\kappa$ on Newton iterations and 
	time}\label{Tab_Feasibility}
\end{table}
}
\it{The increase in time is entirely due to the difficulty in computing
$\mu(A_i)$'s.  We also note that the initial guess for $k\ge 7$ is
just not good, as reflected in the increase in Newton's iterations and
the need for damped Newton for $k=10$.
	
\begin{figure}[ht]
	\centering
	\includegraphics[width=0.5\linewidth]{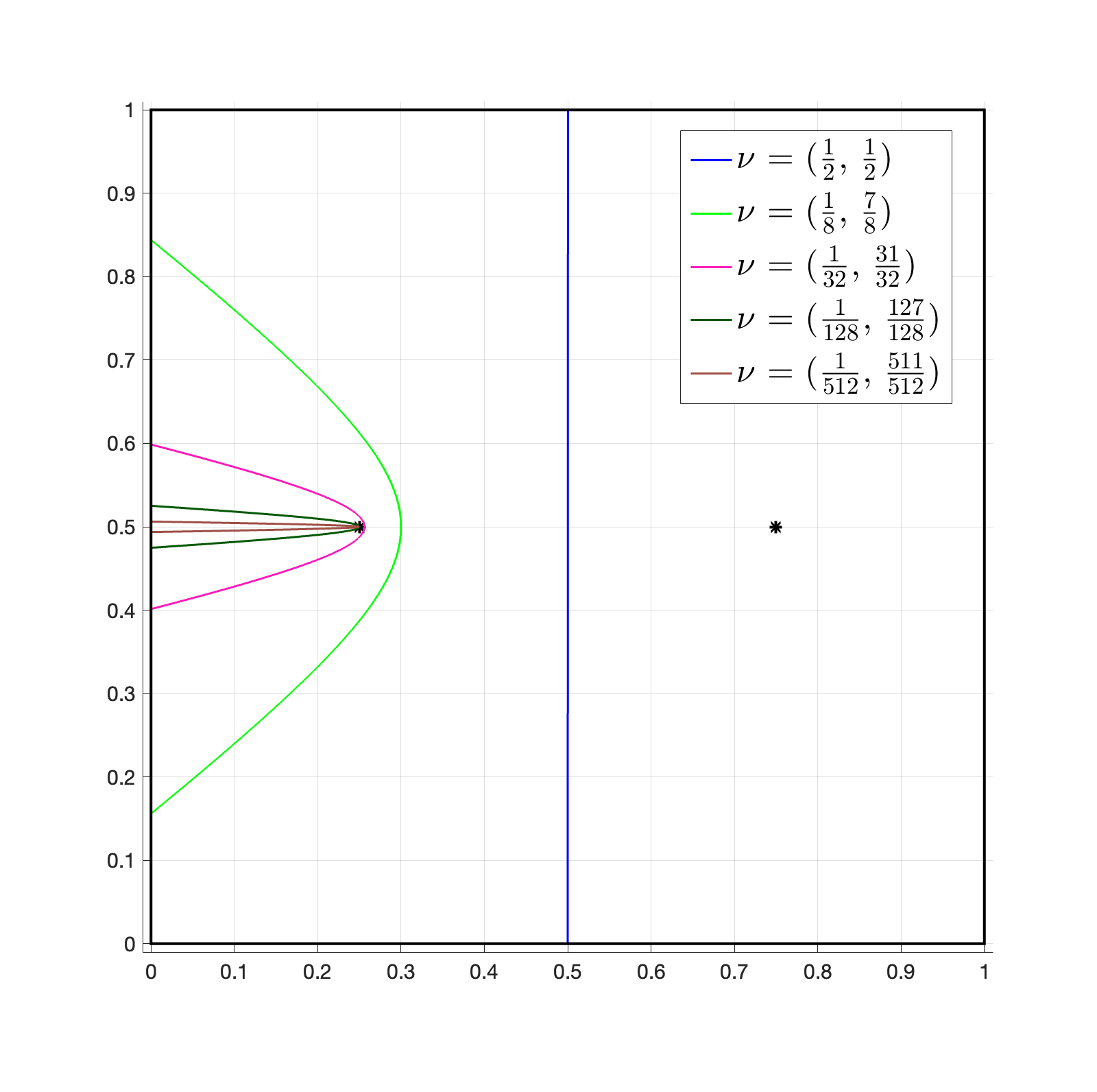}
	\caption{Impact of small $\kappa$ on the tessellation}\label{Fig_Feas}
\end{figure}
}
\end{exm}

\begin{exm}[Interplay of location, $\mathbf{\mu}$ and $\mathbf{\nu}$]
\label{Interplay}
Here we highlight the role played by consistency of the location of the target 
points with respect to the continuous
and discrete densities.  We fix
\begin{gather*}
	\Omega = [0,1]\times [0,1],\ c(\x,\y) = ||\x-\y||_2,\ \y = \bigg{\{}\small{\begin{pmatrix}0.8\\0.8\end{pmatrix},\ \begin{pmatrix}0.8\\0.9\end{pmatrix},\ \begin{pmatrix}0.9\\0.9\end{pmatrix},\ \begin{pmatrix}0.9\\0.8\end{pmatrix}}\bigg{\}},
\end{gather*}
and select three different densities
\begin{itemize}
\item[(A)] Uniform source and uniform target.  Here, the location of the target points is not 
consistent with either $\rho$ or $\nu$.
\item[(B)] Uniform source ($\rho(x) = 1$) and non-uniform target 
($\nu = (0.75, 0.1, 0.05, 0.1)^T$).  Here, the location of the
target points is consistent with $\nu$ but not with $\mu$.
\item[(C)] Non-Uniform source ($\rho(x) = 16x_1^3x_2^3$) and uniform target ($\nu_i = \frac{1}{4}$, $\forall i$).  Here, the location of the
target points is consistent with $\mu$ but not with $\nu$.
\end{itemize}
For all cases, the initial guess was obtained with the
grid-based approach with $h=0.05$.

\rm{
\begin{table}[ht]
	\centering\scalebox{.75}{\renewcommand*{\arraystretch}{2}
		\begin{tabular}{||c|c|c|c|c|c||}\hline\hline
			\textbf{Densities} & \textbf{Error} & \textbf{Iterations} & \textbf{Damping} & \textbf{Time} & \textbf{Feasibility Coefficient} $\mathbf{\kappa}$\\\hline\hline
			Uniform to Uniform & $3.2677*10^{-9}$ & 2 & 0 & $89.095$ & $0.02198$ \\\hline
			Uniform to Non-Uniform & $1.2056*10^{-12}$ & 3 & 0 & $45.295$ & $0.14509$ \\\hline
			Non-Uniform to Uniform & $8.3026*10^{-9}$ & 3 & 0 & $144.78$ & $0.86597$ \\\hline\hline
	\end{tabular}}
	\caption{Interplay between Source and Target Densities with Location of Points}\label{Tab_Interplay}
\end{table}
}
\it{
The small value of $\kappa$ in case (A) reflects the inconsistent location
of the target points with respect to $\mu$ and $\nu$.  The high execution
time for case (C), instead is entirely due to the cost of computing $\mu(A_i)$'s
by our algorithm because of the change of variable we do for
$(x_1,x_2)$, namely 
$\rho(x_1,x_2)=\rho(Y_1,+D\cos(\theta), Y_2+D\sin(\theta))$.
}
\begin{figure}[ht]
	\centering
	\begin{subfigure}[b]{0.3\linewidth}
		\centering
		\includegraphics[width=\linewidth]{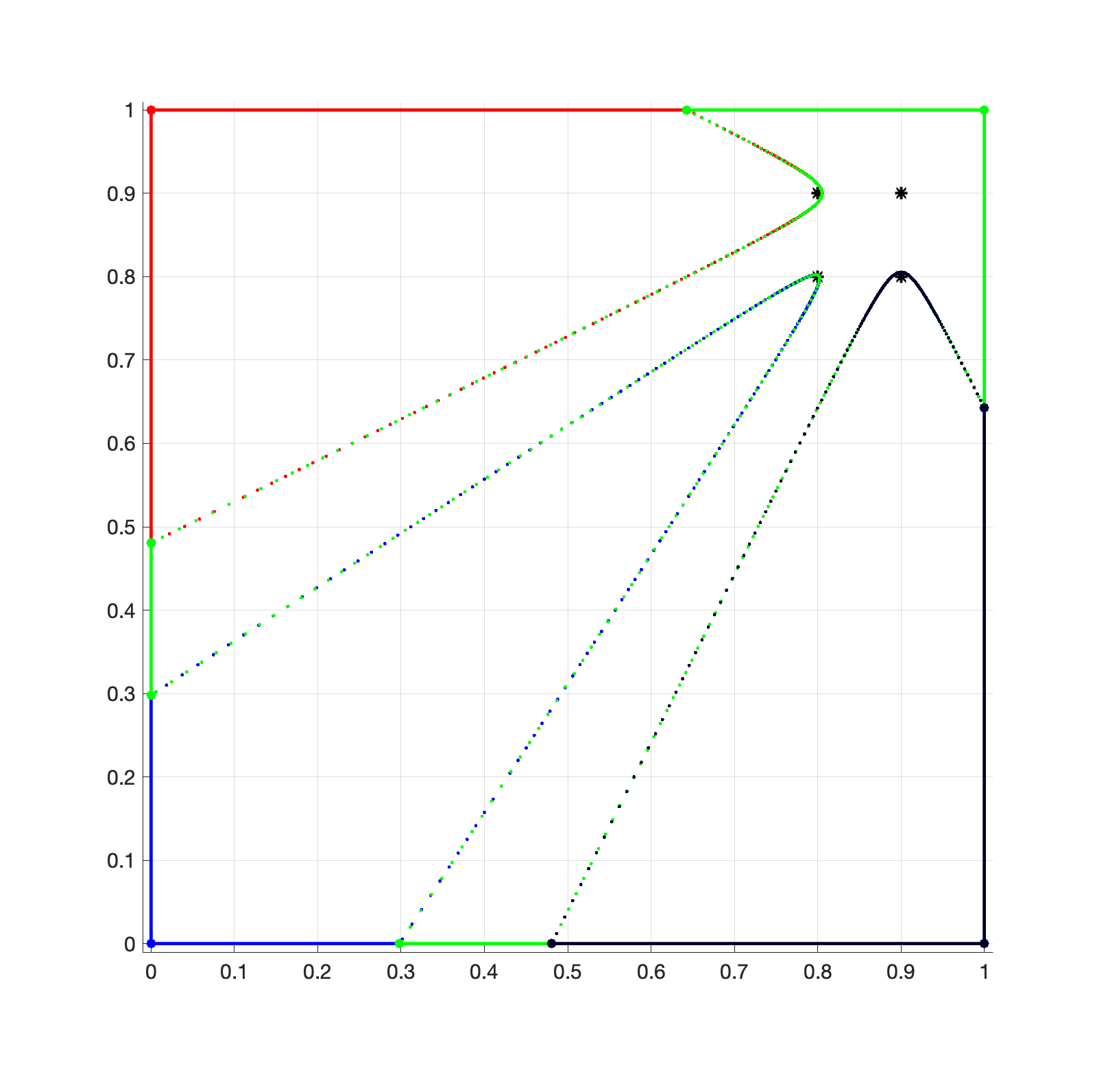}
		\caption{Uniform Source and Target Densities}\label{Fig_Inter_Uniform-Uniform}
	\end{subfigure}
	\hfill
	\begin{subfigure}[b]{0.3\linewidth}
		\centering
		\includegraphics[width=\linewidth]{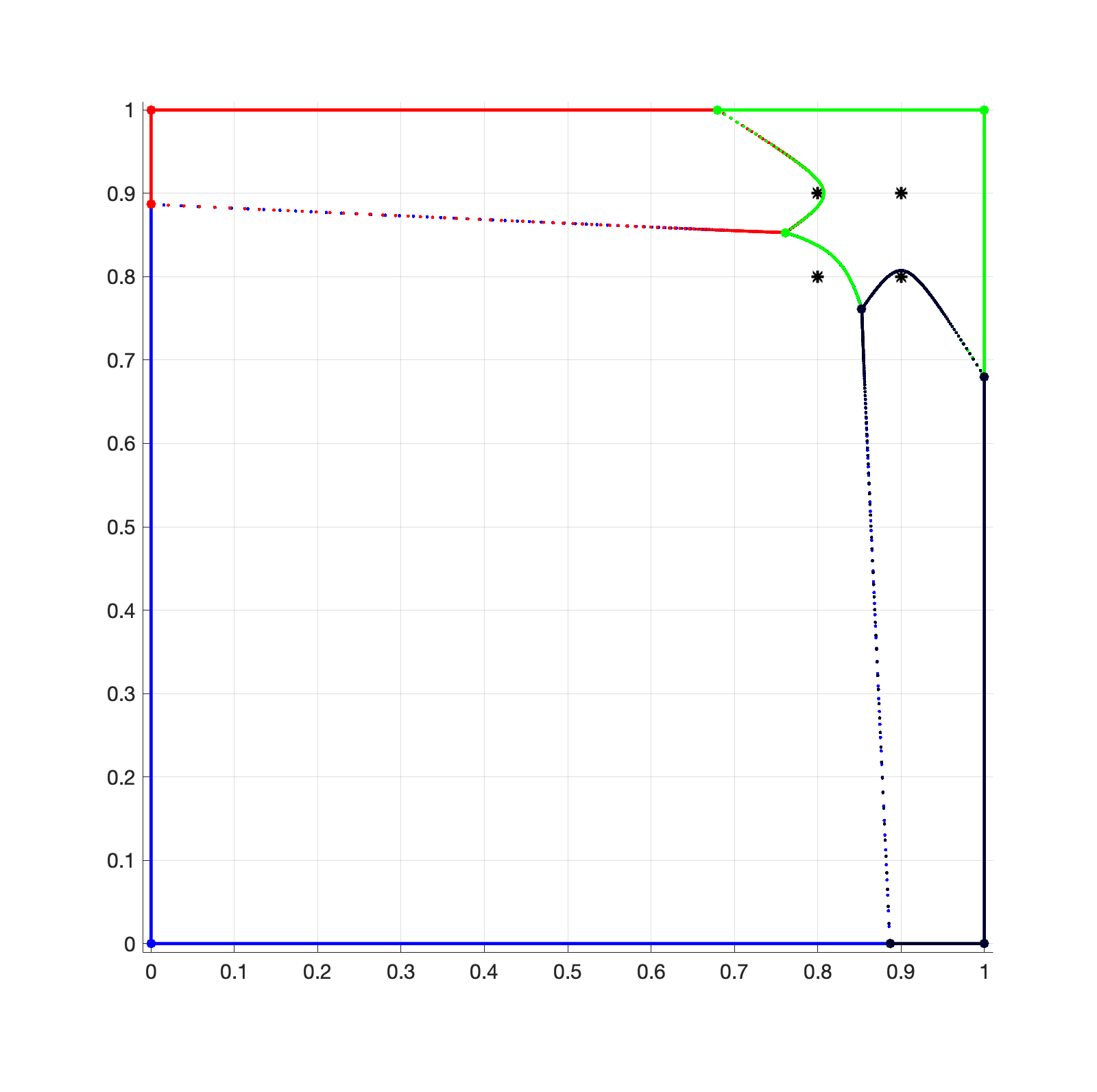}
		\caption{Uniform Source and Non-Uniform Target Densities}\label{Fig_Inter_Uniform-NonUniform}
	\end{subfigure}
	\hfill
	\begin{subfigure}[b]{0.3\linewidth}
		\centering
		\includegraphics[width=\linewidth]{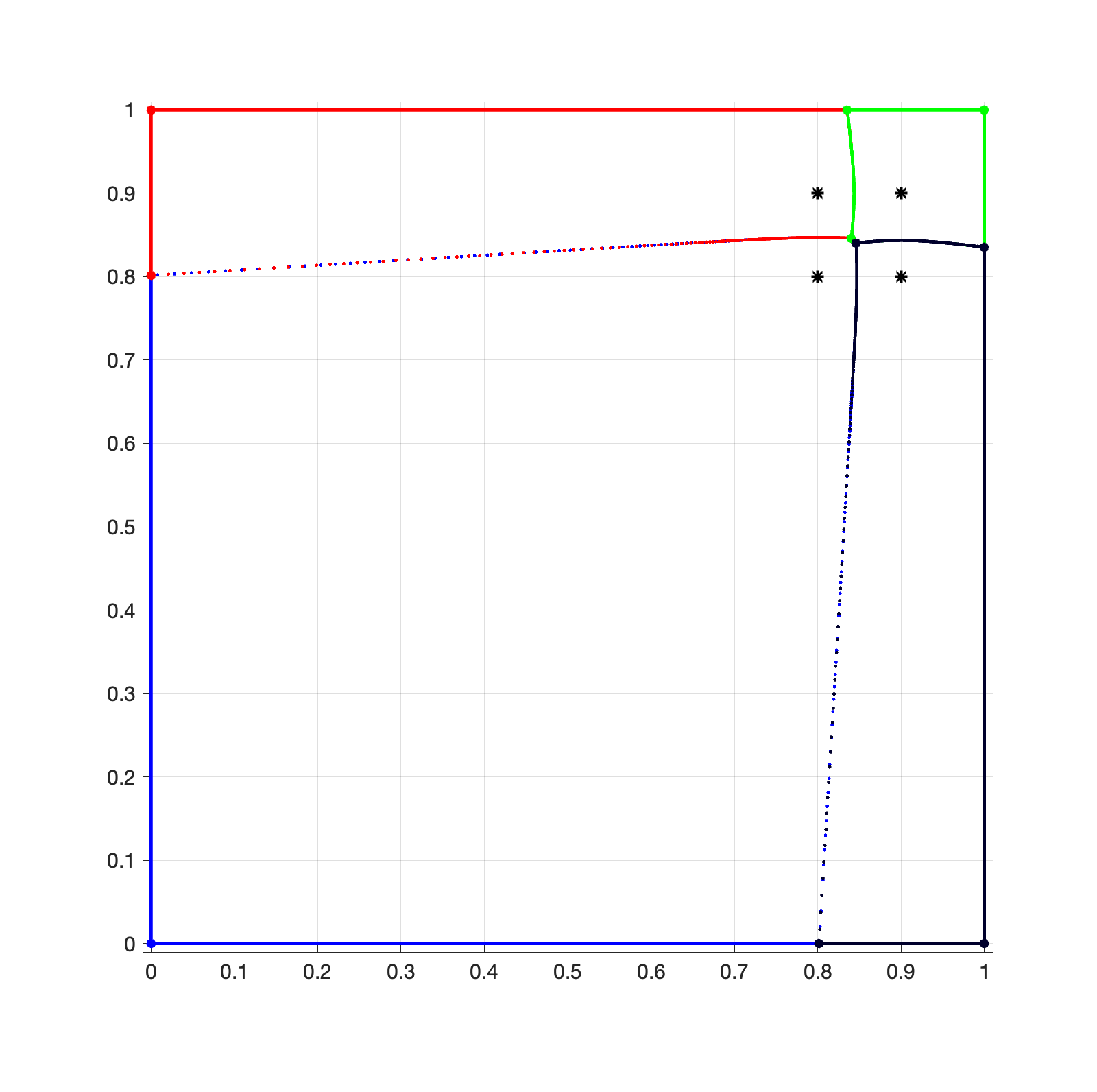}
		\caption{Non-Uniform Source and Uniform Target Densities}\label{Fig_Inter_NonUniform-Uniform}
	\end{subfigure}
	\caption{Interplay between Source and Target Densities with Location of Points}\label{Fig_Interplay}
\end{figure}
\end{exm}

\begin{exm}[Different Domains]\label{DifferentDomains}
Here we show that our algorithms can be easily adapted to work with
non-square domains, with no noticeable difference in performance.
For $\y$ we take 10 random points in domain $\Omega_i$,
obtained by retaining the first $10$ random points uniformly distributed 
in the square circumscribed to $\Omega_i$ that fall inside $\Omega_i$.  Further, we take
$\ c(\x,\y) = ||\x-\y||_2$, $\nu$ is uniform, and so is $\mu$:
$\rho(x) = \frac{1}{|\Omega_i|},\ \forall i$.
As domains we considered:
\begin{itemize}
\item[\rm{(A)}] 
$\Omega_1$ is the equilateral triangle centered at the origin with 
circumradius of $1$;
\item[\rm{(B)}] 
$\Omega_2$ is the regular pentagon centered at the origin with circumradius
equal to $1$;
\item[\rm{(C)}] 
$\Omega_3$ is an irregular pentagon of centroid at 
$(0.44239, 0.45993)$ and circumradius equal to $0.69813$;
\item[\rm{(D)}] 
$\Omega_4$ is the circle centered at the origin with radius equal to $1$;
\end{itemize}
As initial guesses we take $w^0=0$.

\rm{
\begin{table}[ht]
	\centering\scalebox{.75}{\renewcommand*{\arraystretch}{2}
		\begin{tabular}{||c|c|c|c|c|c||}\hline\hline
			\textbf{Domain} & \textbf{Error} & \textbf{Iterations} & \textbf{Damping} & \textbf{Time} & \textbf{Feasibility Coefficient} $\mathbf{\kappa}$\\\hline\hline
			Equilateral Triangle & $4.2979*10^{-11}$ & 6 & 0 & $97.79$ & $0.07999$ \\\hline
			Regular Pentagon     & $2.7436*10^{-10}$ & 6 & 0 & $130.1$ & $0.11848$ \\\hline
			Irregular Pentagon   & $5.2235*10^{-11}$ & 5 & 0 & $59.33$ & $0.25617$ \\\hline
			Circle               & $2.0969*10^{-11}$ & 4 & 0 & $22.251$ & $0.51896$ \\\hline
			\hline
	\end{tabular}}
	\caption{Different domains and uniform source density}\label{Tab_Domains_Unif}
\end{table}

}
\it{Interestingly, the most difficult problems are those with the most regular
domains, equilateral triangle and regular pentagon.  However, all cases
are solved fairly easily.  See Figure \ref{Fig_Domains_Unif}.}

\begin{figure}[ht]
	\centering
	\begin{subfigure}[b]{0.24\linewidth}
		\centering
		\includegraphics[width=\linewidth]{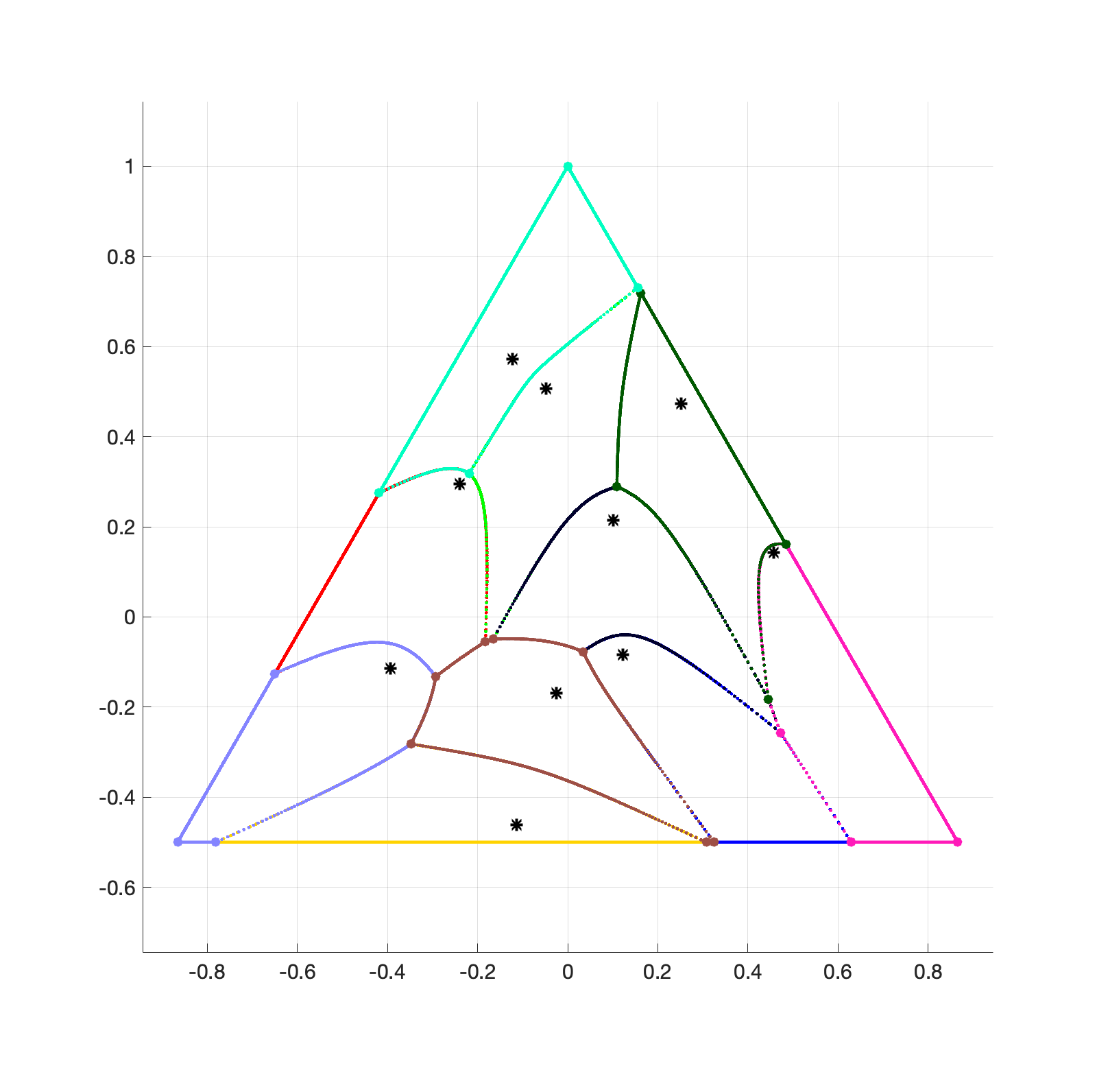}
		\caption{Equilateral Triangle}\label{Fig_Triangle_Unif}
	\end{subfigure}
	\hfill
	\begin{subfigure}[b]{0.24\linewidth}
		\centering
		\includegraphics[width=\linewidth]{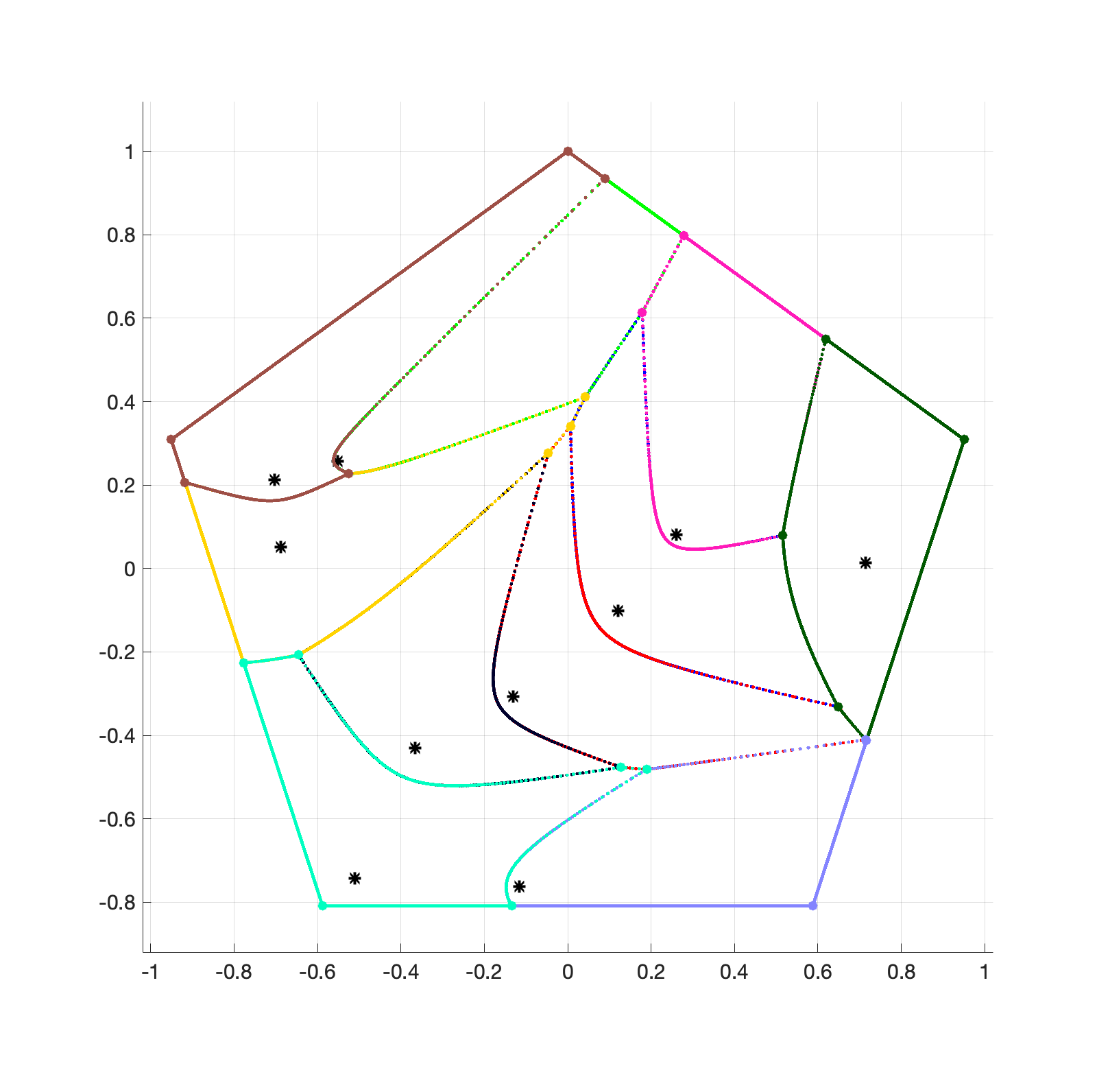}
		\caption{Regular Pentagon}\label{Fig_RegPentagon_Unif}
	\end{subfigure}
	\hfill
	\begin{subfigure}[b]{0.24\linewidth}
		\centering
		\includegraphics[width=\linewidth]{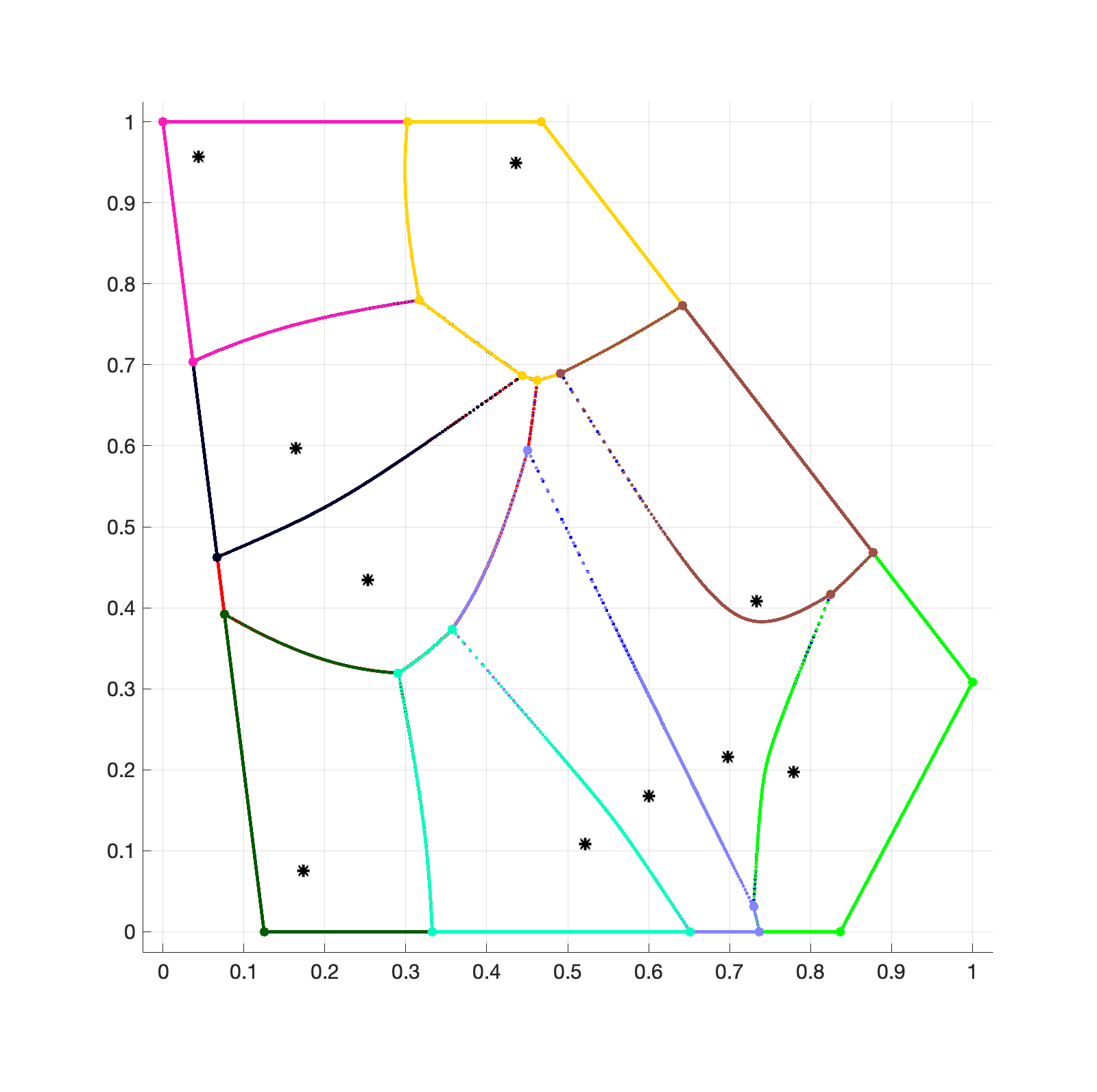}
		\caption{Irregular Pentagon}\label{Fig_IrregPentagon_Unif}
	\end{subfigure}
	\hfill
	\begin{subfigure}[b]{0.24\linewidth}
		\centering
		\includegraphics[width=\linewidth]{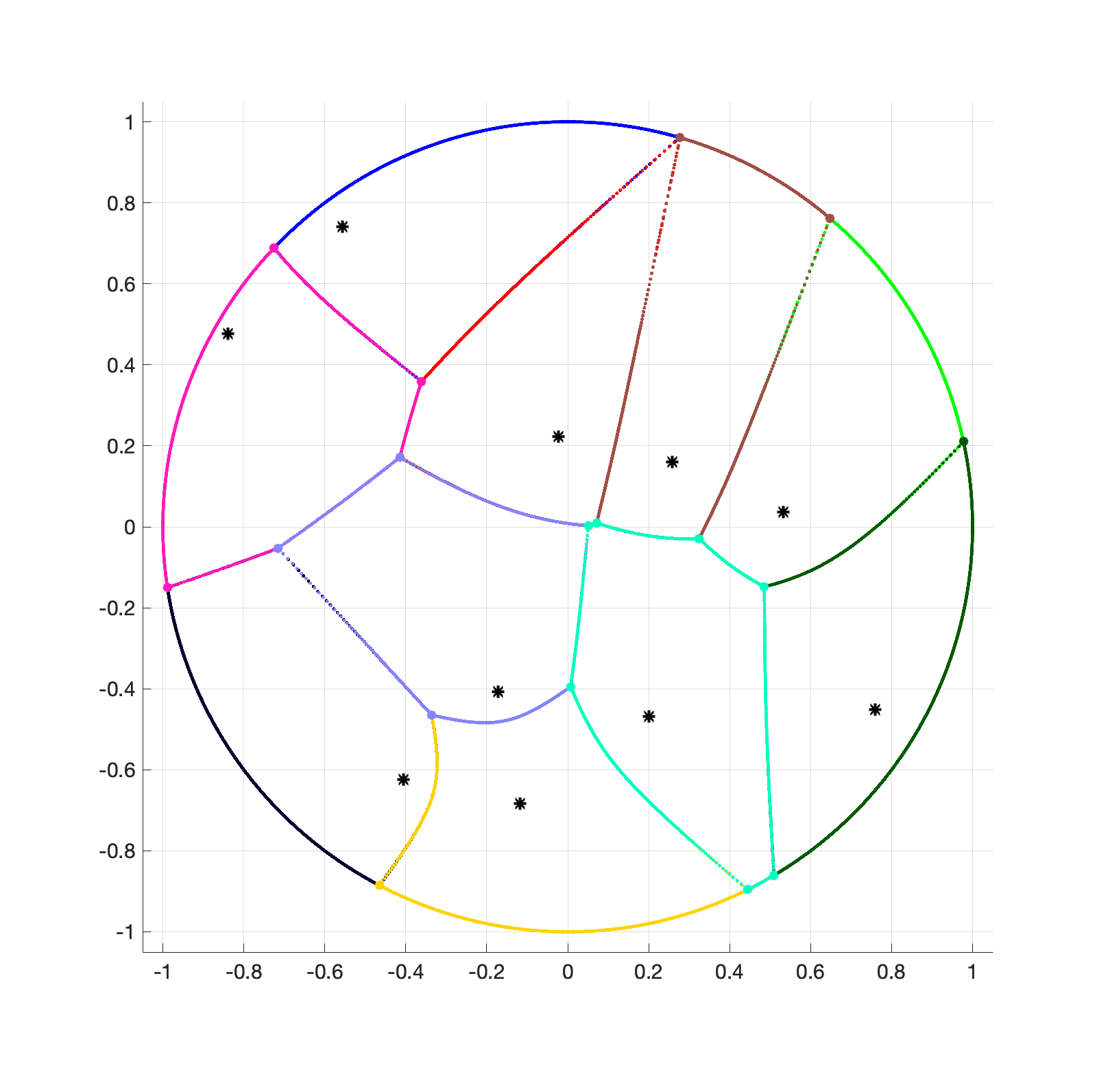}
		\caption{Circle}\label{Fig_Circle_Unif}
	\end{subfigure}
	\caption{Examples with Different domains and uniform source density}\label{Fig_Domains_Unif}
\end{figure}

\end{exm}

\section{Conclusions}\label{Closing}
In this work, we presented a novel implementation of Newton's method for
solving semi-discrete optimal transport problems for cost functions given
by a (positive combination of) $p$-norm(s), $1<p<\infty$, and several
choices of continuous and discrete densities.  To date, there appeared
to have been no implementation of Newton's method for these problems,
and we succeeded in making some progress by proving that the Laguerre
cells providing the solution of the OT problem are star shaped with respect
to the target points and by exploiting this fact in our algorithmic development.
We gave a detailed description of all algorithms we implemented and provided
quantitative testing on several examples, as well as comparison with competing approaches,
well beyond simply showing some qualitatively faithful figures.
Our algorithms proved much more robust and accurate than other methods,
being based on rigorously justified adaptivity and error control, which --in principle--
allows to solve the OT problem at any desired level of accuracy.  The price
we pay for this is the need for the continuous density to be sufficiently smooth.
Future developments call for algorithms suited for 3-d, as well as for development
of Newton-like methods better suited when we only have limited smoothness of the
continuous density.

\bibliography{reflist}{}
\bibliographystyle{plain}

\section*{Appendix}
\begin{claim}\label{p-even}
Let the cost function be a $p-$norm, where $p$ is even:
$c(\x,\y) = ||\x-\y||_p$, $p= 2k$, $k=1,2,\dots$. 
If $|w|<||\y_1 - \y_2||_p$, then $\nabla F(\x_0)\neq 0$, $\forall \x_0$ such that $F(\x_0)=w$, where
\begin{align*}
    F(\x) = ||\x-\y_1||_p - ||\x-\y_2||_p\ .
\end{align*}
\end{claim}
\begin{proof}
Assume, by contradiction, that there exists a point $x_0$ where both
$F(\x_0) = w$ and $\nabla F(\x_0) = 0$. An explicit computation gives
\begin{gather*}
    \nabla F(x) = 
    \frac{1}{||\x-\y_1||_p^{p-1}}\begin{pmatrix} (\x-\y_1)_1 ^{p-1} \\ \vdots \\
    	(\x-\y_1)_d^{p-1}\end{pmatrix} - 
    \frac{1}{||\x-\y_2||_p^{p-1}}\begin{pmatrix} (\x-\y_2)_1 ^{p-1} \\ \vdots \\
    	(\x-\y_2)_d^{p-1} \end{pmatrix} \quad \text{and so} \\
    \nabla F(\x_0) = 0\implies \frac{(\x_0-\y_1)_j^{p-1}}{||\x_0-\y_1||_p^{p-1}} 
    = \frac{(\x_0-\y_2)_j^{p-1}}{||\x_0-\y_2||_p^{p-1}}\implies 
    \frac{(\x_0-\y_1)_j}{||\x_0-\y_1||_p} = \frac{(\x_0-\y_2)_j}{||\x_0-\y_2||_p},
    \,\ j=1,\dots, d\ .
\end{gather*}
Since $F(\x_0)=0$, then $||\x_0-\y_1||_p = w + ||\x_0-\y_2||_p$ and so
\begin{gather*}
    \frac{(\x_0-\y_1)_j}{w + ||\x_0-\y_2||_p} = 
    \frac{(\x_0-\y_2)_j}{||\x_0-\y_2||_p}\implies w(\x_0-\y_2)_j = ||x_0-\y_2||_p(\y_2 - \y_1)_j\ ,\quad j=1,\dots, d\ .
\end{gather*}
Taking the $p-$norm of the vectors on both sides of the last relation gives
\begin{gather*}
    |w|\cdot||\x_0-\y_2||_p = ||x_0-\y_2||_p||\y_2 - \y_1||_p \quad \text{or}\quad 
     |w| = ||\y_2 - \y_1||_p\ ,
\end{gather*}
which is a contradiction and the proof follows.
\end{proof}

\begin{lem}\label{Lem_NNZGrad}
Let $\Omega\subset\Rn$ satisfy Assumption \ref{WhatOmega},
$\y_i\in\Omega^\mathrm{o}$ and $i=1,2,\dots,N$, be $N$
distinct target points, and let the 
shift values be feasible: $|w_{ij}|<c(\y_i,\y_j)$, $\forall i,j = 
1,2,\cdots,N$, $i\neq j$.  Consider the cost function
given by a $p$-norm with $p$ even.   Then, there 
$\exists \epsilon>0$ such that $\forall i,j=1,2,\dots, N,$ $i\neq j$
\begin{gather*}
||\nabla_x c(x,y_i) - \nabla_x c(x,y_j)||\geq \epsilon, \ \ \forall \x\in 
\LL_{ij}\cap \Omega ,
\end{gather*}
where $\|\cdot \|$ is the standard Euclidean distance $\|\cdot \|_2$.
\end{lem}

\begin{proof}
Firstly, we are going to show that exists a constant, which depends on $p$, 
call it $\epsilon_{ij}(p)>0$, such that
\begin{gather*}
||\nabla_x ||x-y_i||_p - \nabla_x ||x-y_j||_p||_p\geq \epsilon_{ij}(p), \ \ 
\forall \x\in L_{ij}\cap \Omega.
\end{gather*}
By direct computation, as in Claim \ref{p-even}, we have
\begin{equation*}\begin{split}
& ||\nabla_x ||x-y_i||_p - \nabla_x ||x-y_j||_p||_p=
\norm{ \begin{pmatrix} a_1^{p-1} \\ \vdots \\
a_d^{p-1}\end{pmatrix} - 
\begin{pmatrix} 
b_1 ^{p-1} \\ \vdots \\
b_d^{p-1} \end{pmatrix} }_p \quad \text{where} \\
& 
a_j=\frac{(\x-\y_1)_j}{||\x-\y_1||_p}\ ,\,\ 
b_j=\frac{(\x-\y_2)_j}{||\x-\y_2||_p}\ ,\,\ j=1,\dots, d\ .
\end{split}\end{equation*}
Now, abusing of notation, we will
write $a^{p-1}$ for the Hadamard product of the vector $a$ with
itself $(p-1)$-times, and thus
$a^{p-1}-b^{p-1}=(a-b)\circ (a^{p-2}+a^{p-3}\circ b+\dots + b^{p-2})$,
where ''$\circ$'' is the Hadamard product\footnote{$a\circ b=
 \begin{pmatrix} a_1b_1 \\ \vdots \\
a_db_d \end{pmatrix}$}.  Next, 
using Claim \ref{Claim_NNZGrad} we have
\begin{gather*}
	\bigg{|}\bigg{|}\frac{(\x-\y_i)^{p-1}}{||\x-\y_i||_p^{p-1}} - 
	\frac{(\x-\y_j)^{p-1}}{||\x-\y_j||_p^{p-1}}\bigg{|}\bigg{|}_p \geq  
	\Delta\bigg{|}\bigg{|}\frac{(\x-\y_i)}{||\x-\y_i||_p} - 
	\frac{(\x-\y_j)}{||\x-\y_j||_p}\bigg{|}\bigg{|}_p.
\end{gather*}
Using that $\x\in L_{ij}\cap \Omega\implies ||\x-\y_i||_p = ||\x-\y_j||_p+w_{ij}$:
\begin{gather*}
	\frac{||||\x-\y_j||_p(\x-\y_i)-||\x-\y_i||_p(\x-\y_j)||_p}{||\x-\y_i||_p||\x-\y_j||_p} = \frac{||||\x-\y_j||_p(\x-\y_i)-(||\x-\y_j||_p+w_{ij})(\x-\y_j)||_p}{||\x-\y_i||_p||\x-\y_j||_p},\\
	\implies \frac{||||\x-\y_j||_p(\y_j-\y_i)-w_{ij}(\x-\y_j)||_p}{||\x-\y_i||_p||\x-\y_j||_p}\geq \frac{\big{|}||\x-\y_j||_p||\y_j-\y_i||_p-|w_{ij}|||\x-\y_j||_p\big{|}}{||\x-\y_i||_p||\x-\y_j||_p},\\
	\implies ||\nabla_x ||x-y_i||_p - \nabla_x ||x-y_j||_p||_p \geq \Delta\frac{\big{|}||\y_j-\y_i||_p-|w_{ij}|\big{|}}{||\x-\y_i||_p}
\end{gather*}
Using that $\x\in \Omega\implies ||\x-\y_i||_p \leq 
C(\y_i,\Omega,p)\text{diam}(\Omega)\implies \frac{1}{||\x-\y_i||_p} \geq 
\frac{1}{C(\y_i,\Omega,p)\text{diam}(\Omega)}$, 
where $C(\y_i,\Omega,p)>0$ is a constant depending on $\y_i,\Omega$
and $p$, and so:
\begin{gather*}
||\nabla_x ||x-y_i||_p - \nabla_x ||x-y_j||_p||_p \geq 
\Delta\frac{\big{|}||\y_j-\y_i||_p-|w_{ij}|\big{|}}{C(\y_i,\Omega,p)\text{diam}
(\Omega)}\equiv \epsilon_{ij}(p)>0 .
\end{gather*}
Next, using equivalency of $p$-norm and $2$-norm,
there is a constant which depends on $p$, call it
$\alpha(p)$, such that
$\|\cdot \|_2 \ge \alpha(p)\|\cdot \|_p$, $\alpha(p)>0$,  
and so we get 
\begin{gather*}
||\nabla_x ||x-y_i||_p - \nabla_x ||x-y_j||_p||\geq \alpha(p) ||\nabla_x 
||x-y_i||_p - \nabla_x ||x-y_j||_p||_p = \alpha(p) \epsilon_{ij}(p) \equiv
\epsilon_{ij}\ .
\end{gather*}
Finally, it follows that $\forall\ i,j=1,2,\dots, N,$ $i\neq j$,
\begin{gather}
||\nabla_x ||x-y_i||_p - \nabla_x ||x-y_j||_p||\geq \epsilon,
\end{gather}
where $\epsilon = \min_{i,j}\epsilon_{ij}>0$.
\end{proof}

\begin{claim}\label{Claim_NNZGrad}
With the notation from Lemma \ref{Lem_NNZGrad},
there exists $\Delta >0$ such that
\begin{align*}
	||a^{p-2}+a^{p-3}\circ b+\dots + b^{p-2}||_p\geq \Delta>0
\end{align*}
where $p$ is an even number and
\begin{align*}
a,b\in \R^n,\,\ 	a_l=\frac{(\x-\y_i)_j}{||\x-\y_i||_p}\ ,\,\ 
	b_l=\frac{(\x-\y_j)_j}{||\x-\y_j||_p}\ ,\,\ l=1,\dots, d\ .
\end{align*}
\end{claim}
\begin{proof}
\begin{align*}
	||a^{p-2}+a^{p-3}\circ b+\dots + b^{p-2}||_p\geq |a_k^{p-2}+a_k^{p-3}b_k+\dots + b_k^{p-2}|,
\end{align*}
for an index $k$ such that $|a_k| = ||a||_{\infty}$. Next, using Claim \ref{Claim_NNZGrad2} we have
\begin{align*}
	|a_k| = \frac{(|\x-\y_i|)_k}{||\x-\y_i||_p} \geq 
	\frac{\delta_a}{C(\y_i,\Omega,p)\text{diam}(\Omega)}>0.
\end{align*}
Next we are going to consider cases:
\begin{itemize}
	\item $a_k<0$ and $b_k\leq 0$:
	\begin{gather*}
		|a_k^{p-2}+a_k^{p-3}b_k+\dots + b_k^{p-2}|\geq |a_k^{p-2}|\geq \frac{\delta_a^{p-2}}{[C(\y_i,\Omega,p)\text{diam}(\Omega)]^{p-2}}.
	\end{gather*}
	\item $a_k>0$ and $b_k\geq 0$:
	\begin{gather*}
		|a_k^{p-2}+a_k^{p-3}b_k+\dots + b_k^{p-2}|\geq |a_k^{p-2}|\geq \frac{\delta_a^{p-2}}{[C(\y_i,\Omega,p)\text{diam}(\Omega)]^{p-2}}.
	\end{gather*}
	\item $a_k<0$ and $b_k> 0$:
	\begin{gather*}
		|a_k^{p-2}+a_k^{p-3}b_k+\dots + b_k^{p-2}|\geq |a_k^{p-2}+a_k^{p-3}b_k+\dots + a_kb_k^{p-3}| = |a_k||a_k^{p-3}+a_k^{p-4}b_k+\dots + b_k^{p-3}|\geq\\
		\geq |a_k||a_k^{p-3}+a_k^{p-4}b_k+\dots + a_kb_k^{p-4}| = |a_k|^2|a_k^{p-4}+a_k^{p-5}b_k+\dots + b_k^{p-4}|\geq \dots \geq |a_k|^{p-2}\implies \\
		\implies |a_k^{p-2}+a_k^{p-3}b_k+\dots + b_k^{p-2}|\geq |a_k|^{p-2}\geq \frac{\delta_a^{p-2}}{[C(\y_i,\Omega,p)\text{diam}(\Omega)]^{p-2}}.
	\end{gather*}
	\item $a_k>0$ and $b_k< 0$: using Claim \ref{Claim_NNZGrad2} we have
	\begin{align*}
		|b_k| = \frac{(|\x-\y_j|)_k}{||\x-\y_j||_p} \geq \frac{\delta_b}{C(\y_j,\Omega,p)\text{diam}(\Omega)}>0.
	\end{align*}
	Then it follows that
	\begin{gather*}
		|b_k^{p-2}+b_k^{p-3}a_k+\dots + a_k^{p-2}|\geq |b_k^{p-2}+b_k^{p-3}a_k+\dots + b_ka_k^{p-3}| = |b_k||b_k^{p-3}+b_k^{p-4}a_k+\dots + a_k^{p-3}|\geq\\
		\geq |b_k||b_k^{p-3}+b_k^{p-4}a_k+\dots + b_ka_k^{p-4}| = |b_k|^2|b_k^{p-4}+b_k^{p-5}a_k+\dots + a_k^{p-4}|\geq \dots \geq |b_k|^{p-2}\implies \\
		\implies |a_k^{p-2}+a_k^{p-3}b_k+\dots + b_k^{p-2}|\geq |b_k|^{p-2}\geq \frac{\delta_b^{p-2}}{[C(\y_j,\Omega,p)\text{diam}(\Omega)]^{p-2}}.
	\end{gather*}
\end{itemize}
As a result, it follows that
\begin{align*}
	||a^{p-2}+a^{p-3}b+\dots + b^{p-2}||_p\geq \Delta
\end{align*}
where
\begin{equation*}
	\Delta=
	\begin{cases}
		\frac{\delta_b^{p-2}}{[C(\y_j,\Omega,p)\text{diam}(\Omega)]^{p-2}}, & \text{if $a_k>0$ and $b_k< 0$},\\
		\frac{\delta_a^{p-2}}{[C(\y_i,\Omega,p)\text{diam}(\Omega)]^{p-2}}, & \text{otherwise.}
	\end{cases}
\end{equation*}
Thus, the result follows.
\end{proof}

\begin{claim}\label{Claim_NNZGrad2}
	With the notation from Lemma \ref{Lem_NNZGrad} and Claim \ref{Claim_NNZGrad},
	there exists $\delta_a >0$ and $\delta_b>0$ such that
	\begin{align*}
		||a||_{\infty}\geq \frac{\delta_a}{C(\y_i,\Omega,p)\text{diam}(\Omega)}>0, \ \ ||b||_{\infty}\geq \frac{\delta_b}{C(\y_j,\Omega,p)\text{diam}(\Omega)}>0
	\end{align*}
	where $p$ is an even number and
	\begin{align*}
		a,b\in \R^n,\,\ 	a_l=\frac{(\x-\y_i)_l}{||\x-\y_i||_p}\ ,\,\ 
		b_l=\frac{(\x-\y_j)_l}{||\x-\y_j||_p}\ ,\,\ l=1,\dots, d\ .
	\end{align*}
\end{claim}
\begin{proof}
	Firstly, using assumption from Lemma \ref{Lem_NNZGrad}, let $w_{ij} = \xi - ||\y_i - \y_j||_p$, $0<\xi<||\y_i - \y_j||_p$. Then
	\begin{gather*}
		\x\in e_{ij}\cap \Omega\implies ||\x-\y_i||_p = ||\x-\y_j||_p+w_{ij} = ||\x-\y_j||_p+\xi - ||\y_i - \y_j||_p\\
		||\x-\y_i||_p \geq \xi - ||\x-\y_i||_p \implies  ||\x-\y_i||_p\geq \frac{\xi}{2}>0
	\end{gather*}
	Next, using equivalency of $p$-norm and $\infty$-norm:
	$\|\cdot \|_\infty \ge \gamma(p)\|\cdot \|_p$, $\gamma(p)>0$,  
	we get
	\begin{gather*}
		||\x-\y_i||_{\infty} \geq \gamma(p)\frac{\xi}{2}\implies
		||a||_{\infty} = \frac{||\x-\y_i||_{\infty}}{||\x-\y_i||_p} \geq \frac{\delta_a}{C(\y_i,\Omega,p)\text{diam}(\Omega)}>0
	\end{gather*}
	where $\delta_a = \gamma(p)\frac{\xi}{2}$. Similarly, it follows that
	\begin{gather*}
		\x\in e_{ij}\cap \Omega\implies ||\x-\y_j||_p = ||\x-\y_i||_p-w_{ij} = ||\x-\y_i||_p-\xi + ||\y_i - \y_j||_p\geq  ||\y_i - \y_j||_p - \xi >0\\
		||\x-\y_j||_{\infty} \geq \gamma(p)(||\y_i - \y_j||_p - \xi)\implies
		||b||_{\infty} = \frac{||\x-\y_j||_{\infty}}{||\x-\y_j||_p} \geq \frac{\delta_b}{C(\y_j,\Omega,p)\text{diam}(\Omega)}>0
	\end{gather*}
	where $\delta_b = \gamma(p)(||\y_i - \y_j||_p - \xi)$.
\end{proof}

\bigskip

\section*{Algorithms Pseudocodes}

\begin{algorithm}\small{
		\caption{\textit{Shooting} -- Shooting Technique}\label{Alg_Shooting}
		\begin{algorithmic}
			\Require $\theta_0$, \{$Y$, $W$\}, \{$\y_j$, $w_j$\}, $\forall j$, 
			$\mathtt{TOL}$\Comment{Default value $\mathtt{TOL}= 10^{-14}$}
			\Ensure $r$ value up to a tolerance $\mathtt{TOL}$ and index $k$
			\State $r_0 \gets 0$, $d \gets \frac{\min_j ||Y-\y_j||_p}{10},\ \forall j\neq i$
			\While{$d \geq \mathtt{TOL}$}
			\State $r_0\gets r_0 + d$
			\State $\x \gets Y + r_0\tiny{\begin{pmatrix}\cos{(\theta_0)}\\\sin{(\theta_0)}\end{pmatrix}}$
			\If {$\x\in A(\y_j)$}
			\State $r^* \gets BisectR(\theta_0, \{r_0-d, r_0\}, \{Y,W\}, \{\y_j,w_j\}, \mathtt{TOL})$ and $k \gets j$
			\State $\x \gets Y + (r^*-\mathtt{TOL})\tiny{\begin{pmatrix}\cos{(\theta_0)}\\\sin{(\theta_0)}\end{pmatrix}}$
			\If{$\x\in A(Y)$}
			\State \Return $r^*$ and $k \gets j$
			\Else
			\State $d\gets \frac{d}{2}$
			\EndIf
			\ElsIf{$\x\not\in \Omega$}
			\State $\{r^*, k\} \gets Bound(\theta_0, Y)$
			\State $\x \gets Y + (r^*-\mathtt{TOL})\tiny{\begin{pmatrix}\cos{(\theta_0)}\\\sin{(\theta_0)}\end{pmatrix}}$
			\If{$\x\in A(Y)$}
			\State \Return $r^*$ and $k$
			\Else
			\State $d\gets \frac{d}{2}$
			\EndIf
			\EndIf
			\EndWhile
		\end{algorithmic}
	}
\end{algorithm}

\begin{algorithm}\small{
		\caption{\textit{BisectR} -- Bisection Method on $r$ Value}\label{Alg_BisectR}
		\begin{algorithmic}
			\Require $\theta_0$, $\{r_a, r_b\}$, \{$Y$, $W$\}, \{$\y_j$, $w_j$\}, $\mathtt{TOL}$\Comment{$\mathtt{TOL}$ is passed by \textit{Shooting}}
			\Ensure $r$ value up to a tolerance $\mathtt{TOL}$
			\While{$r_b - r_a \geq \mathtt{TOL}$}
			\State $r \gets \frac{1}{2}(r_b + r_a)$
			\If {$F(r,\theta_0)<0$}
			\State $r_a \gets r$
			\ElsIf{$F_{b}(r)>0$}
			\State $r_b \gets r$
			\Else \State \Return r
			\EndIf
			\EndWhile
		\end{algorithmic}
	}
\end{algorithm}

\begin{algorithm}\small{
		\caption{\textit{Bound} -- Boundary Point Computation}\label{Alg_Bound}
		\begin{algorithmic}
			\Require $\theta$, $Y$
			\Ensure $r$ value and index $k$
			\State $r_1 \gets \frac{a - Y_1}{\cos(\theta)}$, $y_1 = Y_2 + r_1\sin(\theta)$
			\State $r_2 \gets \frac{b - Y_1}{\cos(\theta)}$, $y_2 = Y_2 + r_2\sin(\theta)$
			\State $r_3 \gets \frac{a - Y_2}{\sin(\theta)}$, $x_3 = Y_1 + r_3\cos(\theta)$
			\State $r_4 \gets \frac{b - Y_2}{\sin(\theta)}$, $x_4 = Y_1 + r_4\cos(\theta)$
			\If{$r_1>0$ \& $y_1\geq a$ \& $y_1\leq b$}
			\State $r\gets r_1$ and $k \gets -4$
			\ElsIf{$r_2>0$ \& $y_2\geq a$ \& $y_2\leq b$}
			\State $r\gets r_2$ and $k \gets -2$
			\ElsIf{$r_3>0$ \& $x_3\geq a$ \& $x_3\leq b$}
			\State $r\gets r_3$ and $k \gets -1$
			\ElsIf{$r_4>0$ \& $x_4\geq a$ \& $x_4\leq b$}
			\State $r\gets r_4$ and $k \gets -3$
			\EndIf
		\end{algorithmic}
	}
\end{algorithm}

\begin{algorithm}\small{
		\caption{\textit{PredCorr} -- Predictor-Corrector Step}\label{Alg_PredCorr}
		\begin{algorithmic}
			\Require $r_0$ $\theta_0$, $\Delta$, \{$Y$, $W$\}, \{$\y_j$, $w_j$\},
			$\mathtt{TOL}$\Comment{Default value $\mathtt{TOL}= 10^{-14} $}
			\Ensure $r_1 = r(\theta_0 + \Delta)$ value up to a tolerance $\mathtt{TOL}$  and increment $\Delta_{new}$
			\State $r^0_1 \gets r_0 - \Delta \frac{F_{\theta}(r_0, \theta_0)}{F_r(r_0, \theta_0)}$, $r_1\gets r_1^0$ \Comment{Tangent Predictor Step}
			\While{$F(r_1, \theta+\Delta)\geq \mathtt{TOL}$}
			\State $r_1 \gets r_1 - \frac{F(r_1, \theta+\Delta)}{F_r(r_1, \theta+\Delta)}$\Comment{Newton Corrector}
			\EndWhile
			\State $\Delta_{new} \gets \frac{10^{-3}}{2\sqrt{r_1 - r_1^0}}\Delta$
		\end{algorithmic}
	}
\end{algorithm}

\begin{algorithm}\small{
		\caption{\textit{BisectTheta} -- Bisection Method on $\theta$ Value}\label{Alg_BisectTheta}
		\begin{algorithmic}
			\Require $r_0$, $\theta_0$, $\theta_1$, \{$Y$, $W$\}, \{$\y_j$, $w_j$\}, $\mathtt{TOL}$\Comment{Default value $\mathtt{TOL}= 10^{-14} $}
			\Ensure $r$ and $\theta$ values up to a tolerance $\mathtt{TOL}$
			\While{$\theta_1 - \theta_0 \geq \mathtt{TOL}$}
			\State $\theta \gets \frac{1}{2}(\theta_0 + \theta_1)$
			\State $r \gets PredCorr(r_0, \theta_0, \theta-\theta_0,\{Y, W\}, \{\y_j, w_j\}, \mathtt{TOL})$
			\State $\x \gets Y + r\tiny{\begin{pmatrix}\cos{(\theta)}\\\sin{(\theta)}\end{pmatrix}}$
			\If {$\x\in A(Y)$}
			\State $\theta_0 \gets \theta$ and $r_0 \gets r$
			\Else
			\State $\theta_1 \gets \theta$
			\EndIf
			\EndWhile
		\end{algorithmic}
	}
\end{algorithm}

\begin{algorithm}\small{
		\caption{\textit{CurveBound} -- Breakpoint Between Curve and Domain Boundary}\label{Alg_CurveBound}
		\begin{algorithmic}
			\Require $k$, $\theta_0$, $\theta_1$, \{$Y$, $W$\}, \{$\y_j$, $w_j$\}, $\mathtt{TOL}$\Comment{$k$ is the index of domain boundary}
			\Ensure $r^*$ and $\theta^*$ values up to a tolerance $\mathtt{TOL}$
			\Comment{Default value $\mathtt{TOL}= 10^{-14}$}
			\If{$k = -4$}
			\State $r(\theta) \gets \frac{a - Y_1}{\cos(\theta)}$
			\ElsIf{$k = -2$}
			\State $r(\theta) \gets \frac{b - Y_1}{\cos(\theta)}$
			\ElsIf{$k = -1$}
			\State $r(\theta) \gets \frac{a - Y_2}{\sin(\theta)}$
			\ElsIf{$k = -3$}
			\State $r(\theta) \gets \frac{b - Y_2}{\sin(\theta)}$
			\EndIf
			\State $\theta^* \gets \mod(\textbf{fzero}(F(r(\theta),\theta), [\theta_0,\ \theta_1], \mathtt{TOL}),\ 2\pi)$
			\State $r^*\gets r(\theta^*)$
		\end{algorithmic}
	}
\end{algorithm}

\begin{algorithm}\small{
		\caption{\textit{BoundBound} -- Breakpoint Between Two Domain Boundary Branches}\label{Alg_BoundBound}
		\begin{algorithmic}
			\Require $k_1$, $k_2$
			\Ensure $r^*$ and $\theta^*$ values
			\If{($k_1 = -4$ \& $k_2 = -1$) | ($k_1 = -1$ \& $k_2 = -4$)}
			\State $\x \gets \tiny{\begin{pmatrix}a\\a\end{pmatrix}}$
			\ElsIf{($k_1 = -1$ \& $k_2 = -2$) | ($k_1 = -2$ \& $k_2 = -1$)}
			\State $\x \gets \tiny{\begin{pmatrix}b\\a\end{pmatrix}}$
			\ElsIf{($k_1 = -2$ \& $k_2 = -3$) | ($k_1 = -3$ \& $k_2 = -2$)}
			\State $\x \gets \tiny{\begin{pmatrix}b\\b\end{pmatrix}}$
			\ElsIf{($k_1 = -3$ \& $k_2 = -4$) | ($k_1 = -4$ \& $k_2 = -3$)}
			\State $\x \gets \tiny{\begin{pmatrix}a\\b\end{pmatrix}}$
			\EndIf
			\State $\theta^* \gets \mod(\arctan(\frac{x_2 - Y_2}{x_1-Y_1}),\ 2\pi)$
			\State $r^*\gets \frac{x_1 - Y_1}{\cos(\theta^*)}$
		\end{algorithmic}
	}
\end{algorithm}

\begin{algorithm}\small{
		\caption{\textit{AdaptSimpson} -- Adaptive Integral Computation}\label{Alg_AdaptSimpson}
		\begin{algorithmic}
			\Require $f(x)$, $[a, b]$, $\mathtt{TOL}$
			\Comment{Default value $\mathtt{TOL}= 10^{-12} $}
			\Ensure $I = \int_a^b f(x)dx$ up to a tolerance $\mathtt{TOL}$
			\State $check \gets 0$;  $N_{max} \gets 1024$; $i_{max}\gets \lfloor \log_2(\mathtt{TOL}/\mathtt{eps})\rfloor$\Comment{$\mathtt{eps}$ is machine
				precision}
			\State $I \gets 0$; $Err \gets 0$
			\For{$i \gets 0$ to $i_{max}$}
			\State $x_i \gets [a:(b-a)/(2^i):b]$
			\State $\mathtt{TOL}_i \gets \mathtt{TOL}/(2^i)$
			\For{$j \gets 1$ to $2^i$}
			\If{$check(j)= 0$}
			\State $[I_j, err_j]\gets CompSimpson(f(x), x_i(j),x_i(j+1), 
			\mathtt{TOL}_i, N_{max})$
			\If{$Err_j < \mathtt{TOL}_i$}
			\State $I \gets I + I_j$; $Err \gets Err + err_j$; $check(j)\gets 1$
			\EndIf
			\EndIf
			\EndFor
			\If{$\min(check) = 1$}
			\State \Return $I$
			\Else
			\State $check \gets \mathbf{repelem}(check, 2)$ 
			\Comment {Repeat each of $check$ vector's elements two times}
			\EndIf
			\If {$\mod(i,3)=2$}
			\State $N_{max}\gets 8\times N_{max}$
			\EndIf
			\EndFor
	\end{algorithmic}}
\end{algorithm}

\begin{algorithm}\small{
		\caption{\textit{CompSimpson} -- Integral Computation using Composite Simpson's 1/3 Rule}\label{Alg_CompSimpson}
		\begin{algorithmic}
			\Require $f(x)$, $[a, b]$, $\mathtt{TOL}$, $N_{max}$
			\Comment{$\mathtt{TOL}$ is passed by {\it{AdaptSimpson}}}
			\Ensure $I = \int_a^b f(x)dx$ up to a tolerance $\mathtt{TOL}$ and relative error value $err$
			\State $i_{max}\gets \lfloor \log_2(N_{max})\rfloor$; $n \gets 2$
			\State $I_n \gets Simpson(f(x), [a, b], n)$ \Comment{Computes integral using Simpson's 1/3 Rule with $n$ sub-intervals} 
			\For{$i \gets 1$ to $i_{max}$}
			\State $I_{2n} \gets Simpson(f(x), [a, b], 2n)$
			\State $err \gets 16|I_{2n}-I_n|/15$; $N \gets n(err/\mathtt{TOL})^{1/4}$; $N\gets 2\lfloor N/2\rfloor$
			\If{$err< \mathtt{TOL}$  $\lor$ $N>N_{max}$}
			\State \Return $I_{2n}$ and $err$
			\Else
			\State $n\gets 2n$; $I_n \gets I_{2n}$
			\EndIf
			\EndFor
	\end{algorithmic}}
\end{algorithm}

\begin{algorithm}\small{
		\caption{\textit{Simpson} -- Integral Computation using Simpson's 1/3 Rule}\label{Alg_Simpson}
		\begin{algorithmic}
			\Require $f(x)$, $[a, b]$, $N$
			\Ensure $I = \int_a^b f(x)dx$ approximation using Simpson's 1/3 Rule with $N$ sub-intervals
			\State $h\gets \frac{b-a}{N}$, $I \gets \frac{1}{3}hf(a)$
			\For{$i \gets 2$ to $N$}
			\If{$\mod(i,2) = 0$}
			\State $I \gets I + \frac{4}{3}hf(a+(i-1)h)$
			\ElsIf{$\mod(i,2)= 1$}
			\State $I \gets I + \frac{2}{3}hf(a+(i-1)h)$
			\EndIf
			\EndFor
			\State $I \gets I + \frac{1}{3}hf(b)$
	\end{algorithmic}}
\end{algorithm}

\begin{algorithm}\small{
		\caption{\textit{GridInitialGuess} -- Initial Guess Computation on a Grid}\label{Alg_GridInitialGuess}
		\begin{algorithmic}
			\Require $h$, $(a, b)$, \{$\y_i$, $\nu_i$\}, $\forall i$, $\mathtt{MAXIT}$
			\Comment{Default is $\mathtt{MAXIT=50}$}
			\Ensure $w^0$ value
			\State $w^0 \gets 0$, $N_h \gets \frac{b-a}{h}$, \Comment{$N_h$ is the number of squares on the grid}
			\State $X_1 \gets [a+\frac{h}{2}:h:b-\frac{h}{2}]$, $X_2 \gets [a+\frac{h}{2}:h:b-\frac{h}{2}]$, \Comment{Coordinates of centers of squares}
			\For{$iter \gets 1$ to $\mathtt{MAXIT}$}
			\State $M \gets GridMeasure$($h$, $N_h$, $\{X_1, X_2\}$, \{$\y_i$, $w^0_i$\}, $\forall i$)
			\For{$k \gets 1$ to $N$} \Comment{$N$ is the number of target points}
			\State $M^0 \gets GridMeasure$($h$, $N_h$, $\{X_1, X_2\}$, \{$\y_i$, $w^0_i$\}, $\forall i$)
			\State $incr \gets \sqrt{h}$
			\While{$|M^0(k)-\nu_k|>h^2$ $\land$ $incr > h^4$}
			\State $w^0(i)\gets w^0(i) + sign(M^0(k)-\nu_k)\frac{incr}{N-1}$, $\forall i\neq k$
			\State $M^1 \gets GridMeasure$($h$, $N_h$, $\{X_1, X_2\}$, \{$\y_i$, $w^0_i$\}, $\forall i$)
			\If{$|M^1(k)-\nu_k|>|M^0(k)-\nu_k|$}
			\State $incr \gets \sqrt{h}\times incr$
			\EndIf
			\State $M^0 \gets M^1$
			\EndWhile
			\EndFor
			\If{$\max_i|M^0(i)-\nu_i|<h^2$ $\lor$ $\max_i|M^0(i)-\nu_i|\geq \max_i|M(i)-\nu_i|$}
			\State \Return $w^0$
			\EndIf
			\EndFor
		\end{algorithmic}
	}
\end{algorithm}

\begin{algorithm}\small{
		\caption{\textit{GridMeasure} -- Measure Approximation on the Grid}\label{Alg_GridMeasure}
		\begin{algorithmic}
			\Require $h$, $N_h$, $\{X_1, X_2\}$, \{$\y_i$, $w_i$\}, $\forall i$
			\Ensure $M$ values, which approximate $\mu(A(\y_i))$ values on the grid
			\State $M \gets 0$
			\For{$l \gets 1$ to $N_h$}
			\For{$m \gets 1$ to $N_h$}
			\State $x \gets \tiny{\begin{pmatrix}X_1(l)\\X_2(m)\end{pmatrix}}$
			\State $ind \gets \arg\min_i c(x,y_i) - w_i$
			\State $M(ind) \gets M(ind) + \rho(x)h^2$
			\EndFor    
			\EndFor
		\end{algorithmic}
	}
\end{algorithm}

\begin{algorithm}\small{
		\caption{\textit{DampedNewton} -- Damped Newton's Method}\label{Alg_DampedNewton}
		\begin{algorithmic}
			\Require $w^0$, $\mu$, \{$\y_i$, $\nu_i$\}, $\forall i$, $\mathtt{MAXIT}$, $\mathtt{TOL}$
			\Comment{Default $\mathtt{MAXIT=20}$, ${\mathtt{TOL}}=10^{-8}$}			
			\Ensure $w^k$ value such that $\max_{i}|\mu(A(\y_i,w^k))-\nu_i|\leq TOL$
			\State $err^0 \gets \mu(A(\y_i,w^0)) - \nu_i$,  $\forall i$
			\For{$k \gets 1$ to $\mathtt{MAXIT}$}
			\State $H^0 \gets \frac{\partial^2\Phi}{\partial w_i^0\partial w_j^0}$, $\forall i,j$
			\State $s \gets - (H^0)^{I}err^0$, $w^1 \gets w^0 + s$
			\While{$\min_{i,j}\left\|1-|w^1_i - w^1_j|/c(\y_i,\y_j)\right\|\leq 0$}
			\Comment{See \eqref{NeedFeas}}
			\State $s \gets \frac{1}{2}s$, $w^1 \gets w^0 + s$
			\EndWhile
			\State $err^1 \gets \mu(A(\y_i,w^1)) - \nu_i$,  $\forall i$
			\While{$||err^1||_2^2 > ||err^0||_2^2$ $\land$ $||s||_{\infty}\geq \frac{\mathtt{TOL}}{2^{\mathtt{MAXIT}}}$}
			\Comment{Damped Newton}
			\State $s \gets \frac{1}{2}s$, $w^1 \gets w^0 + s$
			\State $err^1 \gets \mu(A(\y_i,w^1)) - \nu_i$,  $\forall i$
			\EndWhile
			\If{$||err^1||_{\infty}\leq \mathtt{TOL}$ $\lor$ $||s||_{\infty}\leq \frac{\mathtt{TOL}}{2^{\mathtt{MAXIT}}}$}
			\State \Return $w^1$
			\Else
			\State $w^0 \gets w^1$, $err^0 \gets err^1$
			\EndIf
			\EndFor
	\end{algorithmic}}
\end{algorithm}

\end{document}